\documentclass[reqno, 11pt]{amsart}
\usepackage{amssymb}
\usepackage{amsmath}
\usepackage{mathtools}
\usepackage{ifpdf}
\usepackage{vmargin}
\usepackage{mathrsfs}  
\usepackage{epstopdf}
\usepackage{setspace}

\numberwithin{equation}{section}

\newcommand{\Z}{{\mathbb{Z}}}

\newcommand{\R}{\mathbb{R}}

\theoremstyle{plain}
\newtheorem{theorem}{Theorem}
\newtheorem{proposition}[theorem]{Proposition}
\newtheorem{lemma}[theorem]{Lemma}

\theoremstyle{definition}

\newtheorem{remark}[theorem]{Remark}

\numberwithin{equation}{section}
\numberwithin{theorem}{section}

\numberwithin{equation}{section}

\renewcommand{\Re}{\mathrm{Re}}
\renewcommand{\Im}{\mathrm{Im}}

\newcommand{\bee}{\begin{eqnarray*}}
	\newcommand{\eee}{\end{eqnarray*}}

\numberwithin{equation}{section}
\numberwithin{theorem}{section}

\begin{document}
	\onehalfspacing
	
	\title[Global dynamics above the ground state for the Hartree equation]
	{Global dynamics above the ground state for the energy-critical Hartree equation with radial data}
	
	\author[]{Xuemei Li}
	\address{\hskip-1.15em Xuemei Li
		\hfill\newline Laboratory of Mathematics and Complex Systems,
		\hfill\newline Ministry of Education,
		\hfill\newline School of Mathematical Sciences,
		\hfill\newline Beijing Normal University,
		\hfill\newline Beijing, 100875, People's Republic of China.}
	\email{xuemei\_li@mail.bnu.edu.cn}
	
	\author[]{Chenxi Liu}
	\address{\hskip-1.15em Chenxi Liu
		\hfill\newline Laboratory of Mathematics and Complex Systems,
		\hfill\newline Ministry of Education,
		\hfill\newline School of Mathematical Sciences,
		\hfill\newline Beijing Normal University,
		\hfill\newline Beijing, 100875, People's Republic of China.}
	\email{liuchenxi@mail.bnu.edu.cn}
	
	
	\author[]{Guixiang Xu}
	\address{\hskip-1.15em Guixiang Xu
		\hfill\newline Laboratory of Mathematics and Complex Systems,
		\hfill\newline Ministry of Education,
		\hfill\newline School of Mathematical Sciences,
		\hfill\newline Beijing Normal University,
		\hfill\newline Beijing, 100875, People's Republic of China.}
	\email{guixiang@bnu.edu.cn}

	

	\begin{abstract} Based on  the concentration-compactness-rigidity argument in \cite{KenM:NLS,KenM:NLW} and the non-degeneracy of the ground state in \cite{LLTX:Nondeg,LLTX:g-Hart,LTX:Nondeg},  long time dynamics for the focusing energy-critical Hartree equation with radial data have been classified when the energy $E(u_0)\leq E(W)$ in \cite{LiMZ:crit Hart,LLTX:g-Hart,MWX:Hart,MXZ:crit Hart:f rad}, where $W$ is the ground state. In this paper, we continue the study on the dynamics of the radial solutions with the energy $E(u_0)$ at most slightly larger than that of the ground states. This is an extension of the results \cite{KriNS:NLW rad, KriNS:NLW non,NakR,NakS:NLKG,NakS:book,NakS:NLS,NakS:NLKG:non,Roy} on NLS, NLW and NLKG, which were pioneered by K. Nakanishi and W. Schlag in \cite{NakS:NLKG, NakS:book} in the study of nonlinear Klein-Gordon equation in the subcritical case. The argument is an adaptation of the works in  \cite{KriNS:NLW rad, KriNS:NLW non,NakR,Roy}, the proof uses an analysis of the hyperbolic dynamics near the ground state and the variational structure far from them.  The key components that allow to classify the solutions are the hyperbolic (ejection) dynamical behavior near the ground state and the one-pass lemma.
	\end{abstract}
	
	\maketitle
	
	
	
	\section{Introduction}
	
	In this paper, we consider the focusing, energy-critical Hartree equation
	\begin{equation}\label{equ:Hart}
	i \partial_t u - \Delta u  - \left( |x|^{-4}*|u|^2\right)u =0,\quad
	(t,x)\in \R \times \R^d,
	\end{equation}
	with initial data $u_0 \in \dot H^1(\R^d)$, $d\geq 5$. Here $\dot H^1$ is the  energy space, i.e., the completion of the Schwartz space with respect to the norm $\Vert f \Vert_{\dot H^1}:=\Vert \nabla f \Vert_{L^2}$.
	
	
	
	The Hartree equation arises in the study of Boson stars and other physical phenomena, please refer to \cite{FrohLen:Hart,Pit}. It can also be considered as a classical limit of a field equation and it describes a quantum mechanical non-relativistic many-boson system interacting through a two body potential $V(x)=|x|^{-4}$ , see \cite{GinV:NLS}. As for the mean-field limit of many-body quantum systems, we can refer to the work by Hepp in  \cite{Hepp}, see also \cite{BarEGMY, BarGM,  FrohGS, FrohLen:Hart, GinV:NLS,  Spoh}. Lieb and Yau use it to develop the theory for stellar collapse in \cite{LiebY}.
	
	The equation\eqref{equ:Hart} enjoys several symmetries: If $u(t,x)$ is a solution, then:
	\begin{enumerate}
		\item[(a)]  by scaling invariance, so is $e^{\frac{d-2}{2}\sigma }u(e^{2\sigma } t,
		e^{\sigma}x)$,  for any $\sigma \in \R $ ;
		\item[(b)] by time translation invariance, so is  $u(t+t_0, x)$ for any $t_0\in \R$ ;
		\item[(c)] by spatial translation invariance, so is  $u(t, x+x_0)$ for any $x_0\in \R^d$ ;
		\item[(d)] by phase rotation invariance, so is $e^{i\theta_0}u(t,x)$, for any $\theta_0\in \R$ ;
		\item[(e)] by time reversal invariance,  so is $\overline{u(-t, x)}$.
	\end{enumerate}
	
	By the Strichartz estimate, the local well-posedness  of equation \eqref{equ:Hart} in the energy space was developed in \cite{Caz:book,MXZ:Cauchy}. Namely, if $u_0 \in \dot H^1(\R^d)$, there exists a unique solution defined on a maximal interval $I=\left(-T_-(u), T_+(u)\right)$ and the energy
	\begin{align}\label{quan:energy}
 E(u(t)):=&\frac12 \int_{\R^d} \big| \nabla u(t,x)\big|^2 dx-\frac{1}{4} \iint_{\R^d\times \R^d} \frac{\big|u(t,x)\big|^2\big|u(t,y)\big|^2}{|x-y|^4} dxdy\\
	=&E(u_0) \nonumber
	\end{align}
	is conserved on $I$.
The equation \eqref{equ:Hart} is called energy critical because of the fact that the scaling 
	\begin{equation*}
	u(t,x)\rightarrow u_\sigma (t,x) =e^{\frac{d-2}{2}\sigma }u(e^{2\sigma } t,
	e^{\sigma}x),\;  \sigma \in \R
	\end{equation*}
	makes \eqref{equ:Hart} and the energy \eqref{quan:energy} invariant. The equation \eqref{equ:Hart} can be rewritten in the Hamiltonian form 
	\begin{align*}
	\partial_t u = iE'(u)
	\end{align*}
	where $\langle E'(u), h\rangle=\partial_\lambda E (u+\lambda \, h) |_{\lambda =0}$ and $\langle\cdot,\cdot\rangle$ denotes the real-valued inner product on $L^2(\R^d)$:
	\begin{equation*}
	\left(f|g\right):=\int_{\R^d} f(x)\bar{g}(x)dx,\;\; \langle f,g\rangle:=\Re\left(f|g\right) .
	\end{equation*}
	The symplectic form $\omega$ associated to the Hamiltonian system is
	$
	\omega(u,v):=\langle iu,v \rangle . 
	$
	
	It turns out that the explicit solution
	\begin{equation*}
	W(x)=c \left(1 +|x|^2\right)^{-\frac{d-2}{2}} \text{with some constant}\;\;c>0
	\end{equation*}
	is the ground state of \eqref{equ:Hart} and plays an important role in the long time dynamics of the solutions of equation \eqref{equ:Hart}. In fact, combining the sharp Sobolev inequality \cite{Aubin, LiebL:book} with the sharp Hardy-Littlewood-Sobolev inequality \cite{Lieb:HLS:sp const, LiebL:book}, we have the  following characterization of the ground state $W$. That is, 
	
	\begin{proposition}[\cite{LiebL:book,MWX:Hart,MXZ:crit Hart:f rad}]\label{prop:sharp const} Let $d\geq 5$ and $ u \in \dot H^1(\R ^d)$, we have:
		\begin{equation*}
		\iint_{\R^d \times \R^d} \frac{|u(x)|^2\, |u(y)|^2}{|x-y|^4} \; dxdy	\leq C_d^4\, \|\nabla u \|^4_{L^2}
		\end{equation*} 	
		for some constant $C_d>0$.  Moreover, there is equality in the above inequality if and only if  $$ u(x)=e^{i\theta_0}\lambda ^{-\frac{d-2}{2}} W\left(\frac{x-x_0}{\lambda} \right) $$ for any $\lambda>0, x_0\in {\R}^d,\theta_0\in [0,2\pi)$. 
	\end{proposition}

	%
	
	
	As for the solutions of \eqref{equ:Hart} with energy below the threshold $E(W)$, C. Miao, L. Zhao and the third author in \cite{MXZ:crit Hart:f rad} made use of the concentration-compactness-rigidity argument, which is pioneered by C. Kenig and F. Merle in \cite{KenM:NLS,KenM:NLW} to show the following dichotomy results in the radial case.  Later, D. Li and X. Zhang in  \cite{LiMZ:crit Hart} removed the radial assumption by use of the argument by R. Killip and M. Visan in \cite{KillV} on the energy-critical NLS.
	\begin{theorem}(\cite{LiMZ:crit Hart,MXZ:crit Hart:f rad}) \label{thm:below thresh}Let $d\geq 5$, and $u$ be a solution of \eqref{equ:Hart} with
		\begin{equation*}
		u_0 \in \dot{H}^1(\R^d),\; E(u_0)<E(W).
		\end{equation*}
		Then:
		\begin{enumerate}
			\item[\rm(a)] If $\Vert \nabla u_0 \Vert_{L^2}<\Vert \nabla W \Vert_{L^2}$, then $I=\R$ and $\Vert u \Vert_{L_{t}^{6} (\R;L_{x}^{\frac{6d}{3d-8}})} < \infty$. 
			\item[\rm(b)] If $\Vert \nabla u_0 \Vert_{L^2}>\Vert \nabla W \Vert_{L^2}$, and either $u_0 \in  L^2$ is radial or $x\cdot u_0\in L^2$, then $T_{\pm}<\infty$.
		\end{enumerate}	
	\end{theorem}

	As for the threshold solutions of \eqref{equ:Hart}, it is clear that the ground state $W$ is the global radial solution. In fact, there are two other radial threshold solutions $W^{\pm}$, which were constructed by C.  Miao, Y. Wu and the third author in \cite{MWX:Hart} under the non-degeneracy  assumtion of the ground state by use of the approximation argument with higher order and compact argument in \cite{DMerle:NLS:ThreshSol, DMerle:NLW:ThreshSol, LiZh:NLS, LiZh:NLW}. Recently, the authors in \cite{LLTX:Nondeg, LLTX:g-Hart,LTX:Nondeg} combine the spherical harmonic decomposition, the Funk-Hecke formula of the spherical harmonic functions with the Moser iteration method to solve the non-degeneracy of the ground state in the energy space.  That is, we have
	
	\begin{theorem}[\cite{LLTX:g-Hart,MWX:Hart}]\label{thm:thresd sol} Let $d\geq 5$, there exist two radial solutions $W^\pm$ of \eqref{equ:Hart} on the maximal lifespan interval $(-T_{-}(W^{\pm}), T_{+}(W^{\pm}))$ with
		initial data $W^{\pm}_{0} \in \dot H^1(\R^d)$  such that
		\begin{enumerate}
			\item[\rm(a)] $E(W^\pm)=E(W)$, $T_+(W^{\pm})=\infty$ and
			\begin{equation*} \aligned
			\lim_{t\rightarrow +\infty}W^{\pm}(t) =  W \; \text{in} \; \dot H^1(\R^d).
			\endaligned
			\end{equation*}
			
			\item[\rm(b)] $\big\|\nabla W^-_0\big\|_{2} < \big\|\nabla W\big\|_{2}$, $T_-(W^{-})=\infty$ and $W^-$ scatters for the negative time.
			
			\item[\rm(c)] $\big\|\nabla W^+_0\big\|_{2} > \big\|\nabla W\big\|_{2}$ and $T_-(W^{+})<\infty$.
		\end{enumerate}
	\end{theorem}
	
	As for the classification of the radial solutions of \eqref{equ:Hart} with the threshold energy $E(W)$, we have 
	
	%
	%
	%
	
	\begin{theorem}[\cite{LLTX:g-Hart}] \label{thm:thresh class} Let $d\geq 5$,  and $u$ be a threshold solution of \eqref{equ:Hart} with radial data $u_0\in \dot H^1(\R^d)$, (i. e. $E(u_0)=E(W)$) and $I$ be the maximal interval of existence. Then we have:
		\begin{enumerate}
			\item[\rm(a)] If $\Vert \nabla u_0 \Vert_{L^2}<\Vert \nabla W \Vert_{L^2}$, then $I=\R$. Furthermore, either $u=W^{-}$ up to the symmetries of the equation, or
			\begin{equation*}
			\Vert u \Vert_{L_{t}^{6} (\R;L_{x}^{\frac{6d}{3d-8}})} < \infty.
			\end{equation*}
			\item[\rm(b)] If $\Vert \nabla u_0 \Vert_{L^2}=\Vert \nabla W \Vert_{L^2}$, then $u=W$ up to the symmetries of the equation.
			\item[\rm(c)] If $\Vert \nabla u_0 \Vert_{L^2}>\Vert \nabla W \Vert_{L^2}$, and $u_0 \in  L^2(\R^d)$, then either $u=W^{+}$ up to the symmetries of the equation, or $I$ is finite.
		\end{enumerate}	
	\end{theorem}

	In this paper,  we will make use of the one-pass lemma, which was pionerred by K.~Nakanishi and W.~Schlag in \cite{NakR, NakS:NLKG, NakS:NLS, NakS:NLKG:non} (see also \cite{KriNS:NLKG, KriNS:NLW rad, KriNS:NLW non, NakS:book,Roy}), to give a classification of the radial solutions of \eqref{equ:Hart} with the energy slightly larger than that of the ground states. 
	
	Let $d\geq 5$,  by the invairances of \eqref{equ:Hart} under the rotation and scaling symmetries,  we denote the two dimensional manifold of the stationary solutions by
	\begin{equation*}
	\mathcal{W}:=\left\{W_{\theta,\sigma}:=e^{i\theta}e^{\frac{d-2}{2}\sigma}W(e^{\sigma}x)\, |\; \theta,\sigma\in \R \right\}\subset\dot H^1_{rad}(\R^d),
	\end{equation*} 
	the distance from $\mathcal{W}$ and its $\delta$-neighborhood  by
	\begin{equation*}
	d_{\mathcal{W}} (\varphi):=\inf_{\theta,\sigma \in \R} {\Vert{\varphi-W_{\theta,\sigma}} \Vert}_{\dot H^1} , \;\; 
	B_{\delta}(\mathcal{W}):=\{\varphi \in \dot H^1_{rad} |d_{\mathcal{W}} (\varphi)<\delta\},
	\end{equation*}
	and denote
	\begin{equation*} 
	\begin{aligned}
	\mathcal{H}^{\epsilon}:=&\{\; \varphi\in\dot H^1_{rad}(\R^d) \, |\;E(\varphi)<E(W)+\epsilon^2 \},\\
	\mathcal{S}_{\pm}:=&\{u_0\in\dot H^1_{rad} (\R^d)\, |\text{ the solution $u$ scatters as} \enspace t\rightarrow\pm\infty\}, \\
	\mathcal{B}_{\pm}:=&\{u_0\in\dot H^1_{rad} (\R^d)\, |\text{ the solution $u$ blows up as} \pm t>0 \}.
	\end{aligned}
	\end{equation*}
	\begin{theorem}\label{thm:above thresh} Let $d\geq 5$, there is an absolute constant  $\epsilon_\star \in(0,1)$ such that for each $\epsilon \in \left(0,\epsilon_\star\right]$, there exist a relatively closed set  $\mathcal{X}_\epsilon \subset \mathcal{H}^\epsilon$, and a continuous function $\Theta:\mathcal{H}^\epsilon \setminus\mathcal{X}_\epsilon \rightarrow \{ \pm 1 \} $, with the following properties.
		\begin{enumerate}
			\item[$(1)$] $\mathcal{W} \subset \mathcal{X}_\epsilon \subset B_{C \epsilon} \left( \mathcal{W}\right)$ for some absolute constant $C\in(0,\infty)$. 
			\item[$(2)$] The values of $\Theta$ are independent of $\epsilon$.
			\item[$(3)$] For each $u_0 \in \mathcal{H}^\epsilon$ and the solution of \eqref{equ:Hart},
			\begin{equation*}
			\begin{aligned}
			I_0(u):=\left\{t\in I(u)\; |\;  u(t)\in\mathcal{X}_\epsilon \right\}
			\end{aligned}
			\end{equation*}
			is either empty or an interval. Therefore, $I(u)\setminus I_0(u)$ consists of at most two open intervals. Moreover, let $\sigma\in \{\pm\}$. 
			\begin{enumerate}
				\item[(a)] If $\Theta\left( u(t) \right)=+1$ for t close to $T_\sigma(u)$, then $u_0 \in \mathcal{ S }_{\sigma}$. 
				\item[(b)] If $\Theta\left( u(t) \right)=-1$ for t close to $T_\sigma(u)$ and $u_0\in L^2\left( \mathbb{R}^d \right)$, then  $u_0 \in \mathcal{ B }_{\sigma}$.
			\end{enumerate}
		\end{enumerate}
		
	\end{theorem}	
	
	\begin{remark} We give some remarks.
		\begin{enumerate}
			\item 
			In other words, every solution with energy less than $E(W)+\epsilon ^2$ can stay in $\mathcal{X}_\epsilon$ only for an interval of time, though it can be the entire existence time. Once the solution gets out of $\mathcal{X}_\epsilon$, it has to either scatter or blow-up, according to the sign function $\Theta\left( u(t) \right)$, though we need an additional condition $u_0\in L^2\left( \mathbb{R}^d \right)$ to ensure the blow-up. The above result relies upon two key components: the ejection mechanic and the one-pass lemma. 
			
			\item 
			Compared with the subcritical NLS and NLKG in \cite{KriNS:NLKG, NakS:NLKG, NakS:book, NakS:NLS, NakS:NLKG:non}, we have to introduce a scaling parameter in order to describe the dynamics near the manifold $
			\mathcal{W}$ in the critical case, and we have to take it into acount in the proof of the one-pass lemma by comparing the contribution in the variational region and in the hyperbolic region. we can also refer to \cite{KriNS:NLW rad, KriNS:NLW non, NakR,Roy}. 
			
			\item 
			For other dynamics of Hartree equation, please refer to \cite{CaoG,GinV:Hart,KriLR:crit Hart,KriMR:subcrit Hart,LiMZ:crit Hart,LiZh:class Hart,LLTX:g-Hart, MXZ:crit Hart:def rad,MXZ:subcrit Hart,MXZ:m crit Hart,MXZ:m crit Hart:blowup,MXZ:crit Hart:def non, Nak:Hart, Zhou}.  
			
		\end{enumerate}
	\end{remark}
	
	The above properties hold in both the time directions. Concerning the relation between forward and backward dynamics, we have
	
	\begin{theorem}\label{thm:exist} For any $\epsilon>0$, each of the four intersections 
		\begin{equation*}
		\begin{aligned}
		\mathcal{S}_{-} \cap \mathcal{S}_{+}, \enspace \mathcal{B}_{-} \cap \mathcal{S}_{+},\enspace \mathcal{S}_{-} \cap \mathcal{B}_{+},\enspace \mathcal{B}_{-} \cap \mathcal{B}_{+}
		\end{aligned}
		\end{equation*}
		has non-empty interior in $\mathcal{H}^\epsilon \cap L^2\left( \mathbb{R}^d \right)$.
	\end{theorem}	
	
	In particular, there are infinitely many solutions which scatter on one side of time and blow up on the other. According to Theorem \ref{thm:above thresh}, such transition can occur only by changing $\Theta(u)$ from $+1$ to $-1$ or vice versa, going through small	neighborhood of the manifold $\mathcal{W}$, but the change is allowed at most once for each solution.
	
	%
	
	The paper is organized as follows. In Section \ref{sect:notation}, we state the spectral properties of the linearized operator $\mathcal{L}$ around the ground state $W$, which rely on the nondegeneracy of the ground state $W$. In Section \ref{sect:main thm}, we mainly state the one-pass lemma and the ejection lemma, which are the key steps of this paper. For the ejection lemma, we can show that the dynamics can be ruled by that of its unstable eigenmode of the linearized operator around $\mathcal{W}$ just as the dynamics of solutions of linear differential equations. We explicitly construct the nonlinear distance functionals $\widetilde{d}_{\mathcal{W}}$ and $\Theta$  in terms of the eigenfunctions of the linearized operator. In Section \ref{sect:one pass}, we give the proof of the one-pass lemma. The process involves a parameter $R$, which is the cut-off radius for the localization. It precludes “almost homoclinic orbits”, i.e., those solutions starting in, the moving away from, and finally returning to, a small neighborhood of the ground states. Indeed, this part of analysis is the main part of this paper. In Section \ref{sect:asymp behr},  we prove Proposition \ref{prop:asyp behar}, which is shown by contradiction. For the case $\Theta(u)=+1$, the existence of the critical element $U_c$ leads to the contradiction. In Section \ref{sect:pf main thm}, we prove Theorem \ref{thm:above thresh}, which implies that the fate of the solutions depends on $\Theta(u(t))$ when $u$ is ejected. In Section \ref{sect:pf exist}, we prove Theorem \ref{thm:exist} which shows the existence of the radial solutions in the four intersections.
	
	\noindent \subsection*{Acknowledgements.}
G. Xu is supported by National Key Research and Development Program of China (No.~2020YFA0712900) and by NSFC (No. ~12371240, No.~12431008).

	%
	%
	%
	%
	\section{Notation}\label{sect:notation}
	
	In this section, we introduce some notation that will be used in the context.
	
	If $x$ is a complex number, then $ x=\mathrm{Re}(x)+i\, \mathrm{Im}(x)=x_1+ix_2$. Here $\mathrm{Re}(x)$ and $x_1$ (resp. $\Im(x)$ and $x_2$) denote the real part (resp. the imaginary part) of $x$. 
	
	Given two real numbers $x$, $y$, $x\lesssim y$ (resp. $\gtrsim$) means that there exists a universal constant $C>0$ such that $x\leq Cy$ (resp. $x\geq Cy$). For any funtion $f$ on $\mathbb{R}$ or $\left[0, \infty \right)$, and for any $m \in \left(0, \infty \right)$, we denote by $f_{m}$ the following rescaled function
	$
	f_{m}(r):=f\left( r/m \right).
	$
	
	Let $S_{a}^{\sigma}$ be the one-parameter group of dilation operators defined as follows
	\begin{equation*}
	S_{a}^{\sigma}f(x):=e^{(d/2+a)\sigma}f\left( e^{\sigma}x\right),
	\end{equation*}
	and let $S_{a}':=\partial_{\sigma}S_{a}^{\sigma}|_{\sigma=0}$ be its generator. It is easy to see that the adjoint is given by $\left( S_{a}^{\sigma}\right)^{*}=S_{-a}^{-\sigma}$, hence by differentiating in $\sigma$, we have
	$
	\left( S_{a}'\right)^*=-S_{-a}'.
	$
	
	Let $d\geq 5$ and  $L_{x}^{p}(\R^d)=L^p(\R^d)$ denote the standard $L^p$ space on $\mathbb{R}^d$. The Fourier transform of $\varphi \in \mathcal{S}^{'}(\mathbb{R}^d) $ is denoted by $\widehat{\varphi}$. Let $\phi:\mathbb{R} \rightarrow \left[0, \infty\right)$ be a smooth cut-off function, which is an even function satisfying  $t\phi'(t)\leq0$,  and 
	\begin{equation}\label{funct:cutf funct}
	\phi(t)=1 \;\text{ if }\;
	|t|\leq 1 , \; \text{  and}\; \; 
	\phi(t)=0
	\; \text{ if } \;  |t|\geq 2   . 
	\end{equation}
	For any $m>0$, the Littlewood-Paley operators $P_{<m}$,\enspace$P_{\geq m}$ and $P_{m}$ are defined by
	\begin{equation*}
	\begin{aligned}
	\widehat{P_{<m}f}(\xi):=\phi_m(|\xi|)\widehat{f}(\xi),\quad P_{\geq m}f :=f-P_{<m}f,\quad P_{m}f:=P_{<m} f -P_{<m/2}f.
	\end{aligned}
	\end{equation*}

	We denote by $S,W,\bar{S}$ and $N$ the following space-time spaces on $\mathbb{R}^{1+d}$
	\begin{equation*}
	\begin{aligned}
	S(I):=L^{6}\left( I;L^{\frac{6d}{3d-8}} (\mathbb{R}^d)\right), \quad W(I):=L^{6}\left( I;L^{\frac{6d}{3d-2}} (\mathbb{R}^d)\right),\\
	\bar{S}(I):=L^{3}\left( I;L^{\frac{6d}{3d-4}} (\mathbb{R}^d)\right),\quad N(I):=L^{\frac{3}{2}}\left( I;L^{\frac{6d}{3d+4}} (\mathbb{R}^d)\right).
	\end{aligned}
	\end{equation*}
	We also use the homogeneous Sobolev spaces defined by completion of the Schwartz space with respect to the norm
	\begin{equation*}
	\Vert v\Vert_{X^1}:=\Vert \nabla v\Vert_{X}\quad \left( X=W,\bar{S},N\right).
	\end{equation*}
	For any interval $J\subset \mathbb{R}$ and any function space $X$ on $\mathbb{R}^{1+d}$, the restriction of $X$ onto $J$ is denoted by $X(J)$. The Sobolev embedding implies
	\begin{equation*}
	\Vert v\Vert_{S(J)} \lesssim \Vert v\Vert_{W^1(J)}.
	\end{equation*}
	We recall the $L^p$ decay and the Strichartz estimates. 
	\begin{proposition}[Dispersive and Strichartz estimates, \cite{Caz:book}]\label{prop:disp ST est}
		For any $p\in\left[2,\infty\right]$, we have
		\begin{equation*}
		\Vert e^{it\Delta}\varphi\Vert_{L^p} \lesssim |t|^{-d(1/2-1/p)}\Vert \varphi \Vert_{L^{p'}},
		\end{equation*}
		where $p'=p/(p-1)$ is the H\"older exponent of $p$, and 
		\begin{equation}\label{est:St est}
		\begin{aligned}
		\Vert e^{-it\Delta}\varphi\Vert_{\left( L_{t}^{\infty} L_{x}^2 \cap \bar{S} \cap W \right)(\R)} \lesssim \Vert \varphi \Vert_{L^2},\\
		\Vert \int_{0}^{t} e^{-i(t-s)\Delta}F(s)ds \Vert_{\left( L_{t}^{\infty} L_{x}^2 \cap \bar{S} \cap W \right)(J)} \lesssim \Vert F\Vert_{N(J)}.
		\end{aligned}
		\end{equation}
	\end{proposition}

	In addition to regard the ground state $W$ of \eqref{equ:Hart} as the extremizer of the sharp Hardy-Littlewood-Sobolev inequality and the sharp Sobolev inequality in Proposition \ref{prop:sharp const},  we have the similar constrained variational characterizaion as those in \cite{IbrMN:NLKG, NakS:book}.  
	\begin{proposition}[Variational characterization,  \cite{IbrMN:NLKG, NakS:book}]\label{prop:GS char}For the extremizer $W$ in Proposition \ref{prop:sharp const}, we have
		\begin{equation}\label{char:Neh char}
		\begin{aligned}
		E(W)=&\inf\{E(\varphi)|0\neq\varphi \in\dot H^1(\R^d),\; K(\varphi)=0\} \\
		=&\inf\{G(\varphi)|0\neq\varphi \in\dot H^1(\R^d),\; K(\varphi)\leq0\} \\
		=&\inf\{I(\varphi)|0\neq\varphi \in\dot H^1(\R^d),\; K(\varphi)\leq0\}\ 
		\end{aligned}
		\end{equation}
		where
		\begin{equation}\label{quan:KIG}
		\begin{aligned}
		K(\varphi):=&\Vert \nabla \varphi \Vert_{L^2}^{2} -{\Vert{\left(|x|^{-4}*|\varphi|^2\right)|\varphi|^2}\Vert}_{L^1},\\
		G(\varphi):=&E(\varphi)-K(\varphi)/4=\Vert \nabla \varphi \Vert_{L^2}^{2} /4,\\
		I(\varphi):=&E(\varphi)-K(\varphi)/2={\Vert{\left(|x|^{-4}*|\varphi|^2\right)|\varphi|^2}\Vert}_{L^1}/4.
		\end{aligned}
		\end{equation}
	\end{proposition}

	Let us denote the linearized operator $i \mathcal{L}$ around the ground state $W$ of \eqref{equ:Hart} by
	\begin{align*}
	\mathcal{L}f:=L_{+}f_1+iL_{-}f_2,
	\end{align*}
	where
	\begin{equation*}
	\begin{aligned}
	L_{+}h_1:=&-\Delta h_1-\left( |x|^{-4}*|W|^2 \right)h_1-2\left( |x|^{-4}*(Wh_1) \right)W,\\
	L_{-}h_2:=&-\Delta h_2-\left( |x|^{-4}*|W|^2 \right)h_2.
	\end{aligned}
	\end{equation*}
	
	We define the linearized energy $\Phi$ by
	\begin{equation*}
	\Phi(h):=\frac12 \int_{\R^d} (L_+ h_1) h_1 \; dx + \frac12
	\int_{\R^d} (L_- h_2) h_2 \; dx =\frac12 \langle \mathcal{L} h,h \rangle.
	\end{equation*}
	
	Recently, the authors in \cite{LLTX:Nondeg, LLTX:g-Hart,LTX:Nondeg} combine the spherical harmonic decomposition, the Funk-Hecke formula of the spherical harmonic functions with the Moser iteration method to solve the non-degeneracy of the ground state in the energy space.  Based on the non-degeneracy of the ground state, we have the following spectral properties of the linearized operator $i\mathcal{L}$ by the standard variational argument in \cite{Wein:Modul stab}, (see also  \cite{DMerle:NLS:ThreshSol, DMerle:NLW:ThreshSol}).
	\begin{proposition}[Spectral property of the linearized operator $\mathcal{L}$, \cite{LLTX:g-Hart, MWX:Hart}]\label{prop:spect} Let $d\geq 5$.
		\begin{enumerate}
			\item[\rm (a)]By the symmetries of $\eqref{equ:Hart}$ under the phase rotation and the scaling, 
			$
			iW $, $\widetilde{W}=\frac{d-2}{2}W+x\cdot \nabla W 
			$
			are two null eigenfunctions in the radial case.
			\item[\rm (b)] It has two simple eigenvalues $\pm\mu$ (with $\mu>0$) and two smooth, exponentially decaying  eigenfunctions $g_\pm=g_1\mp ig_2$ that satisfy $i\mathcal{L}g_\pm=\pm\mu g_\pm.$ In other words, $L_{+}g_1=-\mu g_2$ and $L_{-}g_2=\mu g_1$.
			\item[\rm (c)]  $L_{-}\geq0$ and $Ker(L_{-}) =$span$\{W\}$ on $\dot H_{rad}^1(\R^d)$.
			\item[\rm (d)]  $Ker(L_{+})=$span$\{\widetilde{W}\}$ on $\dot H_{rad}^1(\R^d)$.
			\item[\rm (e)] There exists a constant $c>0$ such that  for all radial function $ h
			\in G_{\bot}$, we have
			\begin{equation*}
			\begin{aligned} 
			\Phi(h)\geq c\big\| h \big\|^2_{\dot H^1},
			\end{aligned}
			\end{equation*}
			where \begin{equation*}
			\label{G orthy} G_{\bot}  =\left\{v\in \dot H^1, (iW, v)_{\dot
				H^1}=(\widetilde{W}, v)_{\dot H^1}  = \omega(g_\pm, v)  =
			0\right\}.
			\end{equation*}
		\end{enumerate}
	\end{proposition}
	
	\begin{lemma} For the functionals defined in \eqref{quan:energy} and \eqref{quan:KIG}, we have
		\begin{equation}\label{est:EK expa}
		\begin{aligned}
		E(W+v)=&E(W)+\frac{1}{2}\langle \mathcal{L}v,v\rangle-C(v),\\
		K(W+v)=&-2\langle \nabla W,\nabla v\rangle+O(\Vert v\Vert_{\dot H^1}^2),
		\end{aligned}
		\end{equation}
		where 
		\begin{equation*}
		C(v):=\frac{1}{4}\iint \frac{\big|v(t,x)\big|^2\big|v(t,y)\big|^2+4W(x)v_1(t,x)\big|v(t,y)\big|^2}{|x-y|^4} dxdy=O(\Vert v \Vert_{\dot H^1}^3).
		\end{equation*}
	\end{lemma}
	
	\begin{proof}
		\begin{equation*}
		\begin{aligned}
		E(W+v)=&\frac12 \Vert \nabla(W+v)\Vert_{L^2}^2-\frac{1}{4} \iint \frac{\big|(W+v)(x)\big|^2\big|(W+v)(y)\big|^2}{|x-y|^4} dxdy\\  =&E(W)+\frac{1}{2}\langle \mathcal{L}v,v\rangle-\frac{1}{4}\iint \frac{\big|v(t,x)\big|^2\big|v(t,y)\big|^2+4W(x)v_1(t,x)\big|v(t,y)\big|^2}{|x-y|^4} dxdy\\=&E(W)+\frac{1}{2}\langle \mathcal{L}v,v\rangle+O(\Vert v \Vert_{\dot H^1}^3).
		\end{aligned}
		\end{equation*}	
		\begin{equation*}
		\begin{aligned}
		K(W+v)=&\Vert \nabla(W+v)\Vert_{L^2}^2-\iint \frac{\big|(W+v)(x)\big|^2\big|(W+v)(y)\big|^2}{|x-y|^4} dxdy\\  =&\Vert \nabla v\Vert_{L^2}^2 - 2\langle \nabla W,\nabla v\rangle\\-& \iint \frac{2|W(x)|^2|v(y)|^2+|v(x)|^2|v(y)|^2+4W(y)v_1(y)|v(x)|^2+4W(x)W(y)v_1(x)v_1(y)}{|x-y|^4} dxdy\\=&- 2\langle \nabla W,\nabla v\rangle+O(\Vert v \Vert_{\dot H^1}^2).
		\end{aligned}
		\end{equation*}	
	\end{proof}

	%
	%
	%
	%
	\section{Proof of main theorem.}\label{sect:main thm}
	The proof of Theorem \ref{thm:above thresh} relies upon some propositions stated below. 
	\subsection{Orthogonal decomposition.}
	The following proposition gives a decomposition of a vector $\varphi\in \dot H^1$ close to the ground states $\mathcal{W}$, taking account of two parameters (the rotation parameter and the scaling parameter) and a constraint (the so-called orthogonality condition).

	\begin{proposition}(Orthogonal decomposition of $\varphi$)\label{prop:orth decom}. There exist an absolute constant $0<\delta_E\ll1$ and a $C^1$ function $\left(\widetilde{\theta},\widetilde{\sigma} \right):B_{\delta_E}(\mathcal{W}) \rightarrow (\R/2\pi\Z)\times \R$ with the following properties. For any $\varphi\in B_{\delta_E}(\mathcal{W})$, putting 
		\begin{equation}\label{decom:var}
		\varphi=e^{i\widetilde{\theta}(\varphi)}S_{-1}^{\widetilde{\sigma}(\varphi)}(W+v),
		\end{equation}	
		we have 
		\begin{equation*}
		\begin{aligned}
		(v,iW)_{\dot H^1}=(v,\widetilde{W})_{\dot H^1}=0, \; d_{\mathcal{W}}(\varphi)\sim \Vert v\Vert_{\dot H^1}.
		\end{aligned}
		\end{equation*}
		Moreover, $\left(\widetilde{\theta}(\varphi),\widetilde{\sigma} (\varphi)\right) \in  (\R/2\pi\Z)\times \R$ is unique for the above property. 
		
		Furthermore, if $\Vert \varphi-W_{\theta,\sigma}\Vert_{\dot H^1} \ll 1$ for some $(\theta,\sigma)\in\R^2$, then 
		\begin{equation*}
		\big|\left(e^{i\widetilde{\theta}(\varphi)}-e^{i\theta}, \widetilde{\sigma}(\varphi)-\sigma \right)\big|\lesssim \Vert \varphi-W_{\theta,\sigma}\Vert_{\dot H^1}.
		\end{equation*}
	\end{proposition}

\begin{remark}
	We use the null eigenfunctions $iW, \widetilde{W}$  here. While the authors in \cite{NakR} uses an alternative function $\chi$ instead of  zero resonance functions for the energy-critical NLS in dimension three.
\end{remark}

	\begin{proof}
		$(v,iW)_{\dot H^1}=(v,\widetilde{W})_{\dot H^1}=0$ and the uniqueness of $(\widetilde{\theta}(\varphi),\widetilde{\sigma}(\varphi))$ have been proved in \cite[Lemma 4.1]{MWX:Hart}. For the sake of completeness, we give the details.
		
		Now consider the functional
		\begin{align*}
		J(\theta, \sigma, \psi)=&  \left(J_0(\theta, \sigma, \psi), J_1(\theta, \sigma, \psi)\right)\\ =&  \left( (e^{-i\theta}S_{-1}^{-\sigma}(W+\psi)-W, iW)_{\dot H^1},\;\; (e^{-i\theta}S_{-1}^{-\sigma}(W+\psi)-W, \widetilde{W})_{\dot H^1} \right).
		\end{align*}
		Hence, by simple computations, we have
		\begin{align*}
		J(0,0,0)=0, \;\; \left| \frac{\partial J}{\partial(\theta, \sigma)}
		(0,0,0)\right| = \left| \begin{array}{cc} (-iW,iW)_{\dot H^1} & 0\\
		0 &  (-\widetilde{W},\widetilde{W})_{\dot H^1} \end{array} \right|
		\not = 0.
		\end{align*}
		Thus by the implicit function theorem, there exist $\delta>0$ and a $C^1$ function $\left(\theta,\sigma\right):\;B_\delta(0)\rightarrow\R^2$, such that 
		\begin{equation*}
		J(\theta(\psi), \sigma(\psi), \psi)=0,\;\;\theta(0)=\sigma(0)=0,
		\end{equation*}
		which is unique in $B_\delta(0)$ and a neighborhood of $0\in \R^2$. By modulation decomposition, we know that for any $\varphi\in B_\delta(\mathcal{W})$, there exists $(\alpha,\beta,\psi)\in (\R/2\pi \Z)\times\R\times B_\delta(0)$ such that $\varphi=e^{i\alpha} S_{-1}^{\beta}(W+\psi)$. Then we put
		\begin{equation*}
		\widetilde{\theta}(\varphi):=\theta(\psi)+\alpha\in \R/2\pi \Z,\;\;\widetilde{\sigma}(\varphi):=\sigma(\psi)+\beta\in\R.
		\end{equation*}
		Then defining $v$ by \eqref{decom:var}, we have $v=e^{-i\theta(\psi)}S_{-1}^{-\sigma(\psi)}(W+\psi)-W$ and so
		\begin{align*}
		&0=J(\theta(\psi), \sigma(\psi), \psi)=\left( (v,iW)_{\dot H^1},(v,\widetilde{W})_{\dot H^1} \right),\\
		&\Vert  v\Vert_{\dot H^1}\lesssim |\theta(\psi)|+|\sigma(\psi)|+\Vert  \psi\Vert_{\dot H^1} \lesssim \Vert  \psi\Vert_{\dot H^1}.
		\end{align*}
		This implies
		\begin{equation*}
		\big|\left(e^{i\widetilde{\theta}(\varphi)}-e^{i\theta}, \widetilde{\sigma}(\varphi)-\sigma \right)\big|\lesssim \Vert \varphi-W_{\theta,\sigma}\Vert_{\dot H^1}.
		\end{equation*}
		Choosing ($\alpha$, $\beta$) such  that $d_{\mathcal{W}}(\varphi) \sim \Vert \varphi-e^{i\alpha} S_{-1}^{\beta}W \Vert_{\dot H^1}=\Vert  \psi\Vert_{\dot H^1}$, then $d_{\mathcal{W}}(\varphi)\gtrsim\Vert v\Vert_{\dot H^1}$. Meanwhile,
		\begin{equation*}
		\begin{aligned}
		d_{\mathcal{W}}(\varphi)=&\inf_{\theta,\sigma \in \R} \Vert e^{i\widetilde{\theta}(\varphi)}S_{-1}^{\widetilde{\sigma}(\varphi)}(W+v) -W_{\theta,\sigma}\Vert_{\dot H^1}\\\leq& \Vert W_{\widetilde{\theta}(\varphi),\widetilde{\sigma}(\varphi)} +e^{i\widetilde{\theta}(\varphi)}S_{-1}^{\widetilde{\sigma}(\varphi)}v -W_{\widetilde{\theta}(\varphi),\widetilde{\sigma}(\varphi)}\Vert_{\dot H^1}=\Vert v\Vert_{\dot H^1}.
		\end{aligned}
		\end{equation*}
		This proves $d_{\mathcal{W}}(\varphi)\sim \Vert v\Vert_{\dot H^1}$.
		
	\end{proof}

	\subsection{Evolution around the ground states.} The following proposition describes more precisely the decomposition in Proposition \ref{prop:orth decom}, taking into account the spectral properties of $i \mathcal{L}$. This decomposition does not use the radial symmetry.
	
	\begin{proposition}(Spectral decomposition of $v$)\label{prop:spect decom}. For any $v \in \dot H^1$, there exists a unique decomposition
		\begin{equation}\label{decom:v+}
		v=\lambda_+g_{+} + \lambda_-g_{-}+\gamma,\;\;\lambda_{\pm}\in\R,\;\;\gamma\in  \dot H^1,
		\end{equation}
		such that $\omega(g_{\pm},\gamma)=0$ (This implies that $\gamma\in G_{\bot}$).
		
		After normalizing $g_{\pm}$ (or $g_1$ and $g_2$) such that 
		\begin{equation}\label{con:ome}
		\omega(g_{+},g_{-})=2\langle g_{1},g_{2} \rangle=1,\; \pm\omega(W,g_{\mp})=\langle W,g_2 \rangle>0,
		\end{equation}
		the above decomposition is given by
		\begin{equation*}
		\lambda_{\pm}=\pm \omega(v,g_{\mp}).
		\end{equation*}
		Putting $\lambda_1:=(\lambda_+ +\lambda_-)/2$ and $\lambda_2:=(\lambda_+ -\lambda_-)/2$, it can also be written as 
		\begin{equation}\label{decom:v1}
		v=2\lambda_{1}g_1-2i\lambda_{2}g_2+\gamma,\; \lambda_1=(v_1|g_2),\; \lambda_2=-(v_2|g_1),
		\end{equation}
		with $(\gamma_1|g_2)=(\gamma_2|g_1)=0$.
	\end{proposition}
	
	\begin{proof}
		We only show that we can normalize $g_+$ and $g_-$ such that  \eqref{con:ome} holds. The remaining can be obtained by immediate computations.
		
		Since $L_-\geq0$ with Ker$(L_-)$=span$\{W\}$, we have
		$0<c:=\langle L_{-}g_2,g_2 \rangle=-\langle L_{+}g_1,g_1 \rangle=\mu\omega(g_+,g_-)/2$. By normalizing $g_+$ and $g_-$, we can get $\omega(g_+,g_-)=1$.
		
		Now we show that $\langle W,g_2 \rangle\neq0$. We prove it by contradiction. Assume that $\langle W,g_2\rangle=0$, then $0=\langle W,-\mu g_2\rangle=\langle L_+W,g_1\rangle=2\langle \Delta W,g_1\rangle$. Let $0<\delta\ll1$ and $v=\alpha W+\delta g_1$ with $\alpha=O(\delta^2)$ to be chosen shortly. By expansion of $E,K$ in \eqref{est:EK expa} and
		\begin{equation*}
		\frac{1}{2}\langle L_+v,v\rangle=\frac{1}{2}\langle \mathcal{L}v,v\rangle-\iint \frac{W(x)W(y)v_2(x)v_2(y)}{|x-y|^4} dxdy=\frac{1}{2}\langle \mathcal{L}v,v\rangle+O(\Vert v\Vert_{\dot H^1}^2),
		\end{equation*}
		we can get 
		\begin{equation*}
		\begin{aligned}
		E(W+v)=E(W)+\frac{1}{2}\langle L_+v,v\rangle-O(\Vert v\Vert_{\dot H^1}^2)+O(\Vert v\Vert_{\dot H^1}^3)<E(W)-\frac{c}{2}\delta^2+O(\delta^3),
		\end{aligned}
		\end{equation*}
		\begin{equation*}
		K(W+v)=-2\langle \nabla W,\nabla v\rangle+O(\Vert v\Vert_{\dot H^1}^2)=-2\alpha \Vert W\Vert_{\dot H^1}^2+O(\delta^2).
		\end{equation*}
		Thus one can find $\alpha=O(\delta^2)$ such that $K(W+v)=0$, which contradicts $\eqref{char:Neh char}$. Hence $\langle W,g_2\rangle\neq0$. 
		
		Thus we prove that we can normalize $g_\pm$ such that $\eqref{con:ome}$ holds. 
	\end{proof}
	
	The following proposition aims at describing the dynamics of the solution near the ground states, using the decomposition in Proposition \ref{prop:orth decom}. Again, this does not use the radial symmetry.
	\begin{proposition}(Linearization and parametrization around $\mathcal{W}$)\label{prop:orth decom:u}. Let $u$ be a solution of \eqref{equ:Hart} on an interval I in the form \eqref{decom:var}, i.e. $(\theta,\sigma,v):I\rightarrow (\R/2\pi\Z)\times \R \times \dot H^1$ is defined by 
		\begin{equation}\label{decom:u}
		u(t)=e^{i\theta(t)}S_{-1}^{\sigma(t)}\left(W+v(t) \right),\; \theta(t):=\widetilde{\theta}(u(t)),\;\sigma(t):=\widetilde{\sigma}(u(t)).
		\end{equation}
		Then, letting $\tau:I\rightarrow\R$ such that $\tau'(t):=e^{2\sigma(t)}$, we have 
		\begin{equation} \label{equ:linez}
		\partial_{\tau}v=i\mathcal{L}v-\left( i\theta_{\tau}+\sigma_{\tau}S_{-1}'\right)(W+v)-iN(v),
		\end{equation}
		where $\theta_{\tau}=\frac{\partial\theta}{\partial\tau}$ etc., and 
		\begin{equation*}
		\begin{aligned}
		N(v):=&(W+v)\left( |x|^{-4}*|W+v|^2 \right)\\-&W\left( |x|^{-4}*|W|^2 \right)-v\left( |x|^{-4}*|W|^2 \right)-2\left( |x|^{-4}*(Wv_1) \right)W.
		\end{aligned}
		\end{equation*}
		
		Furthermore, $\partial_{\tau}(\theta,\sigma)=O(\Vert v\Vert_{\dot H^1})$ and, decomposing v by Proposition 3.2, 
		\begin{equation}\label{est:dyn lam+}
		\partial_{\tau}\lambda_\pm=\pm\mu\lambda_\pm+O(\Vert v\Vert_{\dot H^1}^2),
		\end{equation}
		or equivalently,
		\begin{equation}\label{est:dyn lam1}
		\partial_{\tau}\lambda_1=\mu\lambda_{2}+O(\Vert v\Vert_{\dot H^1}^2),\;\;	\partial_{\tau}\lambda_2=\mu\lambda_{1}+O(\Vert v\Vert_{\dot H^1}^2).
		\end{equation}
	\end{proposition}
	
	\begin{proof} Injecting the decomposition \eqref{decom:u} into the equation \eqref{equ:Hart}, we obtain
		\begin{equation*}
		v_t=ie^{2\sigma}\mathcal{L}v-\left(i\theta_t+\sigma_{t}S_{-1}' \right)(W+v)-ie^{2\sigma}N(v).	
		\end{equation*}
		Applying the change of variable $t\mapsto\tau$ with $\dot{\tau}=e^{2\sigma}$ to the above yields \eqref{equ:linez}.
		
		Next we consider the equations for the parameters. Using \eqref{equ:linez}, we have 
		\begin{equation*}
		\begin{aligned}
		\partial_{\tau}\lambda_+=\omega(\partial_{\tau}v,g_-)=&\omega(i\mathcal{L}v,g_-)-\omega\left( \left(i\theta_{\tau}+\sigma_{\tau}S_{-1}' \right)W ,g_-\right)\\+&\langle N(v),g_-\rangle-\omega\left( \left(i\theta_{\tau}+\sigma_{\tau}S_{-1}'\right)v ,g_-\right).
		\end{aligned}
		\end{equation*}
		Using  $\mathcal{L}^{*}=\mathcal{L}$ and $\left( S_{a}'\right)^*=-S_{-a}'$, we see that 
		\begin{equation*}
		\begin{aligned}
		\omega(i\mathcal{L}v,g_-)=\omega(i\mathcal{L}g_-,v)=\mu\lambda_+, \;\;  \omega\left( \left(i\theta_{\tau}+\sigma_{\tau}S_{-1}'\right)W ,g_-\right)=0,\\\omega\left( \left(i\theta_{\tau}+\sigma_{\tau}S_{-1}' \right)v ,g_-\right)=-\theta_{\tau}\langle v,g_-\rangle-\sigma_{\tau}\langle iv,S_{1}'g_-\rangle.
		\end{aligned}
		\end{equation*}
		Hence
		\begin{equation}\label{equ:dyn lam+}
		\partial_{\tau}\lambda_+=\mu\lambda_++\theta_{\tau}\langle v,g_-\rangle+\sigma_{\tau}\langle iv,S_{1}'g_-\rangle+\langle N(v),g_-\rangle.
		\end{equation}
		Similarly, we have
		\begin{equation}\label{equ:dyn lam-}
		\partial_{\tau}\lambda_-=-\mu\lambda_--\theta_{\tau}\langle v,g_+\rangle-\sigma_{\tau}\langle iv,S_{1}'g_+\rangle-\langle N(v),g_+\rangle.
		\end{equation}
		We differentiate the orthogonality condition $(v,iW)_{\dot H^1}=0$ with respect to $\tau$ to obtain the parameter estimate. Using \eqref{equ:linez}, we get
		\begin{equation*}
		0=(v_{\tau},iW)_{\dot H^1}=(i\mathcal{L}v,iW)_{\dot H^1}-\theta_{\tau}(iW,iW)_{\dot H^1}-\sigma_{\tau}(S_{-1}'v,iW)_{\dot H^1} -(iN(v),iW)_{\dot H^1}.
		\end{equation*}
		Combined with $(i\mathcal{L}v,iW)_{\dot H^1}=O(\Vert v\Vert_{\dot H^1})$ and $(iN(v),iW)_{\dot H^1}=O(\Vert v\Vert_{\dot H^1}^2)$, we obtain
		\begin{equation}\label{est:para}
		\partial_{\tau}(\theta,\sigma)=O \left(\Vert v\Vert_{\dot H^1}\right).
		\end{equation}	
		Thus, we have $\theta_{\tau}\langle v,g_\pm\rangle=\sigma_{\tau}\langle iv,S_{1}'g_\pm\rangle=\langle N(v),g_\pm\rangle=O(\Vert v\Vert_{\dot H^1}^2)$. Now plugging the above into \eqref{equ:dyn lam+} and \eqref{equ:dyn lam-}, we see that \eqref{est:dyn lam+} holds. This completes the proof.
	\end{proof}

	\subsection{Control by the linearized energy.} The following proposition shows that the orthogonal direction $\gamma$ of $v$ in \eqref{decom:v+} can be controlled by the linearized energy:
	\begin{proposition}(Control of orthogonal direction)\label{prop:orth cntl}. For function $\gamma$ defined in \eqref{decom:v+}, we have
		\begin{equation*}
		\Vert \gamma \Vert_{\dot H^1}^2 \sim \langle \mathcal{L}\gamma,\gamma \rangle.
		\end{equation*}
	\end{proposition}
	
	\begin{proof}
		By the defination of $\Phi$, we can easily get $\Phi(\gamma)\lesssim \Vert \gamma \Vert_{\dot H^1}^2$. Since $\gamma\in G_{\bot}$, we have $\Phi(\gamma)\gtrsim \Vert \gamma \Vert_{\dot H^1}^2$ by Proposition \ref{prop:spect}. Combined with the fact $\Phi(\gamma)=\frac12\langle \mathcal{L}\gamma,\gamma \rangle$, this leads to the conclusion.
	\end{proof}
	Hence in the subspace $\{ v\in\dot H^1|\; (v,iW)_{\dot H^1}=(v,\widetilde{W})_{\dot H^1}=0 \}$, we can define an equivalent norm $E$ using the decomposition of Proposition \ref{prop:spect decom}
	\begin{equation}\label{quan:equi nor}
	\Vert v \Vert_{E}^2:=\mu (\lambda_{1}^{2}+\lambda_{2}^{2})+\frac{1}{2}\langle \mathcal{L}\gamma,\gamma \rangle \sim \lambda_{1}^{2}+\lambda_{2}^{2}+\Vert \gamma \Vert_{\dot H^1}^2 \sim \Vert v \Vert_{\dot H^1}^2.
	\end{equation}
	In particular, in the decomposition of Proposition \ref{prop:orth decom}, we have
	\begin{equation*}
	d_{\mathcal{W}}(\varphi) \sim \Vert v \Vert_{\dot H^1} \sim \Vert v \Vert_{E}.
	\end{equation*}
	
	Henceforth, we assume that whenever a solution u of \eqref{equ:Hart} is in $B_{\delta_E}(\mathcal{W})$, the coordinates $\sigma,\theta,v, \lambda_\pm,\lambda_{1},\lambda_{2}$ and $\gamma$ are defined by \eqref{decom:u},\eqref{decom:v+},\eqref{decom:v1}, while $\tau(t)$ is a solution of $\dot{\tau}(t):=e^{2\sigma}$. In short, 
	\begin{equation}\label{coor arou W}
	\begin{aligned}
	e^{-i\theta}S_{-1}^{-\sigma}u-W=v=\lambda_+g_{+} + \lambda_-g_{-}+\gamma=2\lambda_{1}g_1-2i\lambda_{2}g_2+\gamma,\\ 
	0=(v,iW)_{\dot H^1}=(v,\widetilde{W})_{\dot H^1}=\omega(g_{\pm},\gamma)=\langle g_1,\gamma_2 \rangle= \langle g_2,\gamma_1 \rangle,\\
	\delta_E>d_{\mathcal{W}}(u) \sim \Vert v \Vert_{\dot H^1} \sim \Vert v \Vert_{E},\;  (\theta,\sigma)=\left( \widetilde{\theta}(u),\widetilde{\sigma}(u) \right), \;\dot{\tau}=e^{2\sigma}.	
	\end{aligned}	
	\end{equation}

	\subsection{Uniform local existence.}
	The following proposition ensures the existence of a solution $u$ of \eqref{equ:Hart} in a neighborhood of $\mathcal{W}$ as long as the scaling parameter $\sigma$ is bounded from above.
	\begin{proposition}(Uniform local existence in $\tau$)\label{prop:unif lwp}. There exists an absolute constant $\delta_L\in(0,\delta_{E}/2)$ such that for any solution $u$ of \eqref{equ:Hart} with $d_{\mathcal{W}}(u(0))=:\delta\in\left[ 0,2\delta_L \right]$, we have $T_{\pm}(u)>3e^{-2\sigma(0)}=:T_0$, $\pm\left( \tau(\pm T_0)-\tau(0) \right)>2$, and for $|t|\leq T_0$,
		\begin{equation*}
		\delta_{E}>d_{\mathcal{W}}(u(t))\sim\delta,\;\; \sigma(t)=\sigma(0)+O(\delta).
		\end{equation*}
	\end{proposition}
	\begin{proof} Let $u(0)\in B_\delta(\mathcal{W})$ for some constant $\delta\in(0,\delta_{E})$, whose smallness will be required in the following. 
		
		By rotation and scaling, we may reduce to the case $\theta(0)=\sigma(0)=0$. Using the stability theory, we have: If $\Vert u(0)-W \Vert_{\dot H^1}=d_\mathcal{W}(u(0))=:\delta$ small enough, then for $|t|\leq3$, $u$ exists and remains in $O(\delta)$ neighborhood of $W$. This implies that $d_\mathcal{W}(u(t))\sim d_\mathcal{W}(u(0))=\delta$ for $|t|\leq3$. Then by Proposition \ref{prop:orth decom:u}, we have, for $|t|\leq3$,
		\begin{equation*}
		\dot\tau=e^{2\sigma},\;\; \dot\sigma=\dot\tau\sigma_{\tau}=e^{2\sigma}O(\Vert v \Vert_{\dot H^1})=e^{2\sigma}O(d_\mathcal{W}(u(t)))=e^{2\sigma}O(\delta).
		\end{equation*}
		Integrating the second identity and using $\sigma(0)=0$, we have, for $|t|\leq3$,
		\begin{align*}
		e^{-2\sigma}=1+O(\delta),
		\end{align*}
		which leads to 
		\begin{equation*}
		\tau(t)=\tau(0)+\left( 1+O(\delta) \right)t,\;\;\text{for}\;\;|t|\leq3.
		\end{equation*}
		In particular, if $\delta>0$ is small enough, $\tau(t)$ reaches $\tau(0)\pm2$ within $I(u)$.
	\end{proof}
	
	Now we are ready to define the nonlinear distance $\widetilde{d}_{\mathcal{W}}$. Let $\varphi\in B_{\delta_E}(\mathcal{W})$. Consider the decomposition \eqref{decom:var} of $\varphi$. Then we define a local distance $d_0:B_{\delta_E}(\mathcal{W}) \rightarrow \left[0,\infty \right)$ by 
	\begin{equation}\label{quan:d_0}
	d_0(\varphi)^2:=E(\varphi)-E(W)+2\mu\lambda_{1}^2.
	\end{equation}
	
	As observed in \cite{NakS:NLS}, this is close to be convex in $\tau$ when the solution is ejected out of a small neighborhood of $\mathcal{W}$, but it may have small oscillation around minima in $\tau$. This is a difference for the Schrödinger equation from the Klein-Gordon equation, for which $d_{0}^{2}$ is strictly convex (see \cite{NakS:NLKG}). We could treat the possible oscillation as in \cite{NakS:NLS} by waiting for a short time before the exponential instability dominates, which would however bring a certain amount of complication to the statements as well as the proof.
	
	Here instead, we introduce a dynamical mollification of $d_{0}^2$, which yields a strictly convex function in $\tau$. The same argument works in the subcritical setting as in \cite{NakS:NLS}. Let $u$ be the solution of \eqref{equ:Hart} with intial data $u(0):=\varphi\in B_{2\delta_L}(\mathcal{W})$. Then Proposition \ref{prop:unif lwp} ensures that $u$ exists at least for $|\tau-\tau(0)|\leq 2$ in $B_{\delta_E}(\mathcal{W})$. Using the decomposition \eqref{coor arou W} with $|\tau(0)|:=0$, let 
	\begin{equation}\label{quan:d_1}
	d_1(u(t))^2:=\int_{\R}\phi(\tau)d_{0}^{2}(u(t-\tau))d\tau=\phi*d_{0}^{2}(u)(t),
	\end{equation}
	where $\phi$ is the cut-off function in \eqref{funct:cutf funct}. This defines the function $d_1:B_{2\delta_L}(\mathcal{W}) \rightarrow \left[0,\infty \right)$. Then, we define the nonlinear distance function $\widetilde{d}_{\mathcal{W}}:\dot H^1_{rad} \rightarrow \left[0,\infty \right)$ by 
	\begin{equation}\label{quan:d equi}
	\widetilde{d}_{\mathcal{W}}(\varphi):=\phi_{\delta_L}\left( {d}_{\mathcal{W}}(\varphi) \right)d_1(\varphi)+\phi_{\delta_L}^{C}\left( {d}_{\mathcal{W}}(\varphi) \right)d_{\mathcal{W}}(\varphi).
	\end{equation}

	\subsection{Nonlinear distance function.}
	The following proposition gives the main static properties of the distance function.
	\begin{proposition}(Nonlinear distance function)\label{prop:nlr dist}. The function $\widetilde{d}_{\mathcal{W}}$ on $\dot H^1_{rad}$ is invariant for the rotation and scaling, and equivalent to $d_{\mathcal{W}}$. Precisely, there exists an absolute constant $C\in\left(1,\infty\right)$ such that for all $\varphi\in \dot H^1_{rad}$ and $(\alpha,\beta)\in \R^2$,
		\begin{equation*}
		d_{\mathcal{W}}(\varphi)/C \leq \widetilde{d}_{\mathcal{W}}(\varphi)=\widetilde{d}_{\mathcal{W}}(e^{i\alpha}S_{-1}^{\beta}\varphi) \leq C	d_{\mathcal{W}}(\varphi).
		\end{equation*}
Moreover, there  exists an absolute constant $c_D\in\left(0,1\right)$ such that putting
		\begin{equation*}
		\check{\mathcal{H}}:=\{ \varphi\in \dot H^1_{rad}| E(\varphi)<E(W)+(c_D\widetilde{d}_{\mathcal{W}}(\varphi))^2\},
		\end{equation*}
		we have
		\begin{equation}\label{est:dist eigen}
		\varphi\in B_{\delta_L}(\mathcal{W}) \cap \check{\mathcal{H}} \Longrightarrow \widetilde{d}_{\mathcal{W}}(\varphi) \sim |\lambda_{1}|.
		\end{equation}
	\end{proposition}
	
	\begin{proof}
		First we proof $\widetilde{d}_{\mathcal{W}} \sim d_{\mathcal{W}}$. If $d_{\mathcal{W}}(\varphi)\geq 2\delta_L$, then $\widetilde{d}_{\mathcal{W}}=d_{\mathcal{W}}$ by the defination of $\widetilde{d}_{\mathcal{W}}(\varphi)$. It remains to consider the case $d_{\mathcal{W}}(\varphi)\leq 2\delta_L$. Since $\varphi\in B_{2\delta_L}(\mathcal{W})\subset B_{\delta_E}(\mathcal{W})$, we can  decompose $\varphi$ by Proposition \ref{prop:orth decom} and \ref{prop:spect decom}. By \eqref{est:EK expa} and the scale invariance of $E$, we have
		\begin{equation}\label{est:E expa}
		E(\varphi)-E(W)=E(W+v)-E(W)=\frac{1}{2}\langle \mathcal{L}v,v \rangle-C(v)=-\mu\lambda_+\lambda_-+\frac{1}{2}\langle \mathcal{L}\gamma,\gamma \rangle-C(v).
		\end{equation}
		Hence, using Proposition \ref{prop:orth cntl} and \eqref{quan:d_0},
		\begin{equation*}
		d_0(\varphi)^2=\Vert v \Vert_{E}^2-C(v)=\Vert v \Vert_{E}^2+O(\Vert v \Vert_{\dot H^1}^3)\sim d_{\mathcal{W}}(\varphi)^2.
		\end{equation*}
		Then by  Proposition \ref{prop:unif lwp}, we have $d_1(\varphi)\sim d_{\mathcal{W}}(\varphi)$, and so $\widetilde{d}_{\mathcal{W}}(\varphi)\sim d_{\mathcal{W}}(\varphi)$.
		
		Next we prove \eqref{est:dist eigen}. On the one hand, since $d_{\mathcal{W}}(\varphi)<\delta_{L}$ and $E(\varphi)-E(W)<(c_D\widetilde{d}_{\mathcal{W}}(\varphi))^2$, we have
		\begin{equation*}
		\widetilde{d}_{\mathcal{W}}(\varphi)^2 \sim d_0(\varphi)^2=E(\varphi)-E(W)+2\mu\lambda_{1}^2,
		\end{equation*}
		and so $\widetilde{d}_{\mathcal{W}}(\varphi)^2 \lesssim \lambda_{1}^2$. On the other hand, we see that $\widetilde{d}_{\mathcal{W}}(\varphi)^2\sim \Vert v \Vert_{E}^2 \gtrsim \lambda_{1}^2$ by \eqref{quan:equi nor}.
		
		Finally, we check the invariance for the rotation and scaling. Let $(\alpha,\beta)\in\R^2$, $\varphi\in B_{\delta_E}(\mathcal{W})$ and let $u$ and $u'$ be the solutions of \eqref{equ:Hart} with the initial data
		\begin{equation*}
		u(0)=\varphi,\;\; u'(0)=e^{i\alpha} S_{-1}^{\beta}\varphi,
		\end{equation*}
		with the decompositions by Proposition \ref{prop:orth decom} and the rescaled time functions
		\begin{align*}
		u=&e^{i\theta} S_{-1}^{\sigma}(W+v),\;\;u'=e^{i\theta'} S_{-1}^{\sigma'}(W+v'),\\ \dot\tau=&e^{2\sigma},\;\;\dot\tau'=e^{2\sigma'},\;\;\tau(0)=0=\tau'(0).
		\end{align*}
		Then 
		\begin{equation*}
		e^{i\theta'(0)} S_{-1}^{\sigma'(0)}(W+v'(0))=e^{i\alpha} S_{-1}^{\beta}\varphi=e^{i\alpha}e^{i\theta(0)}S_{-1}^{\beta} S_{-1}^{\sigma(0)}(W+v(0)),
		\end{equation*}
		that is 
		\begin{equation*}
		\theta'(0)=\theta(0)+\alpha,\;\;\sigma'(0)=\sigma(0)+\beta.
		\end{equation*}
		By the uniqueness of $\left(\widetilde{\theta},\widetilde{\sigma} \right)$ in Proposition \ref{prop:orth decom}, we have $\left(\theta',\sigma' \right)=\left(\theta,\sigma\right)+(\alpha,\beta)$.
		Integrating the following identity
		\begin{equation*}
		\dot\tau'=e^{2\sigma'}=e^{2\sigma+2\beta}=e^{2\beta}\dot\tau,
		\end{equation*}
		we have $\tau'=\tau e^{2\beta}$, while the invariance of the equation \eqref{equ:Hart} implies $u'(t)=e^{i\alpha} S_{-1}^{\beta}u(e^{2\beta}t)$. Hence $v$ is invariant in the rescaled time, namely
		\begin{equation*}
		\tau(t)=\tau'(t') \Longrightarrow v(t)=v'(t'),
		\end{equation*}
		which is inherited by $\lambda_{\pm},\lambda_1,\lambda_2$ and $\gamma$. Therefore $d_0$ and $d_1$ are invariant, so is $\widetilde{d}_{\mathcal{W}}$.
	\end{proof}
	Hence we can use $\widetilde{d}_{\mathcal{W}}(\varphi)$ to measure the distance to $\mathcal{W}$, instead of the standard $d_{\mathcal{W}}(\varphi)$. The $\delta$ neighborhood with respect to this distance function is denoted by
	\begin{equation*}
	\widetilde{B}_{\delta}(\mathcal{W}):=\{ \varphi\in \dot H^1_{rad}\;|\;\widetilde{d}_{\mathcal{W}}(\varphi)<\delta \}.
	\end{equation*}

	\subsection{Dynamics in the ejection mode.}The following proposition describes the dynamics close to the ground states in the ejection mode.
	
	In order to verify that: the dynamics can be ruled by that of its unstable eigenmode of the linearized operator around $\mathcal{W}$ just as the dynamics of solutions of linear differential equations. We have to control the orthogonal component $\gamma$ of the spectral decomposition of the remainder resulting from the linearization around $\mathcal{W}$. We would like to control this component by using the quadratic terms resulting from the Taylor expansion of the energy around $\mathcal{W}$. This can be done if and only if the remainder satisfies two orthogonality conditions: see Proposition \ref{prop:orth cntl}. In order to satisfy these conditions, we have to give two degrees of freedom to the decomposition of the solution around $\mathcal{W}$: a rotation parameter (this was done in \cite{NakS:NLS}) and a scaling parameter: see Propositions \ref{prop:orth decom} and \ref{prop:orth decom:u}. Then, we also have to control the evolution of these two parameters (see \eqref{est:dyn epara}). We prove in Proposition \ref{prop:eject mod} that we can close the argument. More precisely, the dynamic of the solution close to $\mathcal{W}$ and in the exit mode is dominated by the exponential growth of the unstable eigenmode. Moreover, a relevant functional $K$ grows exponentially and its sign eventually becomes opposite to that of the eigenmode( see \eqref{est:dyn Keigen}).

	\begin{proposition}(Dynamics in the ejection mode)\label{prop:eject mod}. There is an absolute constant $\delta_{X}\in(0,1)$ such that $\widetilde{B}_{\delta_{X}}(\mathcal{W})\subset B_{\delta_{L}}(\mathcal{W})$, and that for any solution $u$ of \eqref{equ:Hart} with 
		\begin{equation}\label{eje mode}
		u(t_0)\in \widetilde{B}_{\delta_{X}}(\mathcal{W})\cap \check{\mathcal{H}},\; \partial_{t}\widetilde{d}_{\mathcal{W}}(u(t_0))\geq 0,
		\end{equation}	
		at some $t_0\in I(u)$, we have the following.
		
		$\widetilde{d}_{\mathcal{W}}(u(t))$ is increasing strictly until it reaches $\delta_{X}$ at some $t_{X}\in (t_0,T_{+}(u))$. 
		
		For all $t\in \left[ t_0,t_{X} \right]$, we have
		\begin{equation}\label{est:dyn d}
		\widetilde{d}_{\mathcal{W}}(u(t))\sim |\lambda_{1}(t)|\sim e^{\mu(\tau(t)-\tau(t_0) )}\widetilde{d}_{\mathcal{W}}(u(t_0)),
		\end{equation}
		\begin{equation}\label{est:dyn gamm}
		\Vert \gamma(t) \Vert_{\dot H^1} \lesssim \widetilde{d}_{\mathcal{W}}(u(t_0))+\widetilde{d}_{\mathcal{W}}(u(t))^\frac{3}{2},
		\end{equation}
		\begin{equation}\label{est:dyn epara}
		\Big| \left( e^{i\theta(t)}-e^{i\theta(t_0)},\sigma(t)-\sigma(t_0) \right) \Big| \lesssim \widetilde{d}_{\mathcal{W}}(u(t)),
		\end{equation}
		$sign(\lambda_{1}(t))$ is a constant, and there exists an absolute constant $C_K>0$ such that 
		\begin{equation}\label{est:dyn Keigen}
		-sign(\lambda_{1}(t))K(u(t))\gtrsim \left( e^{\mu(\tau(t)-\tau(t_0) )}-C_K \right)\widetilde{d}_{\mathcal{W}}(u(t_0))
		\end{equation}
	\end{proposition}
	\begin{proof}
		Let $u$ be a solution in the ejection mode \eqref{eje mode} at $t=t_0 \in I(u)$. Choosing $\delta_{X}$ small enough ensures that $ \widetilde{B}_{\delta_{X}}(\mathcal{W})\subset  B_{\delta_{L}}(\mathcal{W})$. Then there are only two cases:
		\begin{enumerate}
			\item[$(1)$] There exists a minimal $\tau_{X}>\tau_{0}$, such that $\widetilde{d}_{\mathcal{W}}(u)=\delta_{X}$ at $\tau=\tau_{X}$. 
			\item[$(2)$] For all $\tau\in(\tau_0,\infty)$, we have $\widetilde{d}_{\mathcal{W}}(u)<\delta_{X}$.
		\end{enumerate}
		Let $\tau_{X}:=\infty$ in the second case, then in the both cases, we have $\widetilde{d}_{\mathcal{W}}(u)< \delta_{X}$ for $\tau\in(\tau_0,\tau_X)$. Then $\widetilde{d}_{\mathcal{W}}(u)=d_1(u)\sim |\lambda_{1}|$ on $\tau\in(\tau_0,\tau_X)$ by \eqref{quan:d equi}. Hence, using \eqref{quan:d_0}, \eqref{quan:d_1} and the equations \eqref{est:dyn lam1} of $\lambda_{j}$,
		\begin{equation*}
		\begin{aligned}
		\partial_{\tau}\widetilde{d}_{\mathcal{W}}(u)^2=&\partial_{\tau}d_1(u)^2=\phi * \partial_{\tau}(d_0(u)^2)=\phi * \partial_{\tau}(2\mu\lambda_{1}^2)=\phi * [4\mu^2 \lambda_{1}\lambda_{2}+O(\Vert v \Vert_{E}^3)],\\ \partial_{\tau}^2\widetilde{d}_{\mathcal{W}}(u)^2=&\phi *[4\mu^3(\lambda_{1}^2+\lambda_{2}^2)+O(\Vert v \Vert_{E}^3)]+\phi'*O(\Vert v \Vert_{E}^3)\sim \lambda_{1}^2\sim \widetilde{d}_{\mathcal{W}}(u)^2,
		\end{aligned}
		\end{equation*}
		where we also used Proposition \ref{prop:unif lwp} to remove the convolution in the last step. Since $\partial_{\tau}\widetilde{d}_{\mathcal{W}}(u(t_0))\geq0$, the last estimate implies that $\widetilde{d}_{\mathcal{W}}(u)$ is strictly increasing for $\tau\in(\tau_0,\tau_X)$. It also implies exponential growth in $\tau$ of $\widetilde{d}_{\mathcal{W}}$, so it is impossible to have the the case $\tau_X=\infty$ above. In other words, there exists $T_X<T_+(u)$ such that $\widetilde{d}_{\mathcal{W}}(u)$ reaches $\delta_{X}$ at $\tau=\tau_X=\tau(T_X)$.
		
		Since $\widetilde{d}_{\mathcal{W}}\sim |\lambda_{1}|$ is positive continuous on $(\tau_0,\tau_X)$, $\lambda_{1}(\tau)$ cannot change the sign. Let $\mathfrak{s}:=sign\lambda_{1}(\tau)\in\{\pm\}$ be its sign.
		
		Next we show the more precise exponential behavior. Since $\partial_{\tau}\widetilde{d}_{\mathcal{W}}(u)^2 \geq0$ at $\tau=\tau_0$, there exists $\tau\in(\tau_0-2,\tau_0+2)$ where $\partial_{\tau}\lambda_{1}^2 \geq0$. Combined with $\Vert v \Vert_{\dot H^1}\sim |\lambda_{1}|$ and \eqref{est:dyn lam1}, we have $\lambda_{1}\lambda_{2}\gtrsim -|\lambda_{1}|^3$. Since $\partial_{\tau}(\lambda_{1}\lambda_{2})\sim \lambda_{1}^2\sim \lambda_{1}(\tau_{0})^2$ for $|\tau-\tau_{0}|<2$, there exists $\tau_{1}\in(\tau_{0},\tau_{0}+2)$ such that $\lambda_{1}(\tau_1)\lambda_{2}(\tau_1)\gtrsim -|\lambda_{1}(\tau_1)|^3$, or equivalently $\mathfrak{s}\lambda_{2}(\tau_1)\gtrsim-\lambda_{1}(\tau_{1})^2$. Then $\mathfrak{s}\lambda_+(\tau_{1})\geq |\lambda_{1}(\tau_{1})|/2$ and $\mathfrak{s}\lambda_-(\tau_{1})\geq 0$. Let $\widetilde{R}:=|\lambda_{1}(\tau_1)|$, and suppose that for some $\tau_{2}\in(\tau_{1},\tau_{X})$
		\begin{equation*}
		\tau_{1}<\tau<\tau_{2} \Longrightarrow\;\;|\lambda_{1}(\tau)|\leq 2\widetilde{R}e^{\mu(\tau-\tau_{1})} \lesssim \delta_{X}.
		\end{equation*}
		Then the equations \eqref{est:dyn lam+} of $\lambda_{\pm}$ together with $\Vert v \Vert_{\dot H^1}\sim |\lambda_{1}|$ imply for $\tau\in(\tau_1,\tau_2)$,
		\begin{equation}\label{equ:dyn elam+}
		\lambda_{+}(\tau)e^{-\mu\tau}-\lambda_{+}(\tau_1)e^{-\mu\tau_{1}}=C\int_{\tau_1}^{\tau}e^{-\mu s}\lambda_1^2(s)ds,
		\end{equation}
		which implies
		\begin{equation*}
		\begin{aligned}
		|\lambda_{+}(\tau)-e^{\mu(\tau-\tau_{1})} \lambda_{+}(\tau_{1})|=&Ce^{\mu\tau} \int_{\tau_1}^{\tau} e^{-\mu s}\lambda_1^2(s)ds\\\lesssim&e^{\mu\tau-2\mu\tau_{1}}\int_{\tau_1}^{\tau} e^{\mu s}\widetilde{R}^2ds \\\lesssim&\widetilde{R}^2 e^{2\mu(\tau-\tau_{1})} \lesssim \delta_{X}\widetilde{R}e^{\mu(\tau-\tau_{1})}.
		\end{aligned}
		\end{equation*}
		Similarly, we have
		\begin{equation*}
		|\lambda_{-}(\tau)-e^{-\mu(\tau-\tau_{1})} \lambda_{-}(\tau_{1})| \lesssim \widetilde{R}^2e^{2\mu(\tau-\tau_{1})} \lesssim \delta_{X}\widetilde{R}e^{\mu(\tau-\tau_{1})}.
		\end{equation*}
		On the one hand, we have 
		\begin{equation*}
		|\lambda_{1}(\tau)|\sim|\lambda_{+}(\tau)|\lesssim \delta_{X}\widetilde{R}e^{\mu(\tau-\tau_{1})}+|\lambda_{+}(\tau_1)|e^{\mu(\tau-\tau_{1})}\\\lesssim\widetilde{R}e^{\mu(\tau-\tau_{1})}(1+\delta_{X}).
		\end{equation*}
		On the other hand, since
		\begin{equation*}
		|\lambda_{+}(\tau)|\geq|\lambda_{+}(\tau_1)|e^{\mu(\tau-\tau_{1})}-C\delta_{X}\widetilde{R}e^{\mu(\tau-\tau_{1})},\;\;|\lambda_{-}(\tau)|\leq e^{\mu(\tau-\tau_{1})}\left(|\lambda_{-}(\tau_1)|+C\widetilde{R}\delta_{X}\right),
		\end{equation*}
		thus we have 
		\begin{equation*}
		|\lambda_{1}(\tau)|\gtrsim|\lambda_{+}(\tau)|-|\lambda_{-}(\tau)|\gtrsim e^{\mu(\tau-\tau_{1})}\left(|\lambda_{+}(\tau_{1})|-C\widetilde{R}\delta_{X}-|\lambda_{-}(\tau_{1})|-C\widetilde{R}\delta_{X}\right)\gtrsim \widetilde{R}e^{\mu(\tau-\tau_{1})}.
		\end{equation*}
		Hence the continuity in $\tau$ allows us to take $\tau_{2}=\tau_{X}$. Moreover the above estimates together with $|\lambda_{1}|\sim \widetilde{R}$ on $(\tau_{0},\tau_{1})$ implies that, with $R:=\widetilde{d}_{\mathcal{W}}(u(\tau_{0}))$,
		\begin{equation*}
		\tau_{0} \leq \tau \leq \tau_{X} \Longrightarrow\;\;\widetilde{d}_{\mathcal{W}} \sim \mathfrak{s}\lambda_{1}\sim Re^{\mu(\tau-\tau_{0})}.
		\end{equation*}
		
		In order to estimate $\gamma$, consider the expansion of the energy \eqref{est:E expa} without the $\gamma$ terms. We denote this expansion by $E_{\gamma^{\bot}}$:
		\begin{equation*}
		E_{\gamma^{\perp}}(u):=-\mu\lambda_+\lambda_{-}-C(\lambda_+g_{+}+\lambda_{-}g_{-}).
		\end{equation*}
		Notice that $C'(f)=N(f)$. By this observation, \eqref{equ:dyn lam+} and \eqref{equ:dyn lam-} we see that
		\begin{equation*}
		\begin{aligned}
		\partial_{\tau}E_{\gamma^{\bot}}(u)=&\langle N(v)-N(\lambda_+g_{+}+\lambda_{-}g_{-}) ,g_{+}\rangle \partial_{\tau}\lambda_+\\& +\langle N(v)-N(\lambda_+g_{+}+\lambda_{-}g_{-}) ,g_{-}\rangle \partial_{\tau}\lambda_-\\& + \theta_{\tau}\left( \langle v,g_{-} \rangle \partial_{\tau}\lambda_- + \langle v,g_{+} \rangle \partial_{\tau}\lambda_+ \right)\\&+ \sigma_{\tau}\left( \langle iv,S_{1}'g_{-} \rangle \partial_{\tau}\lambda_- + \langle iv,S_{1}'g_{+} \rangle \partial_{\tau}\lambda_+ \right).
		\end{aligned}
		\end{equation*}
		Since 
		\begin{equation*}
		\langle N(v)-N(\lambda_+g_{+}+\lambda_{-}g_{-}) ,g_\pm\rangle\lesssim \lambda_{1}\Vert \gamma \Vert_{\dot H^1},\;\;\langle v,g_\pm \rangle=\langle iv,S_{1}'g_{\pm} \rangle=O(\Vert v \Vert_{\dot H^1}),
		\end{equation*}
		this together with $|\lambda_{1}|\sim \Vert v \Vert_{\dot H^1}$, \eqref{est:dyn lam+} and \eqref{est:para} implies that, for $\tau\in \left[ \tau_{0},\tau_{X} \right]$,
		\begin{equation}\label{est:parl uW minu}
		\begin{aligned}
		|\partial_{\tau}\left( E(u)-E(W)-E_{\gamma^{\bot}}(u) \right)|=&|\partial_{\tau}E_{\gamma^{\bot}}(u)|\\\lesssim& \lambda_{1}\Vert \gamma \Vert_{\dot H^1}\left(\partial_{\tau}\lambda_{+}+\partial_{\tau}\lambda_{-}\right) + \Vert v \Vert_{\dot H^1}^2\left(\partial_{\tau}\lambda_{+}+\partial_{\tau}\lambda_{-}\right)\\\lesssim&\lambda_{1}^{2}	\Vert \gamma \Vert_{\dot H^1} +\lambda_{1}^{3}.
		\end{aligned}
		\end{equation}
		Moreover, using Sobolev's inequality and Hölder's inequality, we see that
		\begin{equation*}
		|C(v)-C(\lambda_+g_{+}+\lambda_{-}g_{-})|\lesssim |\lambda_{1}|^{2}	\Vert \gamma \Vert_{\dot H^1}.
		\end{equation*}
		By \eqref{est:EK expa} and  Proposition \ref{prop:orth cntl}, 
		\begin{equation}\label{est:uW minu}
		\begin{aligned}
		E(u)-E(W)-E_{\gamma^{\bot}}(u)=&\frac{1}{2}\langle \mathcal{L}\gamma,\gamma \rangle+C(\lambda_+g_{+}+\lambda_{-}g_{-})-C(v)\\ \sim&\Vert \gamma \Vert_{\dot H^1}^{2}+O(|\lambda_{1}|^{2}	\Vert \gamma \Vert_{\dot H^1}).
		\end{aligned}
		\end{equation}		
		Combined with \eqref{est:dyn d} and \eqref{est:parl uW minu}, we have 
		\begin{equation*}
		\begin{aligned}
		\int_{\tau_0}^{\tau} O(|\lambda_{1}|^{2}\Vert \gamma \Vert_{\dot H^1}+\lambda_{1}
		^3)=&E(u)-E(W)-E_{\gamma^{\bot}}(u)-\left(	E(u(\tau_0))-E(W)-E_{\gamma^{\bot}}(u(\tau_0))\right)\\\sim& \Vert \gamma \Vert_{\dot H^1}^{2}+O(|\lambda_{1}|^{2}	\Vert \gamma \Vert_{\dot H^1})-\Vert \gamma(\tau_0) \Vert_{\dot H^1}^{2}-O(|\lambda_{1}(\tau_0)|^{2}	\Vert \gamma(\tau_0 )\Vert_{\dot H^1}).
		\end{aligned}
		\end{equation*}
		By continuity method, we have
		\begin{equation*}
		\begin{aligned}
		\Vert \gamma \Vert_{\dot H^1}^{2}\lesssim&\lambda_{1}^3+\Vert \gamma(\tau_0) \Vert_{\dot H^1}^{2}-O(|\lambda_{1}|^2\Vert \gamma \Vert_{\dot H^1}) + O(|\lambda_{1}(\tau_0)|^2\Vert \gamma(\tau_0) \Vert_{\dot H^1})\\\lesssim&\lambda_{1}^3+|\lambda_{1}(\tau_0)|^2+|\lambda_{1}(\tau_0)|^3\\\lesssim&\lambda_{1}^3+|\lambda_{1}(\tau_0)|^2.
		\end{aligned}
		\end{equation*}
		Thus we have
		\begin{equation*}
		\Vert \gamma \Vert_{\dot H^1}\lesssim\lambda_{1}^{3/2}+|\lambda_{1}(\tau_0)|\sim\widetilde{d}_{\mathcal{W}}(u(\tau))^{3/2}+ \widetilde{d}_{\mathcal{W}}(u(\tau_0)),
		\end{equation*}
		this proves \eqref{est:dyn gamm}.
		
		Using \eqref{est:para} and $\widetilde{d}_{\mathcal{W}}(u)\sim \Vert v\Vert_{\dot H^1}$, we obtain $\pm\theta_\tau=O(\widetilde{d}_{\mathcal{W}}(u))$. Integrating in $\tau$, we have
		\begin{equation*}
		|\theta(\tau)-\theta(\tau_0)|\sim \int_{\tau_0}^{\tau}Re^{\mu(\tau-\tau_0)}\lesssim Re^{\mu(\tau-\tau_0)}\sim \widetilde{d}_{\mathcal{W}}(u).
		\end{equation*}
		Similarly, we get $|\sigma(\tau)-\sigma(\tau_0)|\lesssim \widetilde{d}_{\mathcal{W}}(u)$. Thus we obtain \eqref{est:dyn epara}.
		
		Next we show \eqref{est:dyn Keigen}. Using $L_{+}W=2\Delta W$, \eqref{est:EK expa} and Proposition \ref{prop:spect decom}, we have 
		\begin{equation}\label{est:K expa}
		\begin{aligned}
		K(W+v)=&-2\langle \nabla W,\nabla v \rangle+O(\Vert v\Vert_{\dot H^1}^2)\\=&-2\mu\lambda_{1} \langle W,g_2 \rangle+\langle 2\Delta W,\gamma \rangle+O(\Vert v\Vert_{\dot H^1}^2).
		\end{aligned}
		\end{equation}
		This together with Proposition \ref{prop:orth decom}, $\langle W,g_2 \rangle>0$ (see \eqref{con:ome}), $|\langle \Delta W,\gamma \rangle|\lesssim \Vert \gamma\Vert_{\dot H^1}$ as well as the above estimates on $\lambda_{1},\gamma$, 
		\begin{equation*}
		\begin{aligned}
		-sign(\lambda_{1}(t)) K(u(t))=&-sign(\lambda_{1}(t)) K(W+v)\\=&-sign(\lambda_{1}(t)) \left( -2\mu\lambda_{1} \langle W,g_2 \rangle+\langle 2\Delta W,\gamma \rangle+O(\Vert v\Vert_{\dot H^1}^2) \right)\\\gtrsim&|\lambda_{1}|-\Vert \gamma\Vert_{\dot H^1}+O(\lambda_{1}^2)\\\gtrsim&|\lambda_{1}|-|\lambda_{1}(\tau_0)|\\\gtrsim& \left( e^{\mu(\tau-\tau_0 )}-C_K \right)\widetilde{d}_{\mathcal{W}}(u(\tau_0)).
		\end{aligned}
		\end{equation*}
		Thus \eqref{est:dyn Keigen} is proved.
	\end{proof}	
	\begin{remark}
		By time-reversal symmetry, a similar result holds in the negative time direction, where the last condition of \eqref{eje mode} is replaced with $\partial_{t}\widetilde{d}_{\mathcal{W}}(u(t))\leq 0$.
	\end{remark}

	The following proposition gives a variational estimate away from the ground states.
	
	\begin{proposition}(Variational estimates)\label{prop:vart est}. There exist two increasing function $\epsilon_V$ and $\kappa$ from $(0,\infty)$ to $(0,1)$, and an absolute constant $c_{V}>0$, such that for any $\varphi\in \dot H^1_{rad}$ satisfying $E(\varphi)<E(W)+\epsilon_V(\widetilde{d}_{\mathcal{W}}(\varphi))^2$, we have
		\begin{equation*}
		K(\varphi)\geq \min \left( \kappa(\widetilde{d}_{\mathcal{W}}(\varphi)),c_{V}\Vert \varphi \Vert_{\dot H^1}^2 \right) \;\text{or} \;\; K(\varphi)\leq -\kappa(\widetilde{d}_{\mathcal{W}}(\varphi)).
		\end{equation*}
	\end{proposition}
	\begin{proof}[Sketch of proof]	
		Indeed, we can give a more general conclusion. If $E(\varphi)<E(W)+\epsilon_{V}(\delta)^2$, where $\delta\leq\widetilde{d}_{\mathcal{W}}(\varphi)$, then
		\begin{equation*}
		K(\varphi)\geq \min \left( \kappa(\delta),c_{V}\Vert\varphi\Vert_{\dot H^1}^2 \right) \;\text{or} \;\; K(\varphi)\leq -\kappa(\delta).		
		\end{equation*}	
		This is proved in the same way as \cite[Lemmma 4.3]{NakS:NLKG}.
	\end{proof}

	\subsection{Sign functional.}The following proposition defines a functional $\Theta$ that decides the fate of the solution around $t= T_{\pm}(u)$, as well as at the exit time $t=t_X$ in the above proposition. The continuity of $\Theta$ implies that for any solution $u$ in $\mathcal{H}^{\epsilon}\subset \mathcal{H}^{\epsilon_{S}}$, $\Theta(u)\in\{\pm1 \}$ can change along $t\in I(u)$ only if $u$ goes through the small neighborhood $\mathcal{H}^{\epsilon}\setminus \check{\mathcal{H}} \subset \widetilde{B}_{\epsilon/{c_D}}(\mathcal{W})$. 
	\begin{proposition}(Sign functional)\label{prop:sign funct}. There exist an absolute constant $\epsilon_S \in(0,1)$ and a continuous function $\Theta:\mathcal{H}^{\epsilon_S} \cap \check{\mathcal{H}} \rightarrow \{\pm1 \}$, such that for some $0<\delta_1<\delta_2<\delta_X$ and for any $\varphi \in \mathcal{H}^{\epsilon_S} \cap \check{\mathcal{H}} $, with the convention $sign 0=+1$,
		\begin{equation}\label{quan:sign func}
		\left\{\begin{aligned}
		&\widetilde{d}_{\mathcal{W}}(\varphi)\geq \delta_1 \Longrightarrow \Theta(\varphi)=signK(\varphi) ,\\
		&\widetilde{d}_{\mathcal{W}}(\varphi)\leq \delta_2 \Longrightarrow \Theta(\varphi)=-sign \lambda_{1} ,
		\end{aligned}\right.
		\end{equation}	
		with the following properties.
		\begin{enumerate}
			\item[$(1)$] For all $(\alpha,\beta)\in \R^2$, $\Theta(e^{i\alpha}S_{-1}^{\beta}\varphi)=\Theta(\varphi)$. 
			\item[$(2)$] Moreover, if $E(\varphi)<E(W)$, then $\Theta(\varphi)=signK(\varphi)$.
		\end{enumerate}
	\end{proposition}
	\begin{proof}
		First, we prove that $\Theta$ is well-defined and continous on $\mathcal{H}^{\epsilon_S} \cap \check{\mathcal{H}}$. We have the following assertions. On the one hand, Proposition \ref{prop:vart est} implies that: if $\epsilon\leq \epsilon_{V}(\delta)$ and $\varphi\in \mathcal{H}^{\epsilon}\setminus \widetilde{B}_{\delta}(\mathcal{W})$, then $signK(\varphi)$ is a constant. On the other hand, Proposition \ref{prop:nlr dist} implies that: if $\varphi\in B_{\delta_L}(\mathcal{W})\cap \check{\mathcal{H}}$, then $sign \lambda_{1}$ is a constant.
		
		Hence, after fixing $\delta_1,\delta_2,\epsilon_{S}$ such that $0<\delta_1<\delta_2\ll \delta_{X}$ and $\epsilon_S\leq \epsilon_{V}(\delta_1)$, the functional $\Theta$ is well-defined  and continuous on $\check{\mathcal{H}}^{\epsilon_S}:=\mathcal{H}^{\epsilon_S}\cap \check{\mathcal{H}}$ by \eqref{quan:sign func}, once we prove that $-sign \lambda_{1}=signK$ on 
		\begin{equation*}
		Y:=\{ \varphi\in \check{\mathcal{H}}^{\epsilon_S}|\;\delta_1\leq \widetilde{d}_{\mathcal{W}}(\varphi)\leq \delta_2 \}.
		\end{equation*}
		To this end, take any solution $u$ of \eqref{equ:Hart} with initial data $u(0)\in Y$. By applying Proposition \ref{prop:eject mod} either forward or backward in time, there exists $t_X\in I(u)$ such that $\widetilde{d}_{\mathcal{W}}(u(t_X))=\delta_X$ and $\widetilde{d}_{\mathcal{W}}(u(t))\in [\widetilde{d}_{\mathcal{W}}(u(0)),\delta_{X}] \subset [\delta_1,\delta_X]$ when $t\in[0,t_X]$. Thus $u(t)\in \check{\mathcal{H}}^{\epsilon_S} \cap B_{\delta_{L}}(\mathcal{W}) \setminus \widetilde{B}_{\delta_1}(\mathcal{W})$ when $t\in[0,t_X]$. Hence $sign\lambda_{1}(u(t))$ and $signK(u(t))$ are unchanged when $t\in[0,t_X]$ using the initial assertion. Combined \eqref{est:dyn Keigen} with $\delta_2\ll \delta_{X}$, we have
		\begin{equation*}
		\begin{aligned}
		-sign(\lambda_{1}(t_X))K(u(t_X))\gtrsim& \left( e^{\mu(\tau(t_X)-\tau(0) )}-C_K \right)\widetilde{d}_{\mathcal{W}}(u(0))\\\gtrsim& \widetilde{d}_{\mathcal{W}}(u(t_X))-C_K\widetilde{d}_{\mathcal{W}}(u(0))\geq\delta_{X}-\delta_2C_K>0.
		\end{aligned}
		\end{equation*}
		Thus $-sign\lambda_{1}(u(t_X))=signK(u(t_X))$. Consequently, $-sign\lambda_{1}=signK$ on $Y$.
		
		Next, we proved that $\Theta=signK$ on $\mathcal{H}^0$. Since $0\in \mathcal{H}^0$, we have $\Theta(0)=signK(0)=+1$. It suffices to consider: if $E(u)<E(W)$ and $u\neq0$, then $\Theta(u)=signK(u)$. By \eqref{char:Neh char}, we have $K(u)\neq0$. By 
		\begin{equation*}
		\lambda\partial_\lambda E(\lambda u)=\lambda^2\Vert u\Vert_{\dot H^1}^2-\lambda^4\iint \frac{\big|u(t,x)\big|^2\big|u(t,y)\big|^2}{|x-y|^4} dxdy=K(\lambda u),
		\end{equation*}
		there is a unique $\lambda_0>0$ such that for $0<\lambda_{-}<\lambda_0<\lambda_+<\infty$
		\begin{equation*}
		K(\lambda_{-}u)> 0=K(\lambda_0u) > K(\lambda_{+}u).
		\end{equation*}
		
		If $K(u)>0$, then $\partial_\lambda E(\lambda u)\geq \lambda\Vert u \Vert_{\dot H^1}^2-\lambda\Vert \left( |x|^{-4}*|u|^2\right)|u|^2\Vert_{L^1}\geq0$ if $0\leq\lambda\leq1$. Thus we have $E(\lambda u)\leq E(u)<E(W)$, which implies that $\{\lambda u \}_{0\leq\lambda\leq1}$ is a $C^0$ curve in $\mathcal{H}^0\subset \check{\mathcal{H}}$ connecting $u$ and 0. Hence by continuity $\Theta(u)=+1=signK(u)$.
		
		If $K(u)<0$, then $\partial_\lambda E(\lambda u)<\lambda\Vert u \Vert_{\dot H^1}^2-\lambda^3\Vert u \Vert_{\dot H^1}^2\leq0$ if $1\leq\lambda<\infty$. Thus we have $E(\lambda u)\leq E(u)<E(W)$, which implies that $\{\lambda u \}_{1\leq\lambda<\infty}$ is a $C^0$ curve in $\mathcal{H}^0\subset \check{\mathcal{H}}$ connecting $u$ with the region $E(u)<0$, where $\mathcal{W}$ is far and so $\Theta=signK=-1$. Hence by continuity $\Theta(u)=-1=signK(u)$.
		
		Finally, the invariance of $\Theta$ for the rotation and scaling follows from that of $K$ and $\lambda_{1}$. The latter being proved in the proof of Proposition \ref{prop:nlr dist}, and the former is obviously easy to prove.
	\end{proof}
	Note that the region $\{ \varphi \in \mathcal{H}^{\epsilon_S} \cap \check{\mathcal{H}}|\;\Theta(\varphi)=+1 \}$ is bounded in $\dot H^1$, because 
	\begin{equation}\label{unif bd Te+}
	\left\{\begin{aligned}
	&K(\varphi)\geq 0 \Longrightarrow \Vert \varphi\Vert_{\dot H^1}^2\leq4E(\varphi)\leq4(E(W)+\epsilon_{S}^2) ,\\
	&\widetilde{d}_{\mathcal{W}}(\varphi)\leq \delta_X \Longrightarrow \Vert \varphi\Vert_{\dot H^1}\leq \Vert W\Vert_{\dot H^1}+C\delta_{X} ,\end{aligned}\right.
	\end{equation}
	but it does not imply global existence for solutions staying in this region, because of the critical nature of \eqref{equ:Hart}.

	\subsection{One-pass lemma.} The one-pass lemma implies that every solution can enter and exit a small neighborhood of the ground states $W$ at most one time, which is the most crucial step in proving Theorem \ref{thm:above thresh}. It precludes “almost homoclinic orbits”, i.e., those solutions starting in, the moving away from, and finally returning to, a small neighborhood of the ground states. Indeed, this part of analysis is the main part of this paper.

	\begin{proposition}(One-pass)\label{prop:one pass}. There exist an absolute constant $\delta_{B}\in(0,\delta_{X})$ and an increasing function $\epsilon_{B}:\left( 0,\delta_{B} \right] \rightarrow \left(0,\epsilon_{B} \right]$ satisfying $\epsilon_{B}(\delta)<c_D(\delta)$ for $\delta\in \left(0, \delta_{B} \right]$, and for any solution $u$ of \eqref{equ:Hart} with $u(t_0)\in \mathcal{H}^{\epsilon_{B}(\delta)}\cap \widetilde{B}_{\delta}(\mathcal{W})$ at some $t_0\in I(u)$,
		\begin{equation}\label{one-pass}
		\exists t_{+}(\delta)\in\left( t_0,T_{+}(u) \right], s.t. \left\{\begin{aligned}
		&t_0\leq t < t_{+}(\delta) \Longrightarrow \widetilde{d}_{\mathcal{W}}(u(t))<\delta ,\\
		&t_{+}(\delta)<t<T_{+}(u) \Longrightarrow \widetilde{d}_{\mathcal{W}}(u(t))>\delta.\end{aligned}\right.
		\end{equation}
	\end{proposition} 
	In the case where $\Theta(u(t))=-1$ after ejection, we introduce a threshold $R_{V}^{-}$, which allows us to compare $K$ with the main part of the virial identity. After integrating the virial identity, we can prove that this threshold must be very large; By proving a decay estimate, we can show that this threshold is not so large. This leads to a contradiction.
	
	In the case where $\Theta(u(t))=+1$ after ejection, we introduce a radius of the concentration of the kinetic part of the energy (see definition of $R_{V}^{+}$) and the hyperbolic parameter (see definition of $R_H$). We can estimate $K$ and some error terms (generated by the cut-off) in the hyperbolic region and the variational region. Then we compare these estimates. In the worst case, we prove a decay estimate (see \eqref{est:IV decay}) in the variational region and use this estimate to achieve Bourgain's energy induction method \cite{Bourg}, which can reduce the problem to energy below the ground state. Then the theory below the ground state energy (see \cite{MXZ:crit Hart:f rad}) implies that it is neither close to the ground states nor to the original solution by a perturbation argument, which contradicts the assumption of returning orbit.
	
	This proposition will be proved in Section \ref{sect:one pass}.
	
	\begin{remark}
		By time-reversal symmetry, there  also exists $t_{-}\in \left[-T_{-}(u),t_0 \right)$ such that $\widetilde{d}_{\mathcal{W}}(u(t))<\delta$ for $t_{-}<t<t_0$ and $\widetilde{d}_{\mathcal{W}}(u(t))>\delta$ for $-T_{-}(u)<t<t_{-}$.
	\end{remark}
	\begin{remark}\label{rmk:incln relat}
		$\epsilon_{B}(\delta)\leq \epsilon_{S}$ and $\epsilon_{B}(\delta)<c_D(\delta)$ imply that $\mathcal{H}^{\epsilon_{B}(\delta)}\setminus \widetilde{B}_{\delta}(\mathcal{W})\subset \mathcal{H}^{\epsilon_S} \cap \check{\mathcal{H}}$.
	\end{remark}
	The above proposition tells that if a solution gets out of $\widetilde{B}_{\delta}(\mathcal{W})$, then it can never return there. Moreover, it applies to all $\delta \in \left(0,\delta_{B} \right]$ satisfying $E(u)<E(W)+\epsilon_{B}(\delta)^2$. The solution $u$ stays around $\mathcal{W}$ iff $t_{+}(\delta)=T_{+}(u)$.

	\subsection{Solutions staying around the ground states.} The following proposition gives more precise description of such solutions.
	\begin{proposition}\label{prop:ps des} Under the assumption of Proposition \ref{prop:one pass}, suppose that $t_{+}(\delta)=T_{+}(u)$.
		\begin{enumerate}
			\item[$(1)$] Then there exists $t_{1}\in \left[ t_0,T_{+}(u) \right]$ such that $\widetilde{d}_{\mathcal{W}}(u(t))$ is decreasing on $\left[t_0,t_1\right)$, $u(t)\notin \check{\mathcal{H}}$ for all $t\in \left[ t_1,T_{+}(u) \right)$.
			\item[$(2)$] Moreover, if $t_1=T_{+}(u)$, then $\widetilde{d}_{\mathcal{W}}(u(t))\searrow0$ as $t\nearrow T_{+}(u)$, which implies $E(u)=E(W)$. 
		\end{enumerate}	
		We have similar statements in the case $t_{-}(\delta)=T_{-}(u)$ by the time-reversal symmetry.
	\end{proposition}
	\begin{proof}
		Let $u$ be a solution of \eqref{equ:Hart} with $u(t_0)\in \mathcal{H}^{\epsilon_B(\delta)} \cap \widetilde{B}_{\delta}(\mathcal{W})$, where $\delta\in (0,\delta_{B}], t_0\in I(u)$ and $t_+(\delta)=T_+(u)$. By Proposition \ref{prop:one pass}, we have $\widetilde{d}_{\mathcal{W}}(u(t))<\delta$ for $t\in[t_0,T_+(u))$.
		
		Next, we make the following assertion: for any $t\in[t_0,T_+(u))$, $\partial_t \widetilde{d}_{\mathcal{W}}(u(t))<0$ or $u(t)\notin \check{\mathcal{H}}$. We prove it by contradiction. Indeed, there exists $\bar{t}\in[t_0,T_+(u))$, such that $u(\bar{t})\in \check{\mathcal{H}}$ and $\partial_t \widetilde{d}_{\mathcal{W}}(u(\bar{t}))\geq0$, then Proposition \ref{prop:eject mod} implies that $\widetilde{d}_{\mathcal{W}}(u(t))$ increases strictly until it reaches $\delta_{X}$ at some $t_{X}\in (\bar{t},T_{+}(u))$, which leads to a contradiction.
		
		So by the mean value theorem, there are only two possibilities:
		\begin{enumerate}
			\item[$(1)$] There exists $t_1\in [t_0,T_+(u))$ such that $u(t)\notin \check{\mathcal{H}}$ for all $t\in[t_1,T_+(u))$. 
			\item[$(2)$] For all $t\in[t_0,T_+(u))$, $u(t)\in \check{\mathcal{H}}$ and $\partial_t \widetilde{d}_{\mathcal{W}}(u(t))<0$.
		\end{enumerate}
		In the first case, if we choose the minimal $t_1$, then for $t_0\leq t<t_1$, we have $u(t)\in \check{\mathcal{H}}$, thus $\partial_t \widetilde{d}_{\mathcal{W}}(u(t))<0$ by the assertion. Thus, the above two cases can be summed up as: there exists $t_1\in [t_0, T_+(u)]$, such that $u(t)\notin \check{\mathcal{H}}$ on $[t_1, T_+(u))$, $\widetilde{d}_{\mathcal{W}}(u(t))$ is decreasing  and $u(t)\in \check{\mathcal{H}}$ on $[t_0,t_1)$. This proves the first result.
		
		Since $\widetilde{B}_{\delta}(\mathcal{W})\subset \widetilde{B}_{\delta_X}(\mathcal{W})\subset B_{\delta_L}(\mathcal{W})$, Proposition \ref{prop:unif lwp} implies that $\tau\rightarrow\infty$ as $t\nearrow T_+(u)$. If $t_1=T_{+}(u)$, then $\partial_t\widetilde{d}_{\mathcal{W}}(u(t))<0$ and $u(t)\in \check{\mathcal{H}}$ on $[t_0, T_+(u))$. Using $\widetilde{d}_{\mathcal{W}}(u(t_0))<\delta$ and Proposition \ref{prop:eject mod} backward in time, we have
		\begin{equation*}
		\widetilde{d}_{\mathcal{W}}(u(t_0)) \sim e^{\mu(\tau(t)-\tau(t_0))}\widetilde{d}_{\mathcal{W}}(u(t)),
		\end{equation*}
		for any $t\in(t_0,T_+(u))$, corresponding to $\tau\in(\tau(t_0),\infty)$. Then $\widetilde{d}_{\mathcal{W}}(u(t))\searrow0$ for $t\nearrow T_+(u)$. For $u(t)\in \check{\mathcal{H}}$, we have 
		\begin{equation*}
		0\leq \sqrt{E(u)-E(W)}< c_D\widetilde{d}_{\mathcal{W}}(u(t)).
		\end{equation*}
		Combined with $\widetilde{d}_{\mathcal{W}}(u(t))\rightarrow0$, we have $E(u)=E(W)$.
	\end{proof}

	\subsection{Long-time behavior away from the ground states.} The following proposition describes the asymptotic behavior of solutions which are away from the ground states.
	
	\begin{proposition}(Asymptotic behavior)\label{prop:asyp behar}. There exists an increasing function $\epsilon_{*}:\left(0,\delta_{B}\right]\rightarrow \left(0,\epsilon_{S}\right]$ with $\epsilon_{*}(\delta)< \epsilon_{B}(\delta)$ for $\delta\in (0,\delta_B)$. Suppose that $u$ is a solution of \eqref{equ:Hart} satisfying $u\left( \left[ t_0,T_{+}(u) \right) \right)\subset \mathcal{H}^{\epsilon_{*}(\delta)}\setminus \widetilde{B}_{\delta}(\mathcal{W})$ for some $t_0\in I(u)$. 
		\begin{enumerate}
			\item[$(1)$] If $\Theta(u(t))=+1$ at some $t\in \left[ t_0,T_{+}(u) \right)$, then $T_{+}(u)=\infty$, $u$ scatters as $t\rightarrow\infty$.  
			\item[$(2)$]If $\Theta(u(t))=-1$ and $u(t)\in L_{x}^{2}$ at some $t\in \left[ t_0,T_{+}(u) \right)$, then $T_{+}(u)<\infty$. 
		\end{enumerate}
		By time-reversal symmetry, the same statements hold for the negative time direction $\left(-T_{-}(u),t_0\right]$.
	\end{proposition} 
	
	The fate of the solution depends on $\Theta(u(t))$ when $u$ is ejected. If $\Theta(u(t))=-1$ and $u_0 \in L^2$, then we prove that it blows up in finite time; If $\Theta(u(t))=+1$, then we prove that it is scattering. The scattering is proved by a modification of \cite{MXZ:crit Hart:f rad} and arguments from \cite{NakS:NLS}. Unlike the subcritical case, we have to deal with possible blow-up in finite time, although the $\dot{H}^1$ norm is bounded. Both of the proof are by contradiction. 
	
	For the case $\Theta(u(t))=+1$. If we assume that scattering fails, then we can find a critical energy $E_c(\delta)$. If the energy above $E_c(\delta)$, then scattering does not hold for solution satisfies $\Theta(u(t))=+1$, which is far from the ground states. This means that there exists a sequence $\{u_n\}_{n\geq1}$ that satisfies the properties mentioned above (in fact, the distance can be upgraded from far to very far, by appealing to the ejection lemma). Using the concentration compactness procedure, the energy of $u_n$ is just above that of the ground states. Thus we can construct a critical element $U_c$, such that $E(U_c)=E_c(\delta), \Theta(U_c(t))=+1$ and $U_c$ does not scatter. Moreover, its orbit is precompact up to scaling. By using arguments in \cite{MXZ:crit Hart:f rad}, one sees that $U_c$ does not exist.
	
	Note that by Remark \ref{rmk:incln relat} and $\epsilon_{*}\leq \epsilon_{B}$, $\Theta(u(t))$ is well defined for all $t\in\left[ t_0,T_{+}(u) \right)$ in the above statement.
	
	This proposition will be proved in Section \ref{sect:asymp behr}.
	%
	%
	%
	%

	\section{Proof of Proposition \ref{prop:one pass}}\label{sect:one pass}	
	In this section, we prove Proposition \ref{prop:one pass}.
	\subsection{Setting.} Let $\delta,\epsilon>0$ and let $u$ be a solution satisfying $u(t_0)\in \mathcal{H}^{\epsilon}\cap \widetilde{B}_{\delta}(\mathcal{W})$ at some $t_0\in I(u)$. For the sake of simplicity, we denote
	\begin{equation*}
	\widetilde{d}(t):=\widetilde{d}_{\mathcal{W}}(u(t)).
	\end{equation*}
	We will define the hyperbolic and the variational regions in $\mathcal{H}^{\epsilon}$. In order to distinguish them, we use small parameters $\delta_{V},\delta_{M}\in\left(0,\delta_{X}\right]$, which will be fixed as absolute constants in the end. First we impose the following upper bounds on $\delta$ and $\epsilon$
	\begin{equation}\label{dep small}
	0<\delta\ll \delta_{V}\ll \delta_{M}\ll \delta_{X},\; 0<\epsilon\leq \min(\epsilon_{S},\epsilon_{V}(\delta)),\; \epsilon<c_{D}\delta.
	\end{equation}
	Since $\epsilon\leq \epsilon_{S}$ and $\epsilon<c_{D}\delta$, we have
	\begin{equation*}
	\mathcal{H}^{\epsilon}\setminus \widetilde{B}_{\delta}(\mathcal{W})\subset \mathcal{H}^{\epsilon_S}\cap \check{\mathcal{H}}.
	\end{equation*}
	
	Put $t_{a}:=\sup\{t_{1}\in (t_0,T_{+}(u))| t_0<t\leq t_1 \Longrightarrow \widetilde{d}(t)<\delta \}$. Since $\widetilde{d}(t_0)<\delta$, we have $t_{a}\in (t_0,T_{+}(u)]$.
	\begin{enumerate}
		\item[$(1)$] If $t_{a}=T_{+}(u)$, then \eqref{one-pass} holds with $t_{+}(\delta)=T_{+}(u)$. 
		\item[$(2)$] If $t_{a}<T_{+}(u)$, then $\widetilde{d}(t_{a})=\delta$.
		\begin{enumerate}
			\item[(a)] If $\{t\in (t_a,T_{+}(u))|\widetilde{d}(t)\leq \delta\}=\varnothing$, then \eqref{one-pass} holds with $t_{+}(\delta)=t_a$. 
			\item[(b)]  If $\{t\in (t_a,T_{+}(u))|\widetilde{d}(t)\leq \delta\} \neq \varnothing$, let $t_{b}:=\inf\{t\in (t_a,T_{+}(u))\;|\;\widetilde{d}(t)\leq\delta \}$, then $t_{a}<t_{b}$. Thus in the remaining case, we have
			\begin{equation*}
			t_{0}<\exists t_{a}<\exists t_{b}<T_{+}(u),\; \widetilde{d}(t_{a})=\delta=\widetilde{d}(t_{b})=\min_{t\in \left[t_{a},t_{b}\right]} \widetilde{d}(t),
			\end{equation*}
			from which we will derive a contradiction for small $\delta>0$ and for small $\epsilon>0$ with $\delta$-dependent smallness.
		\end{enumerate}
	\end{enumerate}
	
	The rest of proof concentrates on the interval  $\left[t_{a},t_{b}\right]$, where $u(t)$ stays in $\check{\mathcal{H}}\cap \mathcal{H}^{\epsilon_S}$, and so $\Theta(u(t))\in \{\pm1\}$ is a constant, abbreviated by $\Theta$ in the following. $\widetilde{d}(t_{a})=\delta$ implies $E(u)=\left(1+O(\delta^2)\right)E(W) \gtrsim 1$.

	\subsection{Hyperbolic and the variational regions.} Let $s$ be any local minimizer of the function $\widetilde{d}(t)$ on $\left[t_{a},t_{b}\right]$ such that $\widetilde{d}(s)<\delta_{V}$. Then Proposition \ref{prop:eject mod} from $t=s$, forward and backward in time if $s\notin \{t_{a},t_{b}\}$, forward in time if $s=t_{a}$, and backward in time if $s=t_{b}$, yields a unique subinterval $I[s]\subset \left[t_{a},t_{b}\right]$ such that
	\begin{enumerate}
		\item[(1)] $\widetilde{d}(t)\sim -\Theta\lambda_{1}(t) \sim \widetilde{d}(s)e^{\mu|\tau(t)-\tau(s)|}$ on $I[s]$, 
		\item[(2)] $\widetilde{d}(t)^2$ is strictly convex as a function of $\tau$ on $I[s]$ with minimum $<\delta_{V}$ at $t=s$,
		\item[(3)] $\widetilde{d}(t)=\delta_{M}$ on $\partial I[s]\setminus \{t_{a},t_{b}\}$.	
	\end{enumerate}
	
	Let $\mathscr{L}$ be the set of those local minimum points. Since $\widetilde{d}(t)$ ranges over $[\delta
	_V,\delta_{M}]$ between any pair of points in $\mathscr{L}$, its uniform continuity on $[t_{a},t_{b}]$ implies $\mathscr{L}$ is a finite set. Decompose the interval $[t_{a},t_{b}]$ into the hyperbolic time $I_H$ and the variational time $I_V$ defined by the following 
	\begin{equation*}
	I_{H}:=\bigcup_{s\in\mathscr{L} }I[s],\;\; I_V:=[t_{a},t_{b}]\setminus I_{H}.
	\end{equation*}
	By the definition of $\mathscr{L}$ and $I[s]$, we have
	\begin{equation*}
	t\in I_{H}\Longrightarrow \delta\leq \widetilde{d}(t)\leq \delta_{M} ,\;\;\; t\in I_V \Longrightarrow  \widetilde{d}(t)\geq \delta_{V}.
	\end{equation*}
	In particular, the coordinates $\sigma,\theta,v,\lambda_+,\lambda_{-},\lambda_{1},\lambda_{2},\gamma$ and $\tau$ are defined on $I_{H}$. Since $u$ is fixed, we regard those as functions of $t\in I_{H}$ in the rest of proof.
	
	The soliton size on $I_{H}$ is measured by
	\begin{equation}\label{quan:size IH}
	R_{H}:=\sup_{t\in I_{H}}e^{-\sigma(t)} \sim \max_{s\in \mathscr{L}}e^{-\sigma(s)},
	\end{equation}
	where the equivalence follows from \eqref{est:dyn epara} on each $I[s]$. The hyperbolic dynamics in $\tau$ on $I[s]$ together with the time scaling $\dot{\tau}=e^{2\sigma}$ implies that
	\begin{equation*}
	e^{2\sigma(s)}|I[s]|\sim \log(\delta_{M}/\widetilde{d}(s))\in [\log(\delta_{M}/\delta_{V}),\log(\delta_{M}/\delta)].
	\end{equation*}

	\subsection{Virial identity.} Now we consider a localized virial identity. For $R>0$, put
	\begin{equation} \label{localV}
	V_{R}(t):=\int_{\R^d} R^2\varphi\left(\frac{x}{R} \right)|u(t,x)|^2 dx,
	\end{equation}
	where $\varphi(x)$ be a smooth radial function that satisfies $\varphi(x)=|x|^2$ for $|x|\leq 1$ and $\varphi(x)=0$ for $|x|\geq 2$. 
	\begin{lemma}[\cite{MWX:Hart}]
		Let $u(t,x)$ be a radial solution to \eqref{equ:Hart}, $V_{R}(t)$ be
		defined by \eqref{localV}. Then
		\begin{align*}
		\dot{V}_R(t)=&\; 2\Im \int_{\R^d} \overline{u}\; \nabla u \cdot R \nabla\varphi\left(\frac{x}{R} \right)\; dx, \\
		\ddot{V}_R(t)=&\; 8 \int_{\R^d} \big|\nabla u (t,x)\big|^2 dx-8
		\iint_{\R^d\times\R^d}\frac{\big|u(t,x)\big|^2\big|u(t,y)\big|^2}{|x-y|^4} dxdy+ A_R\big(u(t)\big)\\=&\; 8K(u)+A_R\big(u(t)\big),
		\end{align*}
		where
		\begin{align*} A_R\big(u(t)\big):=&\;\int_{\R^d} \left( 4\varphi''\Big(\frac{|x|}{R}\Big)-8 \right) \big|\nabla u (t,x) \big|^2 dx
		+  \int_{\R^d} \left(-\frac{1}{R^2}\Delta\Delta \varphi\left(\frac{x}{R} \right) \right) \big| u(t,x)\big|^2 dx \\
		& + 8\iint_{\R^d\times\R^d}\left(1-\frac{1}{2} \frac{R}{|x|}\varphi'\left(\frac{x}{R}\right)\right)\frac{x(x-y)}{|x-y|^6}|u(t,x)|^2|u(t,y)|^2 dxdy\\& - 8\iint_{\R^d\times\R^d}\left(1-\frac{1}{2} \frac{R}{|y|}\varphi'\left(\frac{y}{R}\right)\right)\frac{y(x-y)}{|x-y|^6}|u(t,x)|^2|u(t,y)|^2 dxdy.
		\end{align*}
	\end{lemma}
	Combined with the cut-off function \eqref{funct:cutf funct}, we obtain
	\begin{equation}\label{est:V 2deri}
	\begin{aligned}
	\ddot{V}_R(t)=&\; \int_{\R^d}4\varphi''\Big(\frac{|x|}{R}\Big)\big|\nabla u (t,x) \big|^2 dx
	+  \int_{\R^d} \left(-\frac{1}{R^2}\Delta\Delta \varphi\left(\frac{x}{R} \right) \right) \big| u(t,x)\big|^2 dx\\&
	+8\iint_{\R^d\times\R^d}\left(-\frac{1}{2} \frac{R}{|x|}\varphi'\left(\frac{x}{R}\right)\right)\frac{x(x-y)}{|x-y|^6}|u(t,x)|^2|u(t,y)|^2 dxdy\\& + 8\iint_{\R^d\times\R^d}\left(\frac{1}{2} \frac{R}{|y|}\varphi'\left(\frac{y}{R}\right)\right)\frac{y(x-y)}{|x-y|^6}|u(t,x)|^2|u(t,y)|^2 dxdy\\=&\;8K\left(\phi_{R}u\right)+O(E_R).
	\end{aligned}
	\end{equation}
	where
	\begin{equation*}
	\begin{aligned}
	E_R(t):=\Vert u \Vert_{\dot H^1(R\leq|x|\leq 2R)}^2+\Vert \frac{u}{|x|} \Vert_{L^2(R\leq|x|\leq 2R)}^2+\iint_{\R^d\times \R^d \setminus \varOmega_1} \frac{\big|u(t,x)\big|^2\big|u(t,y)\big|^2}{|x-y|^4} dxdy,
	\end{aligned}
	\end{equation*}
	\begin{equation*}
	\varOmega_1:=\{(x,y)\in\R^d\times\R^d|\;|x|\leq R,\;|y|\leq R \}.
	\end{equation*}
	
	If $t\in\{ t_a,t_b \}$, then $\widetilde{d}(t)=\delta<\delta_{X}$, that is $u(t)\in \widetilde{B}_{\delta_{X}}(\mathcal{W})\subset B_{\delta_{L}}(\mathcal{W})\subset B_{\delta_{E}}(\mathcal{W})$. Using Proposition \ref{prop:orth decom}, we have
	\begin{equation}\label{est:V 1deri bd}
	\begin{aligned}
	|\dot{V}_R(t)|=&|2\Im \int e^{(\frac{d}{2}-1)\sigma} (W+\bar{v})(e^{\sigma}x)e^{\frac{d}{2}\sigma}\nabla(W+v)(e^{\sigma}x)R\nabla\varphi(\frac{x}{R})dx|\\=&|\frac{2}{e^{2\sigma}}\Im \int_{|x|\leq2\widetilde{R}}\left( W\nabla v+\bar{v}\nabla W+\bar{v}\nabla v \right)(x)\widetilde{R} \nabla\varphi(\frac{x}{\widetilde{R}})dx|\\ \leq& \frac{4R}{e^{\sigma}} \int_{|x|\leq2\widetilde{R}}|W\nabla v+\bar{v}\nabla W+\bar{v}\nabla v|(x)dx\lesssim \frac{R\delta}{e^{\sigma}}.
	\end{aligned}
	\end{equation}

	In order to control the cut-off error in $I_V$, we introduce
	\begin{equation*}
	\mathcal{I}_V:=\int_{I_V}\left(\iint \frac{\big|u(t,x)\big|^2\big|u(t,y)\big|^2}{|x-y|^4} dxdy+\Vert u \Vert_{\dot H^1}^2+\Vert \frac{u}{|x|} \Vert_{L^2}^2 \right) dt.
	\end{equation*}
	By Hardy's inequality and $E(u)\sim E(W)>0$, we have $\mathcal{I}_V\sim \int_{I_V} \Vert u \Vert_{\dot H^1}^2 dt$. Since
	\begin{equation*}
	\begin{aligned}
	&\int_{0}^{\infty} \int_{I_V} \int_{|x|>R} \frac{R}{r}\left( |\nabla u|^2+|\frac{u}{r}|^2+\left( |x|^{-4}*|u|^2\right)|u|^2 \right)dxdt\frac{dR}{R}\\=&\int_{\R^d}\int_{I_V} \int_{0}^{r}\frac1r \left( |\nabla u|^2+|\frac{u}{r}|^2+\left( |x|^{-4}*|u|^2\right)|u|^2 \right)dRdtdx\\=&\int_{\R^d}\int_{I_V}\left( |\nabla u|^2+|\frac{u}{r}|^2+\left( |x|^{-4}*|u|^2\right)|u|^2 \right)dtdx=\mathcal{I}_V,
	\end{aligned}
	\end{equation*}
	there exists $R\in(R_0,R_1)$ for any $R_{1}>R_{0}>0$ such that 
	\begin{equation}\label{est:IV int}
	\int_{I_V} \int_{|x|>R} \frac{R}{r}\left( |\nabla u|^2+|\frac{u}{r}|^2+\left( |x|^{-4}*|u|^2\right)|u|^2 \right)dxdt \leq \frac{\mathcal{I}_V}{\ln(R_{1}/R_{0})}.
	\end{equation}

	\subsection{Blow-up region.} We start with the simpler case $\Theta(u)=-1$, where the solution will blow up. 
	
	First, we consider $\ddot{V}_R$ on $I_H$. By \eqref{est:V 2deri}, we only need to estimate $K\left(\phi_{R}u\right)$ and $E_R$ on $I_H$ for $R$ to be chosen properly. For $t\in I_H$, we have $\delta\leq \widetilde{d}(t)\leq \delta_{M}\ll\delta_{X}$, thus $u(t)\in B_{\delta_{E}}(\mathcal{W})$.
	
	Firstly, using the scale invariance of $K$, Sobolev's inequality, Proposition \ref{prop:orth decom} and \eqref{est:EK expa}, we have: for $t\in I_H$ and $\widetilde{R}(t):=Re^{\sigma(t)}$,
	\begin{equation*}
	\begin{aligned}
	K\left(\phi_{R}u\right)&=K\left(\phi_{\widetilde{R}}(W+v)\right)=K\left(W+ \left(v-\phi_{\widetilde{R}}^{C}(W+v) \right)\right)\\&=-2\langle \nabla W,\nabla \left( v-\phi_{\widetilde{R}}^{C}(W+v) \right) \rangle+O\left(\Vert \phi_{\widetilde{R}}v-\phi_{\widetilde{R}}^{C}W\Vert_{\dot H^1}^2\right)\\&=-2\mu\lambda_{1} \langle W,g_2\rangle+2\langle \Delta W,\gamma\rangle+2\langle \nabla W,\nabla \left( \phi_{\widetilde{R}}^{C}(W+v) \right) \rangle+O\left(\Vert \phi_{\widetilde{R}}v-\phi_{\widetilde{R}}^{C}W\Vert_{\dot H^1}^2\right)\\&=-2\mu\lambda_{1} \langle W,g_2\rangle+O\left( \Vert \gamma \Vert_{\dot H^1}+\Vert W \Vert_{\dot H^1(|x|>\widetilde{R})}
	+\Vert v \Vert_{\dot H^1}^2+\Vert \phi_{\widetilde{R}}^{C}W \Vert_{\dot H^1}^2 \right).
	\end{aligned}
	\end{equation*}
	Combined with $\Vert W \Vert_{\dot H^1(|x|>\widetilde{R})}^2\lesssim \langle \widetilde{R} \rangle^{-1} $ and $\Vert \phi_{\widetilde{R}}^{C}W \Vert_{\dot H^1}^2 \lesssim \langle \widetilde{R} \rangle^{-1}$, we have
	\begin{equation}\label{est:IH Kcutf}
	t\in I_H\;\;\Longrightarrow\;\; K\left(\phi_{R}u\right)=-2\mu\lambda_{1} \langle W,g_2\rangle+O\left( \widetilde{d}(s)+\langle \widetilde{R} \rangle^{-\frac{1}{2}}+\lambda_{1}^2 \right).
	\end{equation}
	
	Secondly, using the decomposition \eqref{decom:u} and Hardy's inequality, we have: for $t\in I_H$,
	\begin{equation}\label{est:IH remn}
	\begin{aligned}
	&\iint_{\R^d\times \R^d \setminus \varOmega_1} \frac{\big|u(t,x)\big|^2\big|u(t,y)\big|^2}{|x-y|^4} dxdy+\Vert u \Vert_{\dot H^1(|x|\geq R)}^2+\Vert \frac{u}{|x|} \Vert_{L^2(|x|\geq R)}^2\\\lesssim&\Vert W+v\Vert_{\dot H^1(|y|\geq \widetilde{R})}^2+\Vert W+v\Vert_{\dot H^1(|y|\geq \widetilde{R})}^4\lesssim \langle \widetilde{R} \rangle^{-1}+\lambda_{1}^2.
	\end{aligned}
	\end{equation}
	
	Finally, putting \eqref{est:IH Kcutf} and \eqref{est:IH remn} into \eqref{est:V 2deri}, we have
	\begin{equation*}
	t\in I_H\;\;\Longrightarrow\;\;\ddot{V}_R(t)=-16\mu\lambda_{1} \langle W,g_2\rangle+O\left( \widetilde{d}(s)+\langle \widetilde{R} \rangle^{-\frac{1}{2}}+\lambda_{1}^2 \right).
	\end{equation*}
	Then using the hyperbolic dynamics of $\lambda_{1}$ and $\dot{\tau}=e^{2\sigma}\sim e^{2\sigma(s)}$, we obtain, with some absolute constant $C_H\geq 1$,
	\begin{equation*}
	\begin{aligned}
	[-\dot{V}_R(t)]_{\partial I[s]}=&\int_{I[s]} 16\mu\lambda_{1} \langle W,g_2\rangle-O\left( \widetilde{d}(s)+\langle \widetilde{R} \rangle^{-\frac{1}{2}}+\lambda_{1}^2 \right)dt\\\sim&e^{-2\sigma(s)}\int_{I[s]} \lambda_{1}-O\left( \widetilde{d}(s)+\langle \widetilde{R} \rangle^{-\frac{1}{2}}+\lambda_{1}^2 \right)d\tau\\\gtrsim& e^{-2\sigma(s)}\left( \delta_{M}-C_{H}\langle \widetilde{R}(s) \rangle^{-\frac{1}{2}}\ln \left(\delta_{M}/\delta \right) \right)\\\geq& e^{-2\sigma(s)}\left( \delta_{M}-C_{H}\left(R_H/R\right)^{\frac{1}{2}}\ln \left( \delta_{M}/\delta\right) \right).
	\end{aligned}
	\end{equation*}
	Thus if
	\begin{equation*}
	R>R_X :=\frac{4C_{H}^2}{\delta_{M}^2}\left( \ln \left( \frac{\delta_{M}}{\delta}\right) \right)^2 R_H, 
	\end{equation*}
	then for any $s\in \mathscr{L}, t\in I[s]$,
	\begin{equation*}
	\int_{I[s]}-\ddot{V}_R(t)dt=[-\dot{V}_R(t)]_{\partial I[s]}\gtrsim e^{-2\sigma(s)}\delta_{M}/2.
	\end{equation*}
	Then by the defination of \eqref{quan:size IH}, we have 
	\begin{equation}\label{est:IH int blow}
	\int_{I_H} -\ddot{V}_R(t) dt \gtrsim R_H^2 \delta_{M}.
	\end{equation}
	
	Next, we consider $\ddot{V}_R$ on $I_V$. Since $\widetilde{d}(t)\geq\delta_{V}$ on $I_V$ and $u
	(t)\in\mathcal{H}^{\epsilon_S}\cap \check{\mathcal{H}}$, Proposition \ref{prop:sign funct} implies $-1=\Theta(u(t))=signK(u(t))$. Thus $K(u(t))<0$.
	
	In order to estimate $\ddot{V}_R$ with $K(u)$ on $I_V$, the optimal cut-off radius is given by
	\begin{equation*}
	R_{V}^{-}:=\sup\{m>0\;\Big|\;\int_{I_V}\left( \int_{|x|>m}|\nabla u|^{2}dx-\iint_{\R^d\times \R^d \setminus \varOmega'_1} \frac{\big|u(t,x)\big|^2\big|u(t,y)\big|^2}{|x-y|^4} dxdy\right)dt\leq0  \},
	\end{equation*}
	where $\varOmega'_1 :=\{(x,y)\in\R^d\times\R^d\;|\;|x|\leq m,|y|\leq m\}$. $K(u)<0$ on $I_V$ implies $R_{V}^{-}>0$, while the Sobolev inequality implies that $R_{V}^{-}<\infty$. 
	
	In order to bound $R_{V}^{-}$, we use the equation $\partial_{t}|u|^2=2\Im\left( \nabla\left(\nabla u\cdot \bar{u}\right) \right)$.
	
	First, we give the estimate on $\sup_{t\in I_V} \Vert u/r \Vert_{L^2(|x|>R)}$. On the one hand, for any interval $J\subset I_V, R\geq \mathcal{I}_V^{1/2}$, we have
	\begin{equation}\label{est:decay var}
	\begin{aligned}
	[\langle |u/r|^2,\phi_{R/2}^{C} \rangle]_{\partial J}=&\int\int_{J}\frac{\phi_{R/2}^{C}}{r^2} 2\Im\left( \nabla\left(\nabla u\cdot \bar{u}\right) \right)dtdx= -2\Im\int_{|x|>\frac{R}{2}}\int_{J}\left(\nabla u\bar{u}\right)\nabla\left( \frac{\phi_{R/2}^{C}}{r^2} \right)dtdx \\ \lesssim& \int_{J} \int_{|x|>\frac{R}{2}} \frac{|uu_{r}|}{r^3} dxdt \leq \int_{I_V}\int_{|x|>\frac{R}{2}}\frac{R^2}{\mathcal{I}_V r^2} \frac{1}{2}\left( |u_r|^2+|u/r|^2 \right)dxdt\\ \sim& \frac{1}{\mathcal{I}_V}\int_{I_V}\int_{|x|>\frac{R}{2}} \frac{R^2}{r^2}\left( |u_r|^2+|u/r|^2 \right) dxdt .
	\end{aligned}
	\end{equation}
	On the other hand, we have from \eqref{est:IH remn},
	\begin{equation*}
	t\in \partial I_V \Longrightarrow \int_{|x|>R/2} |\frac{u}{r}|^2 dx\lesssim \langle \frac{Re^{\sigma(t)}}{2}\rangle^{-1}+\lambda_{1}^2 \lesssim \frac{R_H}{R}+\delta_{M}^2.
	\end{equation*}
	Hence, by \eqref{est:IV int}, there exists $R\sim \max(R_H,\mathcal{I}_V^{1/2})$ such that
	\begin{equation}\label{ext u/r small}
	\sup_{t\in I_V} \Vert u/r \Vert_{L^2(|x|>R)} \ll 1.
	\end{equation}
	
	Next, it is obvious that
	\begin{equation*}
	\Vert r^{\frac{d-2}{2}}u \Vert_{L^\infty(|x|>R)}^2 \lesssim \Vert \nabla u \Vert_{L^2(|x|>R)} \Vert u/r \Vert_{L^2(|x|>R)}.
	\end{equation*}
	For any $R\geq0$ and $\varphi \in \dot H^1$, we have
	\begin{equation*}
	\int_{|x|>R} |u|^{\frac{2d}{d-2}} dx \lesssim  \Vert u/r \Vert_{L^2(|x|>R)}^2 \Vert r^{\frac{d-2}{2}}u \Vert_{L^\infty(|x|>R)}^{\frac{4}{d-2}}\lesssim \Vert u/r \Vert_{L^2(|x|>R)}^{\frac{2d-2}{d-2}} \Vert \nabla u \Vert_{L^2(|x|>R)}^{\frac{2}{d-2}}.
	\end{equation*}
	Using \eqref{ext u/r small}, we have: for $t\in I_V$,
	\begin{equation*}
	\begin{aligned}
	\iint_{|x|>R} \frac{|u(t,x)|^2|u(t,y)|^2}{|x-y|^4} dxdy\lesssim& \Vert \chi_{\cdot>R} u\Vert_{L^{2d/(d-2)}}^2 \Vert\nabla u \Vert_{L^2(|x|>R)}^2\\  \lesssim& \Vert u/r \Vert_{L^2(|x|>R)}^{\frac{2d-2}{d}} \Vert \nabla u \Vert_{L^2(|x|>R)}^{\frac{2}{d}+2}\ll \Vert \nabla u \Vert_{L^2(|x|>R)}^{\frac{2}{d}+2},
	\end{aligned}
	\end{equation*}
	which implies
	\begin{equation}\label{est:conv remn}
	\int_{|x|>R} |\nabla u|^2 dx-\iint_{\R^d\times \R^d \setminus \varOmega_1} \frac{|u(t,x)|^2|u(t,y)|^2}{|x-y|^4} dxdy \gg 0.
	\end{equation}
	By the definition of $R_{V}^{-}$ ,
	\begin{equation}\label{est:RV- bd}
	R_{V}^{-}\leq R\sim \max(R_H,\mathcal{I}_V^{1/2}).
	\end{equation}
	
	Since $u(t)\in \mathcal{H}^{\epsilon}, \epsilon \leq \epsilon_{V}(\delta_{V})$ and $\widetilde{d}(t) \geq \delta_{V}$ on $I_V$, Proposition \ref{prop:vart est} implies 
	\begin{equation}\label{est:IV lbd Kblow}
	t\in I_V \Longrightarrow -K(u)\geq \kappa(\delta_{V}) \geq c_0(\delta_{V})\Vert u \Vert_{\dot H^1}^2 ,
	\end{equation}
	for some constant $c_0(\delta_{V})>0$.
	For any $R\geq R_{V}^{-}, t\in I_V $, we have from \eqref{est:V 2deri} and \eqref{est:IV lbd Kblow},
	\begin{equation*}
	\begin{aligned}
	-\ddot{V}_R(t)=&-8\int_{|x|\leq R} |\nabla u|^{2}dx+8\iint_{\varOmega_1} \frac{\big|u(t,x)\big|^2\big|u(t,y)\big|^2}{|x-y|^4} dxdy+O(E_R)\\\geq&-8K(u)+O(E_R) \geq8 c_0(\delta_{V})\Vert u \Vert_{\dot H^1}^2+O(E_R).
	\end{aligned}
	\end{equation*}
	
	Since
	\begin{equation*}
	\begin{aligned}
	\int_{I_V} E_R(t)dt=&\int_{I_V}\left( \iint_{\R^d\times \R^d \setminus \varOmega_1} \frac{\big|u(x)\big|^2\big|u(y)\big|^2}{|x-y|^4}dxdy \right)dt \\+&\int_{I_V}\left( \Vert u \Vert_{\dot H^1(R\leq|x|\leq 2R)}^2+\Vert \frac{u}{|x|} \Vert_{L^2(R\leq|x|\leq 2R)}^2  \right)dt=:A+B,
	\end{aligned}
	\end{equation*}
	then by \eqref{est:IV int}, \eqref{est:conv remn} and $\mathcal{I}_V\sim \int_{I_V} \Vert u \Vert_{\dot H^1}^2 dt$, we have
	\begin{equation*}
	\begin{aligned}
	A\ll&\int_{I_V} \Vert u \Vert_{\dot H^1(|x|>R)}^2dt\lesssim \mathcal{I}_V,\\ B\sim&\int_{I_V}\int_{R\leq|x|\leq 2R} \frac{R}{r}\left( |\frac{u}{|x|}|^2+|\nabla u|^2 \right)dxdt\lesssim \frac{\mathcal{I}_V}{\ln (R_1/R_0)}.
	\end{aligned}
	\end{equation*}
	Hence for any $R_{0}>R_{V}^{-}$, there exists $R\in \left(R_0,C_1(\delta_{V})R_0 \right)$ for some constant $C_1(\delta_{V})>1$ such that
	\begin{equation*}
	\int_{I_V} -\ddot{V}_R(t)dt\geq 8 c_0(\delta_{V})\int_{I_V}\Vert u \Vert_{\dot H^1}^2 dt
	+\int_{I_V} O(E_R)dt \gtrsim c_0(\delta_{V})\mathcal{I}_V.
	\end{equation*}
	The minimal $R$ satisfying this and \eqref{est:IH int blow} satisfies
	\begin{equation*}
	R\leq C_1(\delta_{V})\max \left(R_{V}^{-},R_X\right),
	\end{equation*}
	for which \eqref{est:V 1deri bd} with $[t_a,t_b]=I_H\cup I_V$ implies 
	\begin{equation}\label{vir all-}
	R_{H}^2 \delta_{M}+c_0(\delta_{V})\mathcal{I}_V\lesssim[-\dot{V}_R]_{t_a}^{t_b}\lesssim \frac{R \delta}{e^{\sigma(t)}} \leq \frac{\delta}{e^{\sigma(t)}} C_1(\delta_{V}) \left(R_{V}^{-}+R_{X}  \right) .
	\end{equation}
	
	Now we impose an upper bound on $\delta$ by the condition
	\begin{equation}\label{con:blow 1}
	\frac{\delta}{e^{\sigma(t)}} C_1(\delta_{V})R_{X} \ll R_{H}^2 \delta_{M}.
	\end{equation}
	Then the last term in \eqref{vir all-} is absorbed by the first one, hence
	\begin{equation}\label{vir red-}
	R_{H}^2 \delta_{M}+c_0(\delta_{V})\mathcal{I}_V\lesssim \frac{\delta}{e^{\sigma(t)}} C_1(\delta_{V})R_{V}^{-}.
	\end{equation}
	
	Imposing another upper bound on $\delta$ by the condition
	\begin{equation}\label{con:blow 2}
	\frac{\delta}{e^{\sigma(t)}} C_{1}(\delta_{V}) \ll c_0(\delta_{V})\mathcal{I}_V^{\frac12}.
	\end{equation}
	Combined with \eqref{est:RV- bd}, we see that \eqref{con:blow 2} contradicts \eqref{vir red-}.
	
	In conclusion, the smallness conditions on $\delta,\epsilon$ in the case $\Theta=-1$ are \eqref{dep small}, \eqref{con:blow 1} and  \eqref{con:blow 2}, which determine $\delta_{B}$ and $\epsilon_{B}$.

	\subsection{Scattering region.} Now we consider the case $\Theta(u)=+1$, where the solution will scatter. The argument is similar to that in the previous case, but more involved. In particular, we need several smallness conditions on $\delta_{M}$.
	
	It can be observed that there exists an absolute constant $C_E \sim 1$ such that
	\begin{equation}\label{est:H1 bd}
	1/C_E \leq \Vert u(t) \Vert_{\dot H^1}^2 \leq C_E
	\end{equation}
	for all $t\in I(u)$.
	
	Indeed, it is obvious that $E(u)\sim E(W)\sim 1$. On the one hand, by \eqref{quan:KIG}, we have $\Vert u\Vert_{\dot H^1}^2\leq4E(u)\lesssim1$. On the other hand, $E(\varphi)\sim \Vert \varphi \Vert_{\dot H^1}^2$ for small $\varphi \in \dot H^1$, then $\Vert u \Vert_{\dot H^1}^2\geq1/C_E $.
	
	First, we consider $\ddot{V}_R$ on $I_H$. The proof is exactly the same as before, that is, 
	\begin{equation*}
	t\in I_H\;\Longrightarrow \;\ddot{V}_R(t)=-16\mu\lambda_{1} \langle W,g_2\rangle+O\left( \widetilde{d}(s)+\langle \widetilde{R} \rangle^{-\frac{1}{2}}+\lambda_{1}^2 \right).
	\end{equation*}
	Then using the hyperbolic dynamics of $\lambda_{1}$ and $\dot{\tau}=e^{2\sigma}\sim e^{2\sigma(s)}$, we obtain
	
	\begin{equation*}
	\begin{aligned}
	[\dot{V}_R(t)]_{\partial I[s]}=&\int_{I[s]} -16\mu\lambda_{1} \langle W,g_2\rangle+O\left( \widetilde{d}(s)+\langle \widetilde{R} \rangle^{-\frac{1}{2}}+\lambda_{1}^2 \right)dt\\\sim&e^{-2\sigma(s)}\int_{I[s]} -\lambda_{1}+O\left( \widetilde{d}(s)+\langle \widetilde{R} \rangle^{-\frac{1}{2}}+\lambda_{1}^2 \right)d\tau\\\gtrsim& e^{-2\sigma(s)}\left( \delta_{M}-C_{H}\langle \widetilde{R}(s) \rangle^{-\frac{1}{2}}\ln \left(\delta_{M}/\delta \right) \right)\\\geq& e^{-2\sigma(s)}\left( \delta_{M}-C_{H}\left(R_H/R\right)^{\frac{1}{2}}\ln \left( \delta_{M}/\delta\right) \right).
	\end{aligned}
	\end{equation*}
	Thus if $R>R_X$, then for any $s\in \mathscr{L}, t\in I[s]$,
	\begin{equation*}
	\int_{I[s]} \ddot{V}_R(t)dt=[ \dot{V}_R(t)]_{\partial I[s]}\gtrsim e^{-2\sigma(s)}\delta_{M}/2.
	\end{equation*}
	Then by the defination of \eqref{quan:size IH}, we have 
	\begin{equation}\label{est:IH int scat}
	R>R_X  \Longrightarrow \int_{I_H} \ddot{V}_R(t) dt \gtrsim R_H^2 \delta_{M}.
	\end{equation}

	Next, we consider $\ddot{V}_R$ on $I_V$. Since $\widetilde{d}(t)\geq\delta_{V}$ on $I_V$ and $u
	(t)\in\mathcal{H}^{\epsilon_S}\cap \check{\mathcal{H}}$, Proposition \ref{prop:sign funct} implies $+1=\Theta(u(t))=signK(u(t))$. Thus $K(u(t))\geq0$.
	
	In order to estimate $\ddot{V}_R$ with $K\left(\phi_{R}u\right)$ on $I_V$. Since $u(t)\in \mathcal{H}^{\epsilon}, \epsilon \leq \epsilon_{V}(\delta_{V})$ and $\widetilde{d}(t) \geq \delta_{V}$ on $I_V$, Proposition \ref{prop:vart est} implies
	\begin{equation*}
	K(u(t))\geq \min\{\kappa(\widetilde{d}(t)), c_{V}\Vert u\Vert_{\dot H^1}^2 \}\geq \min\{\kappa(\delta_{V}), c_{V}/C_E \}=:\widetilde{\kappa}(\delta_{V}).
	\end{equation*}
	By \eqref{est:H1 bd}, if $t\in I_V, 4\epsilon^2 \leq \widetilde{\kappa}(\delta_{V})$, then 
	\begin{equation*}
	\begin{aligned}
	I\left(\phi_{R}u\right) \leq& I(u)<E(W)+\epsilon^2-K(u)/2\\  \leq& E(W)+\epsilon^2-\widetilde{\kappa}(\delta_{V})/2 \leq I(W)-\widetilde{\kappa}(\delta_{V})/4.
	\end{aligned}
	\end{equation*} 
	By Proposition \ref{prop:sharp const}, we obtain
	\begin{equation}\label{est:IV Kcutf}
	\begin{aligned}
	K(\phi_{R}u)\geq& \Vert \phi_{R}u \Vert_{\dot H^1}^2 \left( 1-\frac{ \Vert \left(|x|^{-4}*|\phi_{R}u|^2 \right)|\phi_{R}u|^2 \Vert_{L^1}^{1/2} \Vert \left(|x|^{-4}*|W|^2 \right)|W|^2 \Vert_{L^1}^{1/2}  } {\Vert W \Vert_{\dot H^1}^2}  \right)\\ \geq& \Vert \phi_{R}u \Vert_{\dot H^1}^2\left(1-\frac{\left( \Vert \left(|x|^{-4}*|W|^2 \right)|W|^2 \Vert_{L^1}-\widetilde{\kappa}(\delta_{V}) \right)^{1/2}}{\Vert W \Vert_{\dot H^1}}  \right)\\\gtrsim& \widetilde{\kappa}(\delta_{V}) \Vert \phi_{R}u \Vert_{\dot H^1}^2
	\geq \widetilde{\kappa}(\delta_{V}) \Vert u \Vert_{\dot H^1(|x|<R
		)}^2.
	\end{aligned}
	\end{equation}
	
	In order to decide the cut-off for $I_V$, we introduce
	\begin{equation*}
	R_{V}^{+}(\delta_{M}):=\inf\{m>0\Big|\int_{I_V}\int_{|x|<m} |\nabla u|^2 dxdt\geq \delta_{M}^3|I_V| \}\in (0,\infty).
	\end{equation*}
	As above,
	\begin{equation*}
	\int_{I_V} E_R(t)dt=A+B,\;\;A\ll \mathcal{I}_V,\;\;B\lesssim\frac{\mathcal{I}_V}{\ln (R_1/R_0)}.
	\end{equation*}
	Using \eqref{est:V 2deri}, \eqref{est:IV int} and \eqref{est:IV Kcutf}, for any $R_0 \geq R_{V}^{+}(\delta_{M})$, there exists $R\in \left(R_0,C_2(\delta_{V})R_0 \right)$ for some constant $C_2(\delta_{V})>1$ such that
	\begin{equation*}
	\begin{aligned}
	\int_{I_V} \ddot{V}_R(t)dt= \int_{I_V}8K(\phi_{R}u)
	+O(E_R)dt \gtrsim& \int_{I_V}\widetilde{\kappa}(\delta_{V}) \Vert u \Vert_{\dot H^1(|x|<R)}^2+E_R\; dt\\\gtrsim& \widetilde{\kappa}(\delta_{V})\int_{I_V} \Vert u \Vert_{\dot H^1(|x|<R)}^2 dt\geq \widetilde{\kappa}(\delta_{V})\delta_{M}^3|I_V|.
	\end{aligned}
	\end{equation*}
	Note that $\delta_{V}$ depends on $\delta_{M}$ through the condition $\delta_{V}\ll \delta_{M}$ in \eqref{dep small}, which allows us to determine $C_2$ in terms of $\delta_{V}$ only. 
	
	The minimal $R>R_0$ satisfying both this and \eqref{est:IH int scat} must satisfy
	\begin{equation*}
	R\leq C_2(\delta_{V})\max\left(R_{V}^{+}(\delta),R_X \right),
	\end{equation*}
	for which \eqref{est:V 1deri bd} implies
	\begin{equation}\label{vir all +}
	R_{H}^{2}\delta_{M}+\widetilde{\kappa}(\delta_{V})\delta_{M}^{3}|I_V|\lesssim [\dot{V}_R(t)]_{t_a}^{t_b} \lesssim \frac{R\delta}{e^{\sigma(t)}}  \leq \frac{\delta}{e^{\sigma(t)}} C_2(\delta_{V}) \left(  R_{V}^{+}(\delta)+R_{X}  \right).
	\end{equation}
	
	Now we impose an upper bound on $\delta$ by the condition
	\begin{equation*}
	\frac{\delta}{e^{\sigma(t)}} C_2(\delta_{V})R_{X} \ll R_{H}^{2}\delta_{M},
	\end{equation*}
	which is equivalent to 
	\begin{equation}\label{con:scat 1}
	\frac{\delta}{e^{\sigma(t)}}C_2(\delta_{V}) \left(C_H \ln\left( \delta_{M}/\delta \right)\right)^2 \ll R_{H}\delta_{M}^3.
	\end{equation}
	Then the last term in \eqref{vir all +} is absorbed by the first one, combined with \eqref{con:scat 1}, we have
	\begin{equation*}
	R_{H}^{2}\delta_{M}+\widetilde{\kappa}(\delta_{V})\delta_{M}^{3}|I_V|\lesssim	\frac{\delta}{e^{\sigma(t)}}C_2(\delta_{V}) R_{V}^{+}(\delta) \ll R_{H}\delta_{M}^3\left( \ln\left( \delta_{M}/\delta \right) \right)^{-2} R_{V}^{+}(\delta
	_M).
	\end{equation*}
	Imposing another upper bound on $\delta$ by 
	\begin{equation}\label{con:scat 2}
	\widetilde{\kappa}(\delta_{V})^{1/2}\ln \left( \delta_{M}/\delta \right) \gg e^{1/\delta_{M}^{3}} ,
	\end{equation}
	yields
	\begin{equation*}
	R_H^2 +|I_V|\leq \frac{R_H^2}{\widetilde{\kappa}(\delta_{V})\delta_{M}^2} +|I_V|\ll \frac{R_HR_{V}^{+}(\delta_M)}{\left( \widetilde{\kappa}(\delta_{V})^{1/2}\ln\left( \delta_{M}/\delta \right) \right)^2}\ll R_H R_{V}^{+}(\delta_M) e^{-2/\delta_{M}^{3}},
	\end{equation*}
	thus
	\begin{equation}\label{RV super}
	R_H +|I_V|\ll  e^{-2/\delta_{M}^{3}} R_{V}^{+}(\delta_M).
	\end{equation}
	
	First, we compare $R_H$ with $|I_V|$. By $\mathcal{I}_V\sim \int_{I_V} \Vert u \Vert_{\dot H^1}^2 dt$, \eqref{est:H1 bd}, \eqref{RV super} and \eqref{est:IV int} with $R_0:=R_H +|I_V|, R_1:= R_{V}^{+}(\delta_M)/2$, then there exists $R\in (R_0,R_1)$ such that
	\begin{equation}\label{def R}
	\int_{I_V} \int_{|x|>R} \frac{R}{r}\left( |\nabla u|^2+|u/r|^2+\left( |x|^{-4}*|u|^2\right)|u|^2 \right)dxdt \leq \frac{\mathcal{I}_V}{\ln(R_1/R_0)}\lesssim \delta_{M}^{3} |I_V|.
	\end{equation}
	Using Hardy's inequality, we have
	\begin{equation*}
	\begin{aligned}
	\Vert u/r \Vert_{L^2(|x|<2R)} \lesssim& \Vert \phi_{R}u \Vert_{\dot H^1}+ \Vert u/r \Vert_{L^2(R<|x|<2R)} \\ \lesssim& \Vert u/r \Vert_{L^2(R<|x|<2R)} + \Vert \nabla u \Vert_{L^2(|x|<2R)}.
	\end{aligned}
	\end{equation*}
	By the defination of $R_{V}^{+}(\delta_{M})$ and $R_{V}^{+}(\delta_{M}
	)>2R$, we have $\int_{I_V} \Vert \nabla u \Vert_{L^2(|x|<2R)}^2 dt<\delta_{M}^3|I_V|$. Using \eqref{def R}, we have
	\begin{equation*}
	\int_{I_V} \Vert u/r \Vert_{L^2(R<|x|<2R)}^2\sim \int_{I_V}\int_{R<|x|<2R}\frac{R}{r}|\frac{u}{r}|^2dxdt\lesssim \delta_{M}^{3} |I_V|.
	\end{equation*}
	Combining the above estimates yieds
	\begin{equation*}
	\int_{I_V} \int_{|x|<2R} \left( |\nabla u|^2+|u/r|^2 \right)dxdt\lesssim \int_{I_V}\Vert\frac{u}{r}\Vert_{L^2(R<|x|<2R)}^2  dt+ \int_{I_V}\Vert\nabla u \Vert_{L^2(|x|<2R)}^2 dt\lesssim \delta_{M}^{3} |I_V|.
	\end{equation*}
	By $E(u)\gtrsim0$, we have 
	\begin{equation*}
	\int_{I_V}\int_{|x|<2R} \left( |x|^{-4}*|u|^2\right)|u|^2dxdt\lesssim \int_{I_V}\Vert \nabla u \Vert_{L^2(|x|<2R)}^2 dt \lesssim \delta_{M}^{3} |I_V|.
	\end{equation*}
	Thus we have
	\begin{equation}\label{est:IV remn}
	\int_{I_V}\int_{|x|<2R} \left( |\nabla u|^2 + |u/r|^2 + \left( |x|^{-4}*|u|^2\right)|u|^2 \right) dxdt \lesssim \delta_{M}^{3} |I_V|.
	\end{equation}
	If $R_{H}^{2}\gg \delta_{M}^{2}|I_V|$, then by \eqref{est:V 2deri}, \eqref{est:IV remn} and $R_H<R$, we have
	\begin{equation*}
	\int_{I_V} |\ddot{V}_{R_H}|dt\lesssim \int_{I_V}\int_{|x|\leq 2R}\left( |\nabla u|^2+|u/r|^2+\left(|x|^{-4}* |u|^2 \right)|u|^2  \right)dxdt\lesssim \delta_{M}^{3} |I_V| \ll R_{H}^{2}\delta_{M}.
	\end{equation*}
	Noting that $R_X \ll R_H$ by \eqref{con:scat 2}, using the above estimate, \eqref{est:V 1deri bd} and \eqref{est:IH int scat}, we have
	\begin{equation*}
	\frac{\delta R_{H}}{e^{\sigma(t)}}\gtrsim [\dot{V}_{R_H}]_{t_a}^{t_b} =\int_{I_H}\ddot{V}_{R_H} dt+\int_{I_V}\ddot{V}_{R_H} dt \gtrsim R_{H}^{2}\delta_{M},
	\end{equation*}
	which contradicts $\delta\ll \delta_{M}$. Therefore $R_{H}^{2} \lesssim  \delta_{M}^{2}|I_V|$.
	
	Next, we estimate $\Vert u/r \Vert_{L^2(I_V \times \R^d)}$. On the one hand, for any interval $J\subset I_V, R>R_0>|I_V|^{1/2}$, using \eqref{def R} and the same argument as \eqref{est:decay var} , we have
	\begin{equation*}
	[\langle |u/r|^2,\phi_{R}^{C} \rangle]_{\partial J}\lesssim \frac{1}{|I_V|}\int_{I_V}\int_{|x|>R} \frac{R^2}{r^2}\left( |u_r|^2+|u/r|^2 \right) dxdt \lesssim \delta_{M}^{3}.
	\end{equation*}
	On the other hand, we have from \eqref{est:IH remn} and  $R_{H}^{2} \lesssim  \delta_{M}^{2}|I_V|$,
	\begin{equation*}
	t\in \partial I_V \Longrightarrow \int_{|x|>R} |u/r|^2 dx \lesssim \frac{R_H}{R}+\delta_{M}^{2} < \frac{R_H}{|I_V|^{1/2}}+\delta_{M}^{2} \lesssim \delta_{M}.
	\end{equation*}
	Combining the above two estimates yields 
	\begin{equation*}
	\sup_{t\in I_V} \Vert u/r \Vert_{L^2(|x|>2R)}^{2} \lesssim \delta_{M}.
	\end{equation*}
	Using above and \eqref{est:IV remn}, we obtain
	\begin{equation*}
	\int_{I_V} \int_{\R^d} |u/r|^2 dxdt=\int_{I_V} \int_{|x|<2R} |u/r|^2 dxdt+\int_{I_V} \int_{|x|>2R} |u/r|^2 dxdt \lesssim \delta_{M}|I_V|.
	\end{equation*}
	Decomposing $I_V$ into its connected components, we obtain an interval $I\subset I_V$ such that $\partial I\subset \partial I_V$ and
	\begin{equation}\label{est:IV decay}
	\int_I \int_{\R^d} |u/r|^2 dxdt \lesssim \delta_{M}|I|.
	\end{equation}
	
	Using \eqref{est:IV decay}, we can achieve Bourgain's energy induction method \cite{Bourg}. It allows us to construct a solution whose energy is smaller than the original one by a nontrivial amount. In particular, it is smaller than that of the ground states (see \eqref{est:E widew}). Although \cite{Bourg,MXZ:crit Hart:def rad} treated the defocusing case, the perturbative argument works as well for the focusing equation \eqref{equ:Hart} under the uniform bound \eqref{est:H1 bd} in $\dot{H}^1$, while the non-perturbative argument with the Morawetz estimate can be replaced with \eqref{est:IV decay}, as is shown below.
	
	In order to apply the argument to the interval $I$, we have the following key observation.
	\begin{lemma} For the interval $I$ defined above, we have
		\begin{equation*}
		\Vert u \Vert_{S(I)} \gtrsim 1.
		\end{equation*}
	\end{lemma}
	\begin{proof}
		Let $t_0:=\inf I\in\partial I\subset\partial I_V$, then by the definition of $I_V$, we have $\widetilde{d}(t_0)=\delta_{M}$. Without loss of generality, we use Proposition \ref{prop:eject mod} forward in time, there exists $t_1\in I_V$ such that $\widetilde{d}(t_1)=\delta_{X}$ and $\partial_t\widetilde{d}(t)>0$ on $(t_0,t_1)\subset I$. By the scaling invariance, we can reduce the case to $\sigma(t_0)=0$. Using \eqref{est:dyn d}, \eqref{est:dyn epara}, $\dot{\tau}=e^{2\sigma}$ and  $\delta_{M}\ll \delta_{X}$, we have $t_1>t_0 +1$. It suffices to show $\Vert u \Vert_{S(t_0,t_0+1)}\gtrsim1$.
		
		Put $v:=u-W$, then 
		\begin{equation*}
		\begin{aligned}
		i\dot{v}-\Delta v=&(W+v)\left( |x|^{-4}*|W+v|^2 \right)-W\left( |x|^{-4}*|W|^2 \right)\\=&-Vv-iR(v),
		\end{aligned}
		\end{equation*}
		where
		\begin{equation*}
		\begin{aligned}
		Vh:=&-\left( |x|^{-4}*|W|^2 \right)h-2\left( |x|^{-4}*(Wh_1) \right)W,\\R(h):=&i\left( |\cdot|^{-4}*|h|^2 \right)(W+h)
		+2i\left( |\cdot|^{-4}*(Wh_1) \right)h.
		\end{aligned}
		\end{equation*}
		Using Lemma 2.10 in \cite{MWX:Hart} and the Strichartz estimate \eqref{est:St est}, we have, for any interval $J=[a,b]\subset(t_0,t_0 +1)$, 
		\begin{equation*}
		\begin{aligned}
		\Vert v \Vert_{\left( W^1 \cap L_{t}^{\infty}\dot H_x^1 \right)(J)} \lesssim&\Vert v(a)\Vert_{\dot H^1}+\Vert Vv+iR(v) \Vert_{N^1(J)}\\\lesssim&\Vert v(a)\Vert_{\dot H^1}+|J|^{1/3}\Vert v\Vert_{l(J)}+|J|^{1/6}\Vert v\Vert_{l(J)}^2+\Vert v\Vert_{l(J)}^3,
		\end{aligned}
		\end{equation*}
		where 
		\begin{align*}
		l(J):=&S(J)\cap L^{3}\left( I;\dot{W}^{1,\frac{6d}{3d-4}} (\mathbb{R}^d)\right),\\\Vert u\Vert_{l(J)}:=&\Vert u\Vert_{S(J)}+\Vert \nabla u\Vert_{\bar{S}(J)}.
		\end{align*}
		Hence if $|J|\ll 1$, then
		\begin{equation*}
		\Vert v \Vert_{\left( W^1 \cap L_{t}^{\infty}\dot H_x^1 \right)(J)} \lesssim \Vert v(a) \Vert_{\dot H^1}.
		\end{equation*}
		Repeating this estimate from $t=t_0$ on consecutive small intervals as well as applying Proposition \ref{prop:orth decom}, we have
		\begin{equation*}
		\Vert v \Vert_{S(t_0,t_0 +1)} \lesssim \Vert v \Vert_{W^1(t_0,t_0 +1)} \lesssim \Vert v(t_0) \Vert_{\dot H^1}\lesssim \delta_{M}.
		\end{equation*}
		So
		\begin{equation*}
		\Vert u \Vert_{S(t_0,t_0 +1)}\geq\Vert W \Vert_{S(t_0,t_0 +1)}-\Vert v \Vert_{S(t_0,t_0 +1)}= \Vert W \Vert_{L_{x}^{6d/(3d-8)} }-O(\delta_{M}) \gtrsim 1.
		\end{equation*}
		This completes the proof.
	\end{proof}	
	
	Hence as in \cite{Bourg,MXZ:crit Hart:def rad}, we can decompose the interval $I$ such that 
	\begin{equation*}
	I=[t_0,t_N],\;\; t_0 <t_1 <\cdots<t_N,\;\; I_j:=[t_j,t_{j+1}],\;\; \Vert u \Vert_{S(I_j)}\in \left[\eta,2\eta \right)
	\end{equation*}
	for a small fixed constant $\eta>0$. In the following, $c$ denotes a small positive constant, and $C(\eta)$ denotes a large positive constant which may depend on $\eta$, both allowed to change from line to line.
	
	By the perturbation argument in \cite[lemma 5.1]{MXZ:crit Hart:def rad}, where the sign of nonlinearity is irrelevant, we have for each $j$, 
	\begin{equation*}
	\Vert u \Vert_{\bar{S}^1(I_j)}\lesssim 1.
	\end{equation*} 
	There exist a subinterval $I'_{j}\subset I_j$ and $R_j \lesssim |I'_{j}|^{1/2}$ such that
	\begin{equation*}
	\inf_{t\in I'_{j}}min\left( \Vert \nabla u(t) \Vert_{L^2\left( |x|<C(\eta)R_j \right)} ,\; \Vert u(t) \Vert_{L^{2d/(d-2)}\left( |x|<C(\eta)R_j \right)} \right) \gtrsim \eta^{a},
	\end{equation*} 
	where $a=a(d)<1$ is a constant.
	Using the radial Sobolev inequality  $|r^{\frac{d-2}{2}}u|\lesssim \Vert u_r \Vert_{L^2}$, we have
	\begin{equation*}
	\int_{|x|<R} |u|^{\frac{2d}{d-2}} dx\leq \Vert u/r \Vert_{L^2(|x|<R)}^2 \Vert r^{\frac{d-2}{2}}u \Vert_{L^{\infty}}^{\frac{4}{d-2}}\lesssim  \Vert u/r \Vert_{L^2(|x|<R)}^2 \Vert u_r \Vert_{L^2}^{\frac{4}{d-2}}.
	\end{equation*}
	Combining the above estimates with \eqref{est:IV decay} yields 
	\begin{equation*}
	\begin{aligned}
	\eta^{a\frac{2d}{d-2}}|I'_{j}|\lesssim \int_{I'_{j}}\Vert u \Vert_{L^{\frac{2d}{d-2} }\left( |x|<C(\eta)R_j \right)}^{\frac{2d}{d-2}}dt \lesssim \int_{I_j} \Vert \frac{u}{r} \Vert_{L^2(|x|<C(\eta)R_j)}^2 \Vert u \Vert_{\dot H^1}^{\frac{4}{d-2}}dt \lesssim \Vert u \Vert_{L_t^{\infty}\dot H_x^1}^{\frac{4}{d-2}} \delta_{M}|I_j|,
	\end{aligned}
	\end{equation*}
	which implies
	\begin{equation*}
	\sum_{j=1}^{N}|I'_{j}|\leq C(\eta)\delta_{M} \sum_{j=1}^{N} |I_j|.
	\end{equation*}
	Hence there exists $j\in \{1,\cdots,N \}$ such that
	\begin{equation}\label{conc Ij}
	R_j^2\lesssim |I'_{j}|\leq C(\eta)\delta_{M}|I_j|.
	\end{equation}
	Fix $s\in I'_{j}$. By the time reversal symmetry, we may assume without loss of generality
	\begin{equation}\label{hyps:s}
	t_{j+1}-s >s-t_j.
	\end{equation}
	By \cite{Bourg,MXZ:crit Hart:def rad}, there exists $R\leq C(\eta)R_j$ such that 
	\begin{equation}\label{est:H1 w}
	\Vert \phi_R^C u(s) \Vert_{\dot{H}^1}^2 < \Vert  u(s) \Vert_{\dot{H}^1}^2-c\eta^{2a}.
	\end{equation}
	Let $v$ be the solution of $iv_t-\Delta v=0$ with initial data $v(s):=\phi_{R}u(s)$, $w:=u-v$. By Proposition \ref{prop:disp ST est}, Hölder's inequality and \eqref{est:H1 bd}, we have
	\begin{equation}\label{est:disp v1}
	\begin{aligned}
	\Vert v(t) \Vert_{L_x^{\frac{2d}{d-2}}} \lesssim& |t-s|^{-1} \Vert v(s)\Vert_{L_x^{\frac{2d}{d+2}}}\leq R^{2}\Vert\phi\Vert_{L_x^{\frac{d}{2}}}\Vert u(s)\Vert_{L_x^{\frac{2d}{d-2}}}|t-s|^{-1}  \lesssim R^{2}|t-s|^{-1},\\ \Vert \nabla v(t) \Vert_{L_x^{\frac{6d}{3d-2}}} \lesssim& |t-s|^{-\frac13} \Vert \nabla v(s) \Vert_{L_x^{\frac{6d}{3d+2}}}\\\leq& R^{\frac13}\left(\Vert\nabla\phi\Vert_{L_x^{\frac{3d}{4}}}\Vert u(s)\Vert_{L_x^{\frac{2d}{d-2}}}+\Vert\phi\Vert_{L_x^{3d}}\Vert\nabla u\Vert_{L_x^2}  \right)|t-s|^{-\frac13} \lesssim R^{\frac13} |t-s|^{-\frac13}.
	\end{aligned}
	\end{equation}
	Hence using \eqref{conc Ij}, \eqref{hyps:s} and $R\leq C(\eta)R_j$, we obtain 
	\begin{equation}\label{vdec}
	\begin{aligned}
	&\Vert v(t_{j+1}) \Vert_{L_x^{2d/(d-2)}} \lesssim R^{2}|t_{j+1}-s|^{-1}\leq C(\eta)\delta_{M},\\ &\Vert v \Vert_{W^1(t_{j+1},\infty)}=\left(\int_{t_{j+1}}^{\infty}\Vert \nabla v\Vert_{L_x^{6d/(3d-2)}}^6 dt \right)^{1/6} \lesssim R^{1/3}|t_{j+1}-s|^{-1/6} \lesssim C(\eta)\delta_{M}^{1/6}.
	\end{aligned}
	\end{equation}
	
	By the equation \eqref{equ:Hart} and Hölder's inequality, we have 
	\begin{equation}\label{est:conv endpo}
	\begin{aligned}
	[ {\Vert{\left(|x|^{-4}*|u|^2\right)|u|^2}\Vert}_{L^1} ]_{s}^{t_{j+1}}=&4\int_{s}^{t_{j+1}}\langle \left(|x|^{-4}*|u|^2\right)u,\dot{u} \rangle dt\\=&4\int_{s}^{t_{j+1}}\langle \nabla\left( \left(|x|^{-4}*|u|^2\right)u \right),i\nabla u \rangle dt\\\lesssim&\int_{s}^{t_{j+1}}\int\left( \left(|x|^{-4}*|u\nabla u|\right)|u\nabla u|+\left(|x|^{-4}*|u|^2\right)|\nabla u|^2 \right)dxdt\\\lesssim&\int_{s}^{t_{j+1}}\Vert u\Vert_{L_{x}^{\frac{6d}{3d-8}}}^2 \Vert\nabla u\Vert_{L_{x}^{\frac{6d}{3d-4}}}^2dt \lesssim \Vert u\Vert_{S(I_j)}^2 \Vert u\Vert_{\bar{S}^1(I_j)}^2\lesssim \eta^2.
	\end{aligned}
	\end{equation}
	By the Strichartz estimate \eqref{est:St est}, we have
	\begin{equation}\label{est:H1 w minus}
	\begin{aligned}
	\Vert w(t_{j+1})\Vert_{\dot{H}^1}-\Vert w(s)\Vert_{\dot{H}^1}\leq& \Vert\int_{s}^{t_{j+1}} e^{-i(t-\tau)\Delta}\left(\left(|x|^{-4}*|u|^2\right)u\right)(\tau)d\tau \Vert_{\dot H^1}\\\lesssim&\Vert\left(|x|^{-4}*|u|^2\right)u \Vert_{N^1(I_j)}\lesssim \Vert u\Vert_{\bar{S}^1(I_j)} \Vert u\Vert_{S(I_j)}^2 \lesssim \eta^2.
	\end{aligned}
	\end{equation} 
	Let $\widetilde{w}$ be the solution of \eqref{equ:Hart} with the initial data $\widetilde{w}(t_{j+1}):=w(t_{j+1})$. We imposed upper bounds on $\delta_{M}$ and $\epsilon$ respectively:
	\begin{equation}\label{con:scat 3} 
	C(\eta)\delta_{M}\ll \eta^2,\;\;\epsilon\ll\eta^a.
	\end{equation} 
	Then by the above estimates together with  \eqref{est:H1 w}, \eqref{est:conv endpo} and \eqref{est:H1 w minus}, we get
	\begin{equation}\label{est:E widew}
	\begin{aligned}
	E(\widetilde{w})\leq& \frac{1}{2}\Vert \nabla w(t_{j+1}) \Vert_{L^2}^2-\frac{1}{4} \Vert{\left(|x|^{-4}*|u(t_{j+1})|^2\right)|u(t_{j+1})|^2}\Vert_{L^1}+C(\eta)\delta_{M}\\=&\frac{1}{2}\Vert \nabla w(t_{j+1}) \Vert_{L^2}^2-\frac{1}{2}\Vert \nabla w(s) \Vert_{L^2}^2\\+&\frac{1}{2}\Vert \nabla w(s) \Vert_{L^2}^2-\frac{1}{4} \Vert{\left(|x|^{-4}*|u(s)|^2\right)|u(s)|^2}\Vert_{L^1}+O(\eta^2)\\=& \frac{1}{2}\Vert \nabla w(s) \Vert_{L^2}^2 - \frac{1}{4} \Vert{\left(|x|^{-4}*|u(s)|^2\right)|u(s)|^2}\Vert_{L^1}+O(\eta^2) \\\leq&\frac{1}{2}\Vert \nabla u(s) \Vert_{L^2}^2-c\eta^{2a}- \frac{1}{4} \Vert{\left(|x|^{-4}*|u(s)|^2\right)|u(s)|^2}\Vert_{L^1}+O(\eta^2) \\=& E(u)-c\eta^{2a} \leq E(W)+\epsilon^2-c\eta^{2a}<E(W)-\frac{c}{2}\eta^{2a}.
	\end{aligned}
	\end{equation} 
	
	Similarly, plugging \eqref{est:H1 w} into \eqref{est:H1 w minus} yields 
	\begin{equation}\label{est:H1 widew}
	\Vert \nabla\widetilde{w}(t_{j+1}) \Vert_{L^2}^2 \leq \Vert \nabla u(s) \Vert_{L^2}^2-c\eta^{2a},
	\end{equation} 
	while $E(u)<E(W)+\epsilon^2$ and \eqref{quan:KIG} together with $K(u(s))\geq 0$ implies 
	\begin{equation*}
	\Vert \nabla u(s) \Vert_{L^2}^2=4E(u(s))-K(u(s))\leq 4E(u)<\Vert \nabla W \Vert_{L^2}^2+4\epsilon^2.
	\end{equation*} 
	Hence $\Vert \nabla\widetilde{w}(t_{j+1}) \Vert_{L^2}^2 < \Vert \nabla W \Vert_{L^2}^2$, thus $G(\widetilde{w}(t_{j+1}))<E(W)$. By Proposition \ref{prop:GS char}, $K(\widetilde{w}(t_{j+1}))>0$.
	
	Hence by the result of C. Miao, the third author and L. Zhao \cite{MXZ:crit Hart:f rad} below the ground state energy, \eqref{est:H1 bd} and \eqref{est:H1 widew}, $\widetilde{w}$ scatters in both time directions with a uniform Strichartz bound:
	\begin{equation}\label{est:W1 widew}
	\Vert \widetilde{w} \Vert_{W^1(\R)}<C(\eta).
	\end{equation} 
	In order to control $u$ by this, we use the long-time perturbation  \cite[Proposition 2.3]{MXZ:crit Hart:f rad}:

	\begin{lemma}(\cite{MXZ:crit Hart:f rad})\label{lem:pert}. Let $u$ be a solution of \eqref{equ:Hart}. Let $I$ be an interval with some $t_0\in I\cap I(u)$. Let $e\in N^1(I)$ and let $\widetilde{u}\in C(I;\dot H^1)$ be a solution of 
		\begin{equation*}
		i\partial_t \widetilde{u} -\Delta \widetilde{u}= \left( |x|^{-4}*|\widetilde{u}|^2\right)\widetilde{u} +e.
		\end{equation*}	
		Assume that for some $B_1,B_2,B_3>0$
		\begin{equation*}
		\Vert \widetilde{u} \Vert_{L_{t}^{\infty} \dot H_x^1(I) }\leq B_1,\; \Vert \widetilde{u} \Vert_{S(I)}\leq B_2,\; \Vert \widetilde{u}(t_0)-u(t_0) \Vert_{\dot H_x^1}\leq B_3 .
		\end{equation*}
		Then there exists $\nu_P=\nu_P(B_1,B_2,B_3)>0$ such that if
		\begin{equation*}
		\Vert e^{-i(t-t_0)\Delta}\left( \widetilde{u}(t_0)-u(t_0) \right) \Vert_{\mathcal{Z}^1(I)} + \Vert e \Vert_{N^1(I)}=:\nu \leq \nu_P ,
		\end{equation*}
		then $I\subset I(u)$ and 
		\begin{equation*}
		\Vert u \Vert_{\mathcal{Z}^1(I)}\lesssim 1,\;\; \Vert \widetilde{u}-u \Vert_{L_{t}^{\infty} \dot H_x^1(I) } \lesssim \nu+B_3,
		\end{equation*}
		where the implicit constants depends on $B_1,B_2,B_3$. For some fixed number $0<\epsilon_0 \ll 1$, define $\mathcal{Z}^1(I)$ by 
		\begin{equation*}
		\Vert u \Vert_{\mathcal{Z}^1(I)}:=\sup_{(q,r)\in \wedge}\Vert u \Vert_{L_{t}^{q}L_{x}^{r}},
		\end{equation*}
		where $\wedge=\{(q,r);\frac{2}{q}=d(\frac{1}{2}-\frac{1}{r})-1,\frac{2d}{d-2}\leq r\leq \frac{2d}{d-4}-\epsilon_0 \}$.
	\end{lemma}
	
	Apply the above lemma to $u$ and $\widetilde{u}:=\widetilde{w}+v$ with $I=\left[t_{j+1},\infty \right)$ and initial data at $t=t_{j+1}$. From the bounds on $\widetilde{w}$ and $v$, we have
	\begin{equation*}
	\Vert \widetilde{u} \Vert_{L_{t}^{\infty} \dot H_x^1(I)}\lesssim 1, \;\; \Vert \widetilde{u} \Vert_{S(I)}\leq C
	(\eta),\;  \widetilde{u}(t_{j+1})-u(t_{j+1}) =0.
	\end{equation*} 
	Using \eqref{vdec}, \eqref{est:W1 widew} and $W^1\subset S$, there exists a large positive constant $C_*(\eta)$ such that 
	\begin{equation*}
	\begin{aligned}
	\Vert e \Vert_{N^1(I)}=\Vert \left(|x|^{-4}*|\widetilde{w}|^2\right)\widetilde{w} - \left(|x|^{-4}*|\widetilde{u}|^2\right)\widetilde{u}  \Vert_{N^1(I)}\lesssim C_*(\eta)\delta_{M}^{1/6} .
	\end{aligned}
	\end{equation*} 
	Imposing another smallness condition on $\delta_{M}$:
	\begin{equation}\label{con:scat 4}
	C_*(\eta)\delta_{M}^{1/6} \ll \nu_P \left( C(\eta),C(\eta),0 \right).
	\end{equation}
	By the above estimates, we can apply the above lemma. Hence there exists another large positive constant $C_{**}(\eta)$ such that
	\begin{equation}\label{appro wideu}
	\Vert \widetilde{u}-u \Vert_{L_{t}^{\infty} \dot{H}_{x}^{1}(t_{j+1},\infty)} \leq C_{**}(\eta)\delta_{M}^{1/6}.
	\end{equation}
	
	By \eqref{est:E widew}, we have $E(\widetilde{w})<E(W)$. Using Proposition \ref{prop:GS char} and Proposition \ref{prop:sign funct}, we have $K(\widetilde{w})>0$. Using \eqref{est:E widew}, we obtain
	\begin{equation}\label{est:conv}
	\Vert{\left(|x|^{-4}*|\widetilde{w}|^2\right)|\widetilde{w}|^2}\Vert_{L^1}=4E(\widetilde{w})-2K(\widetilde{w})<4E(\widetilde{w}) <\Vert W \Vert_{\dot H^1}^{2}-2c\eta^{2a}.
	\end{equation}
	Taking $\delta_{M}$ smaller if necessary, we have 
	\begin{equation}\label{con:scat 5}
	C_{**}(\eta)\delta_{M}^{1/6} \ll \eta^{a/2},
	\end{equation}
	which implies $\delta_{M} \ll \eta^{a/2}$. By \eqref{est:disp v1}, \eqref{vdec}, \eqref{con:scat 3}, \eqref{appro wideu}, \eqref{est:conv} and \eqref{con:scat 5}, we obtain 
	\begin{equation*}
	\begin{aligned}
	\Vert{\left(|x|^{-4}*|u(t)|^2\right)|u(t)|^2}\Vert_{L^1}\lesssim&\Vert{\left(|x|^{-4}*|u-\widetilde{u}|^2\right)|u-\widetilde{u}|^2}\Vert_{L^1}+\Vert{\left(|x|^{-4}*|\widetilde{u}|^2\right)|\widetilde{u}|^2}\Vert_{L^1}\\\lesssim&\Vert u-\widetilde{u}\Vert_{\dot H^1}^4+\Vert{\left(|x|^{-4}*|\widetilde{w}|^2\right)|\widetilde{w}|^2}\Vert_{L^1}+\Vert v \Vert_{L^{2d/(d-2)}}^4\\ \ll& \Vert W \Vert_{\dot H^1}^{2}-c\eta^{2a}.
	\end{aligned}
	\end{equation*}
	On the one hand, using the above estimate, 
	\begin{equation*}
	\begin{aligned}
	\Vert{\left(|x|^{-4}*|u(t_b)-W_{\theta,\sigma}|^2\right)|u(t_b)-W_{\theta,\sigma}|^2}\Vert_{L^1}\gtrsim&\Vert W\Vert_{\dot H^1}^2-\Vert{\left(|x|^{-4}*|u(t_b)|^2\right)|u(t_b)|^2}\Vert_{L^1}\\\geq&c\eta^{2a}.
	\end{aligned}
	\end{equation*}
	On the other hand, using Proposition \ref{prop:orth decom}, we have $\Vert u(t_b)-W_{\theta,\sigma}\Vert_{\dot H^1}=\Vert v(t_b)\Vert_{\dot H^1}\sim\widetilde{d}_{\mathcal{W}}(u(t_b))=\delta $, thus
	\begin{equation*}
	\begin{aligned}
	\Vert{\left(|x|^{-4}*|u(t_b)-W_{\theta,\sigma}|^2\right)|u(t_b)-W_{\theta,\sigma}|^2}\Vert_{L^1}\lesssim \Vert u(t_b)-W_{\theta,\sigma}\Vert_{\dot H^1}^4\sim \delta^4\ll \eta^{2a},
	\end{aligned}
	\end{equation*}
	which leads to a contradiction. 
	
	In conclusion, after fixing the constant $\delta_{M} >0$ such that  \eqref{con:scat 3}, \eqref{con:scat 4} and \eqref{con:scat 5} hold, the smallness conditions on $\delta,\epsilon$ in the case $\Theta=+1$ are \eqref{dep small}, \eqref{con:scat 1}, \eqref{con:scat 2} and \eqref{con:scat 3}, which determine $\delta_{B}$ and $\epsilon_{B}$.

	%
	%
	%
	%
	
	\section{Proof of Proposition \ref{prop:asyp behar}}\label{sect:asymp behr}	
	In this section, we prove Proposition \ref{prop:asyp behar}. Let $u$ be a solution of \eqref{equ:Hart} satisfying $u(\left[t_0,T_+(u)\right)\subset \mathcal{H}^{\epsilon}\setminus \widetilde{B}_{\delta}(\mathcal{W})$ for some $\epsilon\in (0,\epsilon_{B}(\delta))$. By Remark \ref{rmk:incln relat}, $u$ staying in $\check{\mathcal{H}}\cap\mathcal{H}^{\epsilon_S}$ on $[t_0,T_+(u))$, so $\Theta(u)\in\{\pm1\}$ is a constant. Moreover, Proposition \ref{prop:ps des} implies that $t_+(\delta')<T_+(u)$ for all $\delta'\in\left[\delta,\delta_{B}\right]$, so Proposition \ref{prop:one pass} yields $t_1\in I(u)$ such that 
	\begin{equation*}
	u(\left[t_1,T_+(u)  \right)) \subset  \mathcal{H}^{\epsilon}\setminus \widetilde{B}_{\delta_{B}}(\mathcal{W}).
	\end{equation*}
	Without loss of generality, we may assume $t_1=0$ by time translation.

	\subsection{Blow-up after ejection.} In the case of $\Theta(u)=-1$ and $u_0\in H_{rad}^{1}$, we prove that $T_+(u)<\infty$. Using Proposition \ref{prop:sign funct}, we have $-1=\Theta(u(t))=signK(u(t))$ on $[t_0,T_+(u))$, thus $K(u(t))<0$. By Proposition \ref{prop:vart est}, we have $-K(u(t))\geq \kappa(\delta_{B})$. Let $R\gg1$. We rewrite \eqref{est:V 2deri} in the following way
	
	\begin{equation}\label{est:V 2deri rad}
	\begin{aligned}
	\ddot{V}_R(t)=&8K(u)-4\langle|u_r|^2,f_{0,R}\rangle + \langle|u/R^{3}|^2,f_{1,R}\rangle \\+& 8\iint \left( 1-\frac{1}{2}\frac{R}{|x|} \varphi' \left( \frac{|x|}{R} \right) \right) \frac{x(x-y)}{|x-y|^6} |u(t,x)|^2|u(t,y)|^2 dxdy\\-&8\iint \left( 1-\frac{1}{2}\frac{R}{|y|} \varphi' \left( \frac{|y|}{R} \right) \right) \frac{y(x-y)}{|x-y|^6} |u(t,x)|^2|u(t,y)|^2 dxdy
	\end{aligned}
	\end{equation}
	with
	\begin{equation*}
	f_0:=2-\varphi''(r), \;\; f_1:=-\Delta\Delta\varphi.
	\end{equation*}
	By the property of $\varphi$, we have $suppf_{1,R}\subset \{R\leq |x|\leq 2R\}$ and $0\leq f_{0,R}$. Hence using the mass conservation, we obtain $
	\langle|u/R^{3}|^2,f_{1,R}\rangle \lesssim R^{-6}\Vert u_0\Vert_{L^2}^2
	$. As proved in \cite[lemma 5.3]{MWX:Hart}, we have
	\begin{equation*}
	\begin{aligned}
	&8\iint \left( 1-\frac{1}{2}\frac{R}{|x|} \varphi' \left( \frac{|x|}{R} \right) \right) \frac{x(x-y)}{|x-y|^6} |u(t,x)|^2|u(t,y)|^2 dxdy\\-&8\iint \left( 1-\frac{1}{2}\frac{R}{|y|} \varphi' \left( \frac{|y|}{R} \right) \right) \frac{y(x-y)}{|x-y|^6} |u(t,x)|^2|u(t,y)|^2 dxdy\\ &\lesssim \frac{1}{R^{\frac{4d-4}{d}}} \Vert u(t)\Vert_{L^2}^{\frac{4d-4}{d}} \Vert u(t)\Vert_{\dot H^1}^{\frac{4}{d}} + \frac{1}{R^2} \Vert u(t)\Vert_{L^2}^{2} \Vert u(t)\Vert_{\dot H^1}^{2}.
	\end{aligned}
	\end{equation*}
	Hence for $R\gg \Vert u_0\Vert_{L^2}/\kappa(\delta_{B})$ and $0<t<T_+(u)$, we have 
	\begin{equation}\label{est:V 2deri bd}
	-\ddot{V}_R(t)\geq -K(u(t)) \geq \kappa(\delta_{B})>0.
	\end{equation}
	
	Now assume for contradiction that $u$ exists for all time $t>0$, namely $T_+(u)=\infty$. Choosing $R\gg \Vert u_0\Vert_{L^2}/\kappa(\delta_{B})$, we have from \eqref{est:V 2deri bd}
	\begin{equation}\label{est:nablau}
	R\Vert u_r(t)\Vert_{L^2}\Vert u_0\Vert_{L^2} \gtrsim|2R\Im\int_{|x|\leq 2R}\bar{u}\nabla u\nabla\varphi(\frac{x}{R})dx|\geq -\dot{V}_R(t)\rightarrow \infty,
	\end{equation}
	as $t\rightarrow \infty$, hence 
	\begin{equation*}
	-K(u(t))=-4E(u)+ \Vert u_r(t)\Vert_{L^2}^2\rightarrow \infty.
	\end{equation*}
	So one can choose $T_1>0$ such that $-K(u(t))\sim \Vert u_r(t)\Vert_{L^2}^2$ for $t\geq T_1$. Using \eqref{est:V 2deri bd} and \eqref{est:nablau}, for $T_2\geq T_1$ and some absolute constant $c$, we have
	\begin{equation*}
	\begin{aligned}
	c_1R\Vert u_0\Vert_{L^2} \left(\Vert u_r(T_1)\Vert_{L^2}+\Vert u_r(T_2)\Vert_{L^2}\right)\geq& \dot{V}_R(T_1)-\dot{V}_R(T_2)=\int_{T_1}^{T_2}-\ddot{V}_R(t)dt\\\geq& \int_{T_1}^{T_2}-K(u(t))dt=c_2\int_{T_1}^{T_2}\Vert u_r(t)\Vert_{L^2}^2dt,
	\end{aligned}
	\end{equation*}
	which implies that
	\begin{equation*}
	R\Vert u_r(T_2)\Vert_{L^2} \geq -R\Vert u_r(T_1)\Vert_{L^2}+\frac{c}{\Vert u_0\Vert_{L^2}} \int_{T_1}^{T_2} \Vert u_r(t)\Vert_{L^2}^2 dt.
	\end{equation*}
	Therefore, defining
	\begin{equation*}
	f(t):= -R\Vert u_r(T_1)\Vert_{L^2}+\frac{c}{\Vert u_0\Vert_{L^2}} \int_{T_1}^{t} \Vert u_r(s)\Vert_{L^2}^2 ds,
	\end{equation*}
	we see that for large $t>T_1$, $f(t)$ is positive and $\partial_{t}f(t)\gtrsim f(t)^2$. Integrating this differential inequality yields $f(t)<0$, which leads to a contradiction. Therefore $T_+(u)<\infty$.

	\subsection{Scattering after ejection.} In the case $\Theta(u)=+1$, the proof of scattering uses arguments from \cite{MXZ:crit Hart:f rad,NakS:NLS}. Unlike the subcritical case, we have to take account of the scaling parameter and the fact that the maximal time interval of existence might be finite, even though the $\dot H^1$ norm is bounded by \eqref{unif bd Te+}. 
	
	Before stating and proving Lemma \ref{lem:crit exis}, we introduce the profile decomposition lemma in the spirit of the results of Keraani \cite{Ker}.
	\begin{lemma}(Profile decomposition).\label{lem:profile} Let $\{v_{0,n}\}_{n\geq 1}$ be a radial uniformly bounded sequence in $\dot H^1$. Then, there exists a subsequence of $v_{0,n}$, also denoted $v_{0,n}$, and
		\begin{enumerate}
			\item[$(1)$] For each $j\geq0$, there exists a radial profile $V_j$ in $\dot H^1$.
			\item[$(2)$] For each $j\geq0$, there exists a sequence of $(\sigma_{j,n},t_{j,n})$ with
			\begin{equation}\label{orth cond}
			|\sigma_{j,n}-\sigma_{j',n}|+|e^{-2\sigma_{j,n}}(t_{j,n}-t_{j',n})|\rightarrow\infty\;\;\text{as}\;\; n \rightarrow \infty \;\;\text{for}\;\; j\neq j'.
			\end{equation}
			\item[$(3)$] For each $k$, there exists a sequence of radial remainder $\gamma_{k,n}$ in $\dot H^1$, such that
			\begin{equation*}
			\begin{aligned}
			e^{-it\Delta}v_{0,n}=&\sum_{j=0}^k e^{-i(t+t_{j,n})\Delta}S_{-1}^{-\sigma_{j,n}}V_j +\gamma_{k,n}\\=& \sum_{j=0}^k e^{-it\Delta}S_{-1}^{-\sigma_{j,n}}e^{-is_{j,n}\Delta}V_j +e^{-it\Delta}\gamma_{k,n}(0)
			\end{aligned}
			\end{equation*}
			with
			\begin{equation}\label{asym}
			\lim_{k \to \infty}\limsup_{n \to \infty} \Vert \gamma_{k,n} \Vert_{S}=0, \;\;s_{j,n}:=t_{j,n}e^{-2\sigma_{j,n}},
			\end{equation}
			\begin{equation}\label{pathak}
			\Vert v_{0,n} \Vert_{\dot H^1}^2=\sum_{j=0}^k \Vert V_j \Vert_{\dot H^1}^2 +\Vert \gamma_{k,n}(0) \Vert_{\dot H^1}^2 +o_n(1),
			\end{equation}
			\begin{equation*}
			E(v_{0,n})=\sum_{j=0}^k E(e^{-is_{j,n}\Delta}V_j)+E(\gamma_{k,n}(0))+o_n(1).
			\end{equation*}
		\end{enumerate}
	\end{lemma}
	For any $A<E(W)+\epsilon_{S}^2$ and any $\delta\in \left( 0,\delta_{B} \right]$, let $\mathscr{S}(A,\delta)$ be the collection of solutions of \eqref{equ:Hart} such that
	\begin{equation*}
	E(u)\leq A,\; u(\left[0,T_+(u)  \right) )\subset \check{\mathcal{H}}\setminus \widetilde{B}_{\delta}(\mathcal{W}),\;\Theta(u(0))=+1.
	\end{equation*}
	Note that if $u\in\mathscr{S}(A,\delta)$, then $\Theta(u(t))=+1$.
	
	Define the minimal energy where uniform Strichartz bound fails.
	\begin{equation*}
	\begin{aligned}
	S(A,\delta):=&\sup_{u\in \mathscr{S}(A,\delta)} \Vert u \Vert_{S(0,T_{+}(u))},\\
	E_c(\delta):=&\sup\{A<E(W)+ \epsilon_{S}^2\;|\;S(A,\delta)<\infty \}.
	\end{aligned}
	\end{equation*}
	It is well known that $u\in S(0,\infty)$ implies the scattering as $t\to \infty$, see \cite{Caz:book}. For $0<A\ll 1$, $S(A,\delta)\lesssim A^{1/2}$ holds by the small data scattering (see \cite{Caz:book} for example). The result of \cite{MXZ:crit Hart:f rad} implies $E_c(\delta)\geq E(W)$. If $E_c(\delta)<E(W)+\epsilon_{S}^2$, then there exists a sequence of solutions $u_n \in \mathscr{S}(A_n,\delta)$ for some sequence of numbers $A_n \to E_c(\delta)$ such that $\Vert u_n \Vert_{S(0,T_{+}(u_n))} \to \infty$.
	
	Next we use the above sequence $u_n$ to prove the existence of the $\dot H^1$ radial solution $U_c$ to \eqref{equ:Hart} with $U_c \in \mathscr{S}(E_c(\delta),\delta_{B}), E(U_c)=E_c(\delta)$ and $\Vert U_c \Vert_{S(0,T_{+}(u_n))} \to \infty$. Moreover, we will show that this critical solution has a compactness property up to symmetries of this equation in Lemma \ref{lem:crit comp}.
	\begin{lemma}(Existence of a critical solution).\label{lem:crit exis} Let $\delta\in(0,\delta_{B})$. Suppose that $E_c(\delta)\leq E(W)+\epsilon^2$ for some $\epsilon$ such that
		\begin{equation*}	
		0<\epsilon<\min(\epsilon_{V}(\delta),\epsilon_{B}(\delta)),\;\epsilon\ll \min (\epsilon_{S},\delta,\sqrt{\kappa(\delta)}).
		\end{equation*}	
		Let $A_n \to E_c(\delta)$ and $u_n \in \mathscr{S}(A_n,\delta)$ satisfying $\Vert u_n \Vert_{S(0,T_{+}(u_n))} \to \infty$. Then 
		\begin{enumerate}
			\item[$(1)$] there exist $U_c \in \mathscr{S}(E_c(\delta),\delta_{B})$ satisfying $E(U_c)=E_c(\delta)$ and $\Vert U_c \Vert_{S(0,T_{+}(U_c))}=\infty $,
			\item[$(2)$] there exist $(\sigma_{n},s_n)\in \R^2$ such that $e^{is_n \Delta}S_{-1}^{\sigma_{n}}u_n(0)$ is strongly convergent in $\dot H^1_{rad}$.
		\end{enumerate}
	\end{lemma}
	\begin{remark}
		Note that once we have $U_c \in \mathscr{S}(E_c(\delta),\delta)$ with the other properties, then a time translation can yield another minimal element in $\mathscr{S}(E_c(\delta),\delta_{B})$ using the ejection and the one-pass lemmas. Please refer to the proof below for more details. 
	\end{remark}
	
	\begin{proof}
		Using the small data scattering in \cite[Remark 2.1]{MXZ:crit Hart:f rad}, we have 
		\begin{equation*}
		\Vert u_n\Vert_{\dot H^1}\gtrsim1 \;\;\text{or}\;\;\Vert u_n\Vert_{S(0,\infty)}\;\;\text{are uniformly small}.
		\end{equation*}	
		Using $u_n \in \mathscr{S}(A_n,\delta)$ and \eqref{unif bd Te+}, we have, for $t\in [0,T_+(u_n))$,
		\begin{equation*}
		\Vert u_n(t)\Vert_{\dot H^1}\lesssim 1.
		\end{equation*}
		Combined with $\Vert u_n \Vert_{S(0,T_{+}(u_n))} \to \infty$, we have $\Vert u_n(t)\Vert_{\dot H^1}\sim 1$.
		
		Apply Lemma \ref{lem:profile} to $u_n(0)$, we have
		\begin{equation*}
		\begin{aligned}
		u_n(0)=&\sum_{j=0}^k S_{-1}^{-\sigma_{j,n}}e^{-is_{j,n}\Delta}V_j +\gamma_{k,n}(0),\\ \Vert u_n(0)\Vert_{\dot H^1}^2=&\sum_{j=0}^k \Vert V_j\Vert_{\dot H^1}^2+\Vert \gamma_{k,n}(0)\Vert_{\dot H^1}^2+o_n(1),\\ E(u_n(0))=&\sum_{j=0}^k E(e^{-is_{j,n}\Delta}V_j)+E(\gamma_{k,n}(0))+o_n(1).
		\end{aligned}
		\end{equation*}
		We divide the proof into two steps:
		
		\noindent{\bf Step 1: There exists one profile that does not scatter.}
		After passing to a subsequence, we assume that $s_{j,n}\rightarrow s_{j,\infty}\in[-\infty,\infty]$. Let $U_j$ be the nonlinear profile associated with $(V_j,\{s_{j,m}\}_{m\geq1})$, then by \cite[Defination 2.2]{MXZ:crit Hart:f rad}, there exists an interval $I$, with $s_{j,\infty}\in I$ such that $U_j$ is the unique solution of \eqref{equ:Hart} in $I$ satisfying
		\begin{equation}\label{prof appr}
		\lim_{m \to \infty}\Vert U_j(s_{j,m})-e^{-is_{j,m}\Delta}V_j \Vert_{\dot H^1}=0.
		\end{equation}
		By \eqref{prof appr}, we have 
		\begin{equation}\label{eorth seq}
		E(u_n)=\sum_{j=0}^{k} E(U_j)+E(\gamma_{k,n}(0))+o_n(1).
		\end{equation}
		We also define $U_{j,n}(t):=S_{-1}^{-\sigma_{j,n}}U_j((t+t_{j,n})e^{-2\sigma_{j,n}})$. 
		
		Since $u_n \in \mathscr{S}(A_n,\delta)$, we have $+1=\Theta(u_n(t))=signK(u_n(t))$ by Proposition \ref{prop:sign funct}. Thus $K(u_n(t))\geq0$. By $\epsilon<\epsilon_{V}(\delta)$, $u_n \in \mathscr{S}(A_n,\delta)$ and Proposition \ref{prop:vart est}, we have, for $t\geq0$,
		\begin{equation*}
		K(u_n(t))\gtrsim\kappa(\delta)\gg\epsilon^2.
		\end{equation*}
		Combining the above estimates and the conservation of $G$ for the free equation, we have
		\begin{equation*}
		E(W)-\epsilon^2>E(u_n)-K(u_n(0))/4+\epsilon^2=G(u_n(0))+\epsilon^2\geq \sum_{j=0}^{k}G(V_j)+G(\gamma_{k,n}).
		\end{equation*}
		This implies
		\begin{equation*}
		G(\gamma_{k,n})<E(W),\;\;G(e^{-it\Delta}V_j)<E(W).
		\end{equation*}
		Using \eqref{char:Neh char}, we have $K(\gamma_{k,n}(t))>0$ and $K(e^{-it\Delta}V_j)>0$ for all $t\in\R$. Thus we have $E(\gamma_{k,n}(t))>0$, $K(e^{-it\Delta}V_j)\rightarrow \Vert V_j\Vert_{\dot H^1}^2$ as $t\rightarrow \pm\infty$ and both can be zero only if $V_j=0$. Hence by \eqref{prof appr}, for $t\in I$,
		\begin{equation}\label{noneg KE}
		K(U_j(t))\geq0 ,\;\; E(U_j(t))\geq0.
		\end{equation}
		Combined the above estimates with \eqref{eorth seq}, we have, for all $j\geq0$,
		\begin{equation*}
		E(U_j)\leq E(u_n)\leq E_c(\delta).
		\end{equation*}
		
		We have the following assertion:
		\begin{equation*}
		\text{If}\;\; E(U_j)<E(W),\;\;\text{then}\;\;\Vert U_j\Vert_{S\cap W^1\cap L_{t}^{\infty}\dot H^1}<\infty.
		\end{equation*}
		Indeed, using \eqref{char:Neh char} as well as \eqref{noneg KE}, we have $E(W)>G(U_j)$. That is, $\Vert U_j\Vert_{\dot H^1}^2<\Vert W\Vert_{\dot H^1}^2$. Then \cite[Theorem 1.1]{MXZ:crit Hart:f rad} implies that $U_j$ exists globally in time and scatters with $\Vert U_j\Vert_{S\cap W^1\cap L_{t}^{\infty}\dot H^1}<\infty$.
		
		Next, we prove there exists at least one $j\in\{0,1,\cdots,k\}$ such that $\Vert U_j\Vert_S=\infty$. We prove it by contradiction. If $\Vert U_j\Vert_S<\infty$ for all $j$, we want to apply Lemma \ref{lem:pert} to 
		\begin{equation*}
		\widetilde{u}:=\sum_{j=0}^{k} U_{j,n}+\gamma_{k,n},\;\;u:=u_n,\;\;t_0:=0,\;\;I:=\R.
		\end{equation*}
		\begin{enumerate}
			\item[$(1)$] $\widetilde{u}$ is bounded in $L_{t}^{\infty}\dot H_x^1$ as $n\rightarrow \infty$ uniformly in $k$ by \eqref{quan:KIG} and \eqref{eorth seq}. 
			\item[$(2)$] $\widetilde{u}$ is bounded in $S(\R)$ as $n\rightarrow \infty$ uniformly in $k$ by \eqref{orth cond}, \eqref{pathak} and a similar argument in the proof of Proposition 4.2 in \cite{MXZ:crit Hart:f rad}.
			\item[$(3)$] By \eqref{prof appr} and
			\begin{equation*}
			\begin{aligned}
			\widetilde{u}(0)=&\sum_{j=0}^{k} U_{j,n}(0)+\gamma_{k,n}(0)=\sum_{j=0}^k S_{-1}^{-\sigma_{j,n}}U_j(s_{j,n}) +\gamma_{k,n}(0),\\u(0)=&u_n(0)=\sum_{j=0}^k S_{-1}^{-\sigma_{j,n}}e^{-is_{j,n}\Delta}V_j +\gamma_{k,n}(0),
			\end{aligned}
			\end{equation*}
			we have $\Vert \widetilde{u}(0)-u_n(0)\Vert_{\dot H_x^1}\rightarrow0$ as $n\rightarrow \infty$. 
			\item[$(4)$] We only remain to prove
			\begin{equation*}
			i\widetilde{u}_{t}- \Delta\widetilde{u}-\left( |x|^{-4}*|\widetilde{u}|^2\right)\widetilde{u}=\sum_{j=0}^{k}\left( \left( |x|^{-4}*|U_{j,n}|^2\right)U_{j,n} \right)-\left( |x|^{-4}*|\widetilde{u}|^2\right)\widetilde{u}
			\end{equation*}
			is small in $N^1(\R)$ for large $n$. Indeed, using $\Vert U_j\Vert_{S}<\infty$, \eqref{orth cond} and \eqref{asym}, we obtain
			\begin{equation*}
			\lim_{k \to \infty}\limsup_{n \to \infty}\Vert i\widetilde{u}_{t}- \Delta\widetilde{u}-\left( |x|^{-4}*|\widetilde{u}|^2\right)\widetilde{u} \Vert_{N^1(\R)}=0.
			\end{equation*}
			For the proof, please refer to \cite[Proposition 4.2]{MXZ:crit Hart:f rad} and \cite{Ker}. 
		\end{enumerate}
		Thus, for large $k$ and large $n$, Lemma \ref{lem:pert} yields 
		\begin{equation*}
		\Vert u_n\Vert_{S(\R)}\lesssim1,
		\end{equation*}
		contradicting $\Vert u_n\Vert_{S(0,T_+(u_n))}\rightarrow\infty$.
		
		\noindent{\bf Step 2: Construction of the critical element $U_c$.} We assume $\Vert U_0\Vert_{S(I(U_0))}=\infty$ without loss of generality. Using the above assertion, we have $E(W)\leq E(U_0)$. Combined with \eqref{eorth seq}, \eqref{noneg KE} and $E(\gamma_{k,n}(0))>0$, we see that for all $j\in\{1,\cdots,k\}$,
		\begin{equation*}
		E(U_j)<\epsilon^2,\;\;E(\gamma_{k,n}(0))<\epsilon^2.
		\end{equation*}
		Using the assertion, $K(\gamma_{k,n}(0))>0$ and \eqref{noneg KE}, we have $T_{\pm}(U_j)=\infty$, $U_j$ scatters and
		\begin{equation}\label{est:pro re H1}
		\sum_{j=1}^{k}\Vert U_j\Vert_{ L_{t}^{\infty}\dot H^1}^2+\Vert \gamma_{k,n}\Vert_{\dot H^1}^2\lesssim\epsilon^2\ll 1.
		\end{equation}
		We apply Lemma \ref{lem:pert} to $\widetilde{u}$ and $u=u_n$ on $I=I_n:=e^{2\sigma_{0,n}}I_0 -t_{0,n}$, where $I_0$ is any subinterval of $I(U_0)$ satisfying $\Vert U_0\Vert_{S(I_0)}<\infty$. Since $B_3+\nu \rightarrow0$, we have
		\begin{equation*}
		\limsup_{n \to \infty} \Vert u_n\Vert_{S(I_n)}<\infty,\;\; \lim_{k \to \infty}\limsup_{n \to \infty} \Vert \widetilde{u}-u_n\Vert_{L_{t}^{\infty}\dot H_x^1(I_n)}=0.
		\end{equation*}
		By $ \Vert U_j \Vert_{L_{t}^{\infty}\dot H_x^1}\geq \Vert U_j(\frac{t+t_{j,n}}{e^{2\sigma_{j,n}}}) \Vert_{\dot H_x^1}= \Vert U_{j,n} \Vert_{\dot H_x^1}$ and \eqref{est:pro re H1}, we have
		\begin{equation*}
		\sum_{j=1}^{k} \Vert U_{j,n} \Vert_{\dot H_x^1}^2+\Vert \gamma_{k,n}(0) \Vert_{\dot H_x^1}^2\ll 1.
		\end{equation*}
		Using the orthogonality, we have
		\begin{equation*}
		\Vert \widetilde{u} \Vert_{\dot H_x^1}^2=\sum_{j=0}^{k} \Vert U_{j,n} \Vert_{\dot H_x^1}^2+\Vert \gamma_{k,n}(0) \Vert_{\dot H_x^1}^2+o_n(1).
		\end{equation*}
		Combining the above estimates, we see that 
		\begin{equation*}
		\Vert \widetilde{u} \Vert_{\dot H_x^1}^2= \Vert U_{0,n} \Vert_{\dot H_x^1}^2+o_n(1),
		\end{equation*}
		which implies
		\begin{equation}\label{seq prof app}
		\limsup_{n \to \infty} \Vert u_n-U_{0,n}\Vert_{L_{t}^{\infty}\dot H^1(I_n)}\lesssim\epsilon.
		\end{equation}
		
		Next, we prove $s_{0,\infty}<+\infty$ and $\Vert U_0 \Vert_{S(s_{0,\infty},T_+(U_0))}=\infty$. 
		
		Firstly, if $s_{0,\infty}=+\infty$, then $U_0$ scatters in a neighborhood of $+\infty$ by \eqref{prof appr}. That is, there exists $T_0\in I(U_0)$ such that $\Vert U_0 \Vert_{S([T_0,+\infty))}<\infty$. Hence, choosing $I_0=[T_0,+\infty)$ in the above argument yields a uniform bound on $\Vert u_n\Vert_{S(I_n)}$, but $I_n\supset[0,\infty)$ for large $n$, contradicting $\Vert u_n\Vert_{S(0,T_+(u_n))}\rightarrow\infty$. This implies $s_{0,\infty}<+\infty$. Now, there only remain two cases: $s_{0,\infty}=-\infty$ and $s_{0,\infty}\in\R$. 
		
		For the case $s_{0,\infty}=-\infty$: there exists $t_0\in I(U_0)$ such that $\Vert U_0 \Vert_{S(-\infty,t_0)}<\infty$ by \eqref{prof appr}. By $\Vert U_0\Vert_{S(I(U_0))}=\infty$, we have $\Vert U_0 \Vert_{S(t_0,T_+(U_0))}=\infty$, which implies $\Vert U_0 \Vert_{S(-\infty,T_+(U_0))}=\infty$.
		
		For the case $s_{0,\infty}\in\R$: if $\Vert U_0 \Vert_{S(s_{0,\infty},T_+(U_0))}<\infty$, then $T_+(U_0)=\infty$ by the blow-up criterion (\cite[Proposition 2.2]{MXZ:crit Hart:f rad}). Hence, choosing $I_0=(s_-,\infty)$ for some $s_-\in(-T_-(U_0),s_{0,\infty})$, we see that $I_n\supset[0,\infty)$ for large $n$, leading to a contradiction with $\Vert u_n\Vert_{S(0,\infty)}\rightarrow\infty$ as in the case $s_{0,\infty}=+\infty$ above. Therefore $\Vert U_0 \Vert_{S(s_{0,\infty},T_+(U_0))}=\infty$.
		
		There exists $t_0\in (s_{0,\infty},T_+(U_0))$, such that $\Vert U_0 \Vert_{S(t_0,T_+(U_0))}=\infty$. Since $E(U_0)\leq E_c(\delta)\leq E(W)+\epsilon^2$ and $\epsilon<\epsilon_{B}(\delta_{B})$, Propositions \ref{prop:one pass} and \ref{prop:ps des} imply that there are only two options for $U_0$:
		\begin{enumerate}
			\item[$(1)$] There exists $t_+\in I(U_0)$ such that $\widetilde{d}_{\mathcal{W}}(U_0(t))\geq\delta_{B}$ for all $t_+\leq t<T_+(U_0)$.
			\item[$(2)$] $\limsup_{t\nearrow T_+(U_0)} \widetilde{d}_{\mathcal{W}}(U_0(t))\leq \epsilon/c_D$.
		\end{enumerate}
		For the second case, choosing $I_0=(t_0,t_c)\subset I(U_0)$ such that $\widetilde{d}_{\mathcal{W}}(U_0(t_c))<2\epsilon/c_D$. From Proposition \ref{prop:nlr dist} and \eqref{seq prof app}, we have
		\begin{equation*}
		\widetilde{d}_{\mathcal{W}}(u_n(t_n))\lesssim\widetilde{d}_{\mathcal{W}}(U_0(t_c))+\Vert u_n(t_n)-U_{0,n}(t_n)\Vert_{\dot H^1}\lesssim\epsilon\ll \delta,
		\end{equation*}
		where $t_n:=e^{2\sigma_{0,n}}t_c - t_{0,n}>0$ for large $n$, since $t_c
		>s_{0,\infty}$. This contradicts $\widetilde{d}_{\mathcal{W}}(u_n(t))\geq\delta$ on $[0,T_+(u_n))$. Therefore the second case is excluded, that is $\widetilde{d}_{\mathcal{W}}(U_0(t))\geq\delta_{B}$ for $t_+\leq t<T_+(U_0)$. Combined with $\epsilon_{B}(\delta_{B})\leq \epsilon_{S}$ and $\epsilon_{B}(\delta_{B})<c_D\delta_{B}$, we have $U_0(t)\in \check{\mathcal{H}}\cap\mathcal{H}^{\epsilon_S}$ on $[t_+,T_+(U_0))$. Applying Proposition \ref{prop:sign funct}, we have $\Theta(U_0(t))$ is constant on $[t_+,T_+(U_0))$. Let $t_n:=e^{2\sigma_{0,n}}t_+ - t_{0,n}$, using the scaling invariance of $\Theta$, we have 
		\begin{equation*}
		\Theta(U_{0,n}(t_n))=\Theta(U_0(\frac{t_n+t_{0,n}}{e^{2\sigma_{0,n}}}))=\Theta(U_0(t_+)).
		\end{equation*}
		Since $\epsilon< \min(\epsilon_{S},c_D\delta)$, we have $u_n(t_n)\in\check{\mathcal{H}}\cap\mathcal{H}^{\epsilon_S}$. \eqref{seq prof app} implies that $u_n(t_n)$ is in an $O(\epsilon)$ ball around $U_{0,n}(t_n)$ for large $n$. Hence
		\begin{equation*}
		+1=\Theta(u_n(t_n))=\Theta(U_{0,n}(t_n))=\Theta(U_0(t_+)).
		\end{equation*}
		Therefore, $U_0\in\mathscr{S}(E_c(\delta),\delta)$. Putting $U_c(t):=U_0\left(t+\widetilde{t}\right)$ with $\widetilde{t}:=\max(t_0,t_+)$, we obtain
		\begin{equation*}
		U_c\in\mathscr{S}(E_c(\delta),\delta_{B}),\;\;\Vert U_c \Vert_{S(0,T_{+}(U_c))}=\infty.
		\end{equation*}
		By the defination of $E_c(\delta)$, we have $E(U_c)\geq E_c(\delta)$, thus $E(U_c)=E_c(\delta)$. $U_c$ satisfies all the properties in the lemma.
		
		Since $E(u_n)\rightarrow E_c(\delta)=E(U_0)$, by \eqref{eorth seq}, \eqref{noneg KE} and $E(\gamma_{k,n}(0))>0$, we have, for all $j\in\{1,\cdots,n\}$,
		\begin{equation*}
		E(U_j)=0,\;\;\lim_{n \to \infty}E(\gamma_{k,n}(0))=0.
		\end{equation*}
		Using \eqref{quan:KIG}, \eqref{noneg KE} and $K(\gamma_{k,n}(0))>0$, we have, for all $j\in\{1,\cdots,n\}$,
		\begin{equation*}
		U_j=0,\;\;\gamma_{k,n}(0)\rightarrow0\;\;\text{in}\;\;\dot H^1.
		\end{equation*}
		Combining above estimates with \eqref{prof appr}, we have
		\begin{equation*}
		u_n(0)=\sum_{j=0}^k S_{-1}^{-\sigma_{j,n}}U_j(s_{j,n}) +o_n(1)  =S_{-1}^{-\sigma_{0,n}}e^{-is_{0,n}\Delta}V_0 +o_n(1).
		\end{equation*}
		Hence $e^{is_n\Delta}S_{-1}^{\sigma_n}u_n(0)\rightarrow V_0$ with $\sigma_n:=\sigma_{0,n}$ and $s_n:=s_{0,n}$ in $\dot H^1$.
	\end{proof}
	
	\begin{lemma}(Pre-compactness of the flow of the critical solution).\label{lem:crit comp} Let $U_c$ be as in Lemma \ref{lem:crit exis}, then there exists $\sigma_{c}:\left[0,T_+(U_c) \right)\to \R$, such that $\mathscr{K}$ is precompact in $\dot H_{rad}^1$ where
		\begin{equation*}
		\mathscr{K}:=\{v(t,x)\;|\;v(t,x)=S_{-1}^{-\sigma_{c}(t)}U_c(t),\;0\leq t<T_+(U_c) \}.
		\end{equation*}
	\end{lemma}
	\begin{proof}[Sketch of proof]
		We can refer to Proposition 4.2 in \cite{MXZ:crit Hart:f rad} for more details.
	\end{proof}
	Armed with the above lemmas, we are ready for the final step of the proof of Proposition \ref{prop:asyp behar}.
	\begin{lemma} $U_c$ does not exist.
	\end{lemma}
	\begin{proof}
		For all $t\in\left[0,T_+(U_c) \right)$, we have $\widetilde{d}_{\mathcal{W}}(U_c(t))\geq \delta_{B}$ and $\Vert U_c(t) \Vert_{\dot H^1}\gtrsim 1$ by the small data scattering. Using Proposition \ref{prop:sign funct}, we have 
		\begin{equation*}
		+1=\Theta(U_c(t))=signK(U_c(t)),
		\end{equation*}
		thus $K(U_c(t))\geq0$. Since $\epsilon<\epsilon_{V}(\delta_{B})$ and $E(U_c)=E_c(\delta)$, using Proposition \ref{prop:vart est}, we have $K(U_c(t))>0$. Thus
		\begin{equation*}
		\widetilde{\kappa}:=\inf_{t\geq 0}K(U_c(t))>0.
		\end{equation*}
		We divide the proof into three steps:
		
		\noindent{\bf Step 1: $U_c$ exists for all time $t>0$.} We prove it by contradiction. Using the local well-posedness theory and Lemma \ref{lem:crit comp}, we obtain that blow-up is possible only by concentration $\sigma_{c}(t)\to \infty$ as $t\nearrow T_+(U_c)<\infty$, see \cite{MXZ:crit Hart:f rad} for a proof. For any $m>0$ and $t\in I(U_c)$, put 
		\begin{equation*}
		y_m(t):=\langle |U_c(t)|^2,\phi_m \rangle.
		\end{equation*}
		Since $|S_{-1}^{-\sigma_{c}}U_c|^2$ is precompact in $L_{x}^{d/(d-2)}$, $e^{-2\sigma_{c}}\phi_{me^{\sigma_{c}}} \to 0$ weakly in $L_{x}^{d/2}$, we have
		\begin{equation*}
		y_m(t)=\langle |S_{-1}^{-\sigma_{c}}U_c|^2,e^{-2\sigma_{c}}\phi_{me^{\sigma_{c}}} \rangle \to 0,\;\;\text{as}\;\;t\nearrow T_+(U_c).
		\end{equation*}
		Using $\partial_{t}|u|^2=2\Im\left( \nabla\left(\nabla u\cdot \bar{u}\right) \right)$, we have
		\begin{equation*}
		\begin{aligned}
		\dot{y}_m (t)=&\int \partial_{t} |U_c(t)|^2\phi_m dx=2\Im\int \phi_m \nabla\left(\nabla U_c\cdot \bar{U_c}\right) dx\\=&-2\Im\int \nabla \phi_m\nabla U_c\bar{U_c} dx=-2\langle \nabla U_c,iU_c\nabla\phi_m \rangle.
		\end{aligned}
		\end{equation*}
		By Hardy's inequality, we have 
		\begin{equation*}
		-\dot{y}_m (t)\leq|\dot{y}_m (t)|\leq 2\int |\nabla U_c||U_c||\nabla \phi_m|dx\lesssim \Vert U_c \Vert_{\dot H^1}^2 \lesssim1.
		\end{equation*}
		Integrating it on $t<T_+(U_c)$ and using $y_m(t)\rightarrow0$ as $t\nearrow T_+(U_c)$, we have
		\begin{equation*}
		\int |U_c(t)|^2 \phi_mdx=y_m(t)\lesssim|T_+(U_c)-t|.
		\end{equation*}
		Sending $m\to \infty$, we obtain
		\begin{equation*}
		\Vert U_c(t) \Vert_{L^2}^2 \lesssim |T_+(U_c)-t|,
		\end{equation*}
		which implies that $U_c(0)\in L_{x}^2$. By the $L^2$ conservation, we get
		\begin{equation*}
		\Vert U_c(0) \Vert_{L^2}=\Vert U_c(t) \Vert_{L^2}\to 0\;\;\text{as}\;\;t\nearrow T_+(U_c).
		\end{equation*}
		So $U_c=0$, which contradicts with $T_+(U_c)<\infty$.
		
		\noindent{\bf Step 2: The scale bound $\sigma_{c}$ satisfies $\inf_{0\leq t<\infty}\sigma_{c}(t)=-\infty$.} We prove it by contradiction. Define $A:=\inf_{0\leq t<\infty}\sigma_{c}(t)$, then if $A>-\infty$, using Hardy's inequality and Lemma \ref{lem:crit comp}, for all $t\in \left[0,\infty \right)$, there exists $R$ such that 
		\begin{equation}\label{est:crie viri}
		\int_{R\leq|x|}|\nabla U_c|^2 + | U_c/r^3|^2dx\lesssim \Vert U_c\Vert_{\dot H^1(R\leq|x|)}^2+\frac{1}{R^4}\Vert U_c\Vert_{\dot H^1(R\leq|x|)}^2   \ll \widetilde{\kappa}.
		\end{equation}
		Applying \eqref{est:V 2deri rad} to $U_c$, integrating it on $[0,T]$, we have
		\begin{equation*}
		[\dot{V}_R(t)]_0^T\geq 8T\widetilde{\kappa}-4\int_{0}^{T} \langle|\partial_{r}(U_c)|^2,f_{0,R}\rangle dt+\int_{0}^{T} \langle|U_c/R^{3}|^2,f_{1,R}\rangle dt \\-  \int_{0}^{T} |H|dt,
		\end{equation*}
		where
		\begin{equation*}
		\begin{aligned}
		H:=&8\iint \left( 1-\frac{1}{2}\frac{R}{|x|} \varphi' \left( \frac{|x|}{R} \right) \right) \frac{x(x-y)}{|x-y|^6} |U_c(x)|^2|U_c(y)|^2 dxdy\\-&8\iint \left( 1-\frac{1}{2}\frac{R}{|y|} \varphi' \left( \frac{|y|}{R} \right) \right) \frac{y(x-y)}{|x-y|^6} |U_c(x)|^2|U_c(y)|^2 dxdy.
		\end{aligned}
		\end{equation*}
		Combined with \eqref{est:crie viri}, we have
		\begin{equation*}
		\begin{aligned}
		&-\int_{0}^{T} \langle|\frac{U_c}{R^3}|^2,f_{1,R}\rangle dt\leq \bigg|\int_{0}^{T}\int_{R\leq|x|\leq 2R}|\frac{U_c}{R^3}|^2f_{1,R}dxdt\bigg|\lesssim \int_{0}^{T}\int_{R\leq|x|\leq 2R}\frac{|U_c|^2}{r^6}dxdt\ll \widetilde{\kappa}T,\\ &4\int_{0}^{T} \langle|\partial_{r}(U_c)|^2,f_{0,R}\rangle dt\leq\big|4\int_{0}^{T}\int_{R\leq|x|}|\partial_{r}(U_c)|^2 f_{0,R}dxdt\big|\lesssim\int_{0}^{T}\int_{R\leq|x|}|\nabla U_c|^2dxdt\ll \widetilde{\kappa}T,\\&|H|\lesssim \frac{1}{R^{\frac{4d-4}{d}}} \Vert U_c\Vert_{L^2}^{\frac{4d-4}{d}} \Vert U_c\Vert_{\dot H^1}^{\frac{4}{d}} + \frac{1}{R^2} \Vert U_c\Vert_{L^2}^{2} \Vert U_c\Vert_{\dot H^1}^{2}.
		\end{aligned}
		\end{equation*}
		Thus
		\begin{equation*}
		[\dot{V}_R(t)]_0^T\geq \widetilde{\kappa}T.
		\end{equation*}
		Since $\dot{V}_R(t)$ is bounded for $t\to \infty$, the above estimate leads to  a contradiction as $T\rightarrow\infty$.
		
		\noindent{\bf Step 3: The critical element does not exist.} 
		By the continuity of $U_c(t)$, we deduce that there exists sequence $t_n$, $\sigma_{c}(t_n)\to -\infty$ as $t_n\to \infty$ satisfying
		\begin{equation}\label{scalbd appr}
		\min_{0\leq s\leq t_n}\sigma_{c}(s)=\sigma_{c}(t_n).
		\end{equation}
		By Lemma \ref{lem:crit comp}, we may assume that $S_{-1}^{-\sigma_{c}(t_n)}U_c(t_n)$ converges strongly in $\dot{H}_{rad}^{1}$. Let $U_n$ and $U_{\omega}$ be the solution of \eqref{equ:Hart} with the initial data
		\begin{equation*}
		U_n(0)=S_{-1}^{-\sigma_{c}(t_n)}U_c(t_n),\;U_{\omega}(0)=\lim_{n \to \infty}U_n(0).
		\end{equation*}
		The local well-posedness theory in \cite[Proposition 2.1]{MXZ:crit Hart:f rad} implies that for any compact $J\subset I(U_{\omega})$,
		\begin{equation*}
		U_n\to U_{\omega}\;\;\text{in}\;\;C(J;\dot{H}_{x}^{1})\cap S(J)\;\;\text{as}\;\;n\to \infty.
		\end{equation*}
		Since 
		\begin{equation}\label{conve in S}
		\Vert U_n \Vert_{S(-t_n e^{2\sigma_{c}(t_n)},0)}=\Vert U_c \Vert_{S(0,t_n)} \to \Vert U_c \Vert_{S(0,\infty)}=\infty,
		\end{equation}
		which implies that for each $t\in \left(-T_-(U_{\omega}),0\right]$ and large $n$, we have
		\begin{equation*}
		|t|<t_n e^{2\sigma_{c}(t_n)}.
		\end{equation*}
		Putting $s_n:=t_n-|t|e^{-2\sigma_{c}(t_n)}\in \left(0,t_n \right]$, by the scale invariance, Lemma \ref{lem:crit comp} and $\widetilde{d}_{\mathcal{W}}(U_c(t))\geq \delta_{B}$ for $t\in[0,\infty)$, we have
		\begin{equation*}
		S_{-1}^{\sigma_{c}(t_n)-\sigma_{c}(s_n)}U_n(t)=S_{-1}^{-\sigma_{c}(s_n)}U_c(s_n)\in \mathscr{K}\setminus \widetilde{B}_{\delta_{B}}(\mathcal{W}).
		\end{equation*}
		Since $U_n(t)\to U_{\omega}(t)$ in $\dot{H}^1$, $\mathscr{K}$ is precompact and \eqref{scalbd appr}, after passing to a subsequence, we assume that there exists  $\sigma_{\omega}(t)\in \left[0,\infty \right)$ such that
		\begin{equation*}
		\sigma_{c}(s_n)-\sigma_{c}(t_n)\rightarrow \sigma_{\omega}(t).
		\end{equation*}
		Then it is obvious that
		\begin{equation*}
		\{ S_{-1}^{-\sigma_{\omega}(t)}U_{\omega}(t) \}_{t\in \left(-T_-(U_{\omega}),0\right]}\subset \bar{\mathscr{K}},
		\end{equation*}
		which implies $\{ S_{-1}^{-\sigma_{\omega}(t)}U_{\omega}(t) \}_{t\in \left(-T_-(U_{\omega}),0\right]}$ is precompact. For all $t\in \left(-T_-(U_{\omega}),0 \right]$, $\widetilde{d}_{\mathcal{W}}(U_{\omega}(t))\geq \delta_{B}$.
		
		Moreover $\Vert U_{\omega} \Vert_{S(-T_-(U_{\omega}),0)}=\infty$. Indeed, if $\Vert U_{\omega} \Vert_{S(-T_-(U_{\omega}),0)}<\infty$, then the blow-up criterion yields $T_-(U_{\omega})=\infty$. Apply the Lemma \ref{lem:pert} to $u=U_n$, $I=(-\infty,0]$ yields
		\begin{equation*}
		\Vert U_n \Vert_{S(-\infty,0)} \lesssim1,
		\end{equation*}
		which contradicts with \eqref{conve in S}.
		
		Now, we have proved that $U_{\omega}$ is a critical element with the scale bound $\sigma_{\omega}\geq 0$. Define $\widetilde{A}:=\inf_{0\leq t<\infty}\sigma_{\omega}(t)\geq 0$, while we can obtain $\widetilde{A}=-\infty$ using the above argument in step 2. This leads to a contradiction, which proves that critical element does not exist.
	\end{proof}
	This proves Proposition \ref{prop:asyp behar} in the case $\Theta(u)=+1$.
	

	\section{Proof of Theorem \ref{thm:above thresh}}\label{sect:pf main thm}
	Armed with the above propositions, we are now ready to prove Theorem \ref{thm:above thresh}.
	\begin{proof}[Proof of Theorem \ref{thm:above thresh}]
		Using $c_D$, $\delta_{B}$, $\epsilon_{*}$ in Proposition \ref{prop:nlr dist}, \ref{prop:one pass} and \ref{prop:asyp behar}, define 
		\begin{equation*}
		\epsilon_\star:=\epsilon_{*}(\delta_{B})\in(0,1),\; \mathcal{X}_{\epsilon}:=\{\varphi\in \mathcal{H}^{\epsilon}\;|\;c_D\widetilde{d}_{\mathcal{W}}(\varphi)\leq\epsilon  \}.
		\end{equation*}
		Then $\mathcal{X}_{\epsilon}$ is relatively closed in $\mathcal{H}^{\epsilon}$, $\mathcal{W} \subset \mathcal{X}_\epsilon \subset \widetilde{B}_{2\epsilon/ c_D}(\mathcal{W})$. Since $\epsilon_\star \leq \epsilon_{S}$, for all $\epsilon\in \left(0,\epsilon_\star  \right]$,
		\begin{equation*}
		\varphi\in \mathcal{H}^{\epsilon}\setminus \mathcal{X}_{\epsilon} \Longrightarrow E(\varphi)-E(W)<\epsilon^2<\left(c_D \widetilde{d}_{\mathcal{W}}(\varphi) \right)^2\Longrightarrow\varphi \in \check{\mathcal{H}}\cap \mathcal{H}^{\epsilon_{S}}.
		\end{equation*}		
		By Proposition \ref{prop:sign funct}, the functional $\Theta$ can be defined on $\mathcal{H}^{\epsilon}\setminus \mathcal{X}_{\epsilon}$ for all $\epsilon\in \left(0,\epsilon_\star  \right]$.
		
		Next we consider the dynamics. Let $\epsilon\in \left(0,\epsilon_\star  \right]$ and let $u\in \mathcal{H}^{\epsilon}$ be a solution of \eqref{equ:Hart}. Let $I_0(u):=\left\{t\in I(u)\; |\;  u(t)\in\mathcal{X}_\epsilon \right\}$ and let $I_C(u):=I(u)\setminus I_0(u)$. Take any $\delta\in \left(0,\delta_{B} \right]$ satisfying $\epsilon \leq \epsilon_{*}(\delta)$.
		
		\noindent{\bf Case1:} Suppose that there exists $t_0\in I_C(u)$ such that $\widetilde{d}_{\mathcal{W}}\left(u(t_0)\right)<\delta$.
		\begin{enumerate}
			\item[(1)] If $\partial_{t}\widetilde{d}_{\mathcal{W}}\left(u(t_0)\right)\geq 0$.
			
			Since $\delta\leq\delta_{B}<\delta_{X}$, $u(t_0)\in \mathcal{H}^{\epsilon}$ and Proposition \ref{prop:eject mod}, $\widetilde{d}_{\mathcal{W}}\left(u(t)\right)$ is increasing until it reaches $\delta_{X}$ at some $t_X \in \left(t_0,T_+(u)\right)$. Since $\widetilde{d}_{\mathcal{W}}\left(u(t_0)\right)<\delta<\delta_X$, there exists $t'\in \left(t_0,t_X \right)$ such that $\widetilde{d}_{\mathcal{W}}\left(u(t')\right)=\delta$. Then Proposition \ref{prop:one pass} implies that $\widetilde{d}_{\mathcal{W}}\left(u(t)\right)>\delta$ for all $t\in \left(t',T_+(u) \right)$. Hence $\widetilde{d}_{\mathcal{W}}\left(u(t)\right)\geq \widetilde{d}_{\mathcal{W}}\left(u(t_0)\right)$ for all $t\in \left[t_0,T_+(u)\right)$, which implies $\left[t_0,T_+(u)\right)\subset I_C(u)$ by $\widetilde{d}_{\mathcal{W}}\left(u(t_0)\right)> \epsilon/c_D$. This implies $I_0(u)$ is either empty or an interval. By Proposition \ref{prop:asyp behar}, 
			\begin{equation*}
			\left\{\begin{aligned}
			&\text{If}\;\;\Theta(u(t))=+1\;\;\text{as}\;\;t\rightarrow T_+(u)\Longrightarrow u_0\in S_+\\
			&\text{If}\;\;\Theta(u(t))=-1\;\;\text{and}\;\;u_0\in L_x^2(\R^d)\;\;\text{as}\;\;t\rightarrow T_+(u)\Longrightarrow u_0\in B_+
			\end{aligned}\right.
			\end{equation*}	
			
			\item[(2)] If $\partial_{t}\widetilde{d}_{\mathcal{W}}\left(u(t_0)\right)\leq 0$.
			
			Then the time reversed version of the above argument implies that $\left(-T_-(u),t_0 \right)\subset I_C(u)$. Therefore, $I_0(u)$ is either empty or an interval. By Proposition \ref{prop:asyp behar}, 
			\begin{equation*}
			\left\{\begin{aligned}
			&\text{If}\;\;\Theta(u(t))=+1\;\;\text{as}\;\;t\rightarrow T_-(u)\Longrightarrow u_0\in S_-\\
			&\text{If}\;\;\Theta(u(t))=-1\;\;\text{and}\;\;u_0\in L_x^2(\R^d)\;\;\text{as}\;\;t\rightarrow T_-(u)\Longrightarrow u_0\in B_-
			\end{aligned}\right.
			\end{equation*}	
		\end{enumerate}	
		
		\noindent{\bf Case2:} Suppose that there exists $t_1\in I_0(u)$ and $t_2\in I_C(u)$. 
		\begin{enumerate}
			\item[(1)] If $t_1 <t_2$.
			
			Since $\widetilde{d}_{\mathcal{W}}(u)\leq \epsilon/c_D<\delta$ on $I_0(u)$, we may assume $\widetilde{d}_{\mathcal{W}}\left(u(t_2)\right)<\delta$ by decreasing $t_2$ if necessary. Then the above argument works with $t_0:=t_2$, we can prove $\partial_{t}\widetilde{d}_{\mathcal{W}}\left(u(t_2)\right)\geq 0$. Indeed, if $\partial_{t}\widetilde{d}_{\mathcal{W}}\left(u(t_2)\right)<0$, then by (2) of case 1, we have $\left(-T_-(u),t_2 \right)\subset I_C(u)$, which implies $I_0(u)\subset[t_2, T_+(u))$. Since $t_1\in I_0(u)$, we have $t_1\geq t_2$, which leads to a contradiction. Therefore,  $\left[t_2,T_+(u)\right)\subset I_C(u)$, which implies $I_0(u)$ is either empty or an interval. Then by Proposition \ref{prop:asyp behar}, $\Theta(u(t))$ on $\left[t_2,T_+(u)\right)\subset I_C(u)$ decides the behavior of $u$ towards $T_+(u)$.
			\item[(2)] If $t_1 >t_2$.
			
			In the same way, we deduce that $\partial_{t}\widetilde{d}_{\mathcal{W}}\left(u(t_2)\right)\leq 0$ and $(-T_-(u),t_2)\subset I_C(u)$. Therefore, $I_0(u)$ is either empty or an interval. By Proposition \ref{prop:asyp behar}, $\Theta(u(t))$ on $(-T_-(u),t_2)\subset I_C(u)$ decides the behavior of $u$ towards $-T_-(u)$.
		\end{enumerate}	
		
		\noindent{\bf Case3:} It only remains to consider the following case: $I_0(u)=\varnothing$ and for any $t_0\in I_C(u)(=I_0(u))$, $\widetilde{d}_{\mathcal{W}}\left(u(t_0)\right)\geq\delta$.	
		
		By Proposition \ref{prop:asyp behar}, we complete the investigation around $T_+(u)$, and the behavior towards $-T_-(u)$ is treated in the same way. 
		
		This proves Theorem \ref{thm:above thresh}.  	
	\end{proof}

	\section{Proof of Theorem \ref{thm:exist}}\label{sect:pf exist}
	
	In this section, we prove Theorem \ref{thm:exist}.
	\begin{proof}[Proof of Theorem \ref{thm:exist}]
		Let $0<\beta\ll \epsilon$ and $R>0$ be such that $\Vert \phi_{R}^{C}W \Vert_{\dot H^1}^2 \leq \beta^4$. We consider four solutions $u$ around $\mathcal{W}$ with the following initial data at $t=0$ in the coordinate \eqref{coor arou W} with $\vec{\lambda}=(\lambda_{1},\lambda_{2})$,
		\begin{equation*}
		\begin{aligned}
		\gamma(0):=&-P_{B}^{\perp}\phi_{R}^{C}W+\omega\left(\phi_{R}^{C}W,g_-
		\right)g_{+} -  \omega\left(\phi_{R}^{C}W,g_+\right)g_{-},\\
		\vec{\lambda}(0)=&\beta(\pm1,0),\;\; \beta(0,\pm1),
		\end{aligned}
		\end{equation*}
		where $B$ satisfies $\omega(g_{\pm}, B)=0$, $P_{B}^{\perp}$ denotes the orthogonal projection onto $\{B\}^{\perp}$. Note that $\omega(g_+, g_-)=1$, we have
		\begin{equation*}
		\omega(\gamma(0), g_{\pm})=\omega(-P_{B}^{\perp}\phi_{R}^{C}W, g_{\pm})+\omega(\phi_{R}^{C}W, g_{\pm}).
		\end{equation*}
		Since $\omega(\phi_{R}^{C}W, g_{\pm})=\omega(P_{B}^{\perp}\phi_{R}^{C}W, g_{\pm})$, we have $\omega(\gamma(0), g_{\pm})=0$. Thus $\gamma(0)$ is the symplectic projection of $-P_{B}^{\perp}\phi_{R}^{C}W$ to the subspace that is perpendicular (with respect to $\omega$) to $span\{g_-,g_+\}$. This ensures $u(0)\in L_{x}^{2}$ so that we can apply the above blow-up result.
		
		Let $I_E(u)\subset I(u)$ be the maximal interval satisfying $u(t)\in B_{\delta_{E}}(\mathcal{W})$. Thus we can use coordinate \eqref{coor arou W} on $I_E(u)$. For brevity, put $\widetilde{d}(t):=\widetilde{d}_{\mathcal{W}}(u(t))$ on $I_E(u)$.  By Proposition \ref{prop:nlr dist} and the same argument as in \eqref{est:parl uW minu}-\eqref{est:uW minu}, we get
		\begin{equation*}
		\begin{aligned}
		\Vert \gamma(t) \Vert_{\dot H^1}^2+O(|\lambda_{1}|^2\Vert\gamma \Vert_{\dot H^{1}})\sim& E(u(0))-E(W)-E_{\gamma^{\bot}}(u(0))+\int_{0}^{\tau}\partial_{\tau}\left(E(u)-E(W)-E_{\gamma^{\bot}}(u)\right)\\\lesssim& E(u(0))-E(W)-E_{\gamma^{\bot}}(u(0))+\int_{0}^{\tau}\lambda_{1}^2\Vert\gamma \Vert_{\dot H^{1}}+\lambda_{1}^3\\\lesssim&\beta^4+\int_{0}^{\tau} \widetilde{d}^2\Vert \gamma \Vert_{\dot H^{1}}+\widetilde{d}^3,
		\end{aligned}
		\end{equation*}
		which implies 
		\begin{equation*}
		\Vert \gamma(t) \Vert_{\dot H^1}^2+O(\widetilde{d}(t)^4)\lesssim \beta^4+\int_{0}^{\tau} \widetilde{d}^2\Vert \gamma \Vert_{\dot H^{1}}+\widetilde{d}^3  
		\end{equation*}
		within $I_E(u)$. Hence on any interval $J\subset I_E(u)$ where $|\tau|\leq \delta_{E}^{-1/2}$, we have 
		\begin{equation}\label{est:H1 gam}
		\begin{aligned}
		\Vert \gamma \Vert_{L_{t}^{\infty}(J;\dot H^1)}^2\lesssim&\beta^4+\left( \Vert\widetilde{d}\Vert_{L_{t}^{\infty}(J)}^2\Vert\gamma\Vert_{L_{t}^{\infty}(J;\dot H^1)}+\Vert\widetilde{d}\Vert_{L_{t}^{\infty}(J)}^3\right)\delta_{E}^{-1/2}\\\lesssim&\Vert\vec{\lambda}\Vert_{L_{t}^{\infty}(J)}^3+\left( \Vert\gamma\Vert_{L_{t}^{\infty}(J;\dot H^1)}^2+\Vert\widetilde{d}\Vert_{L_{t}^{\infty}(J)}^3\right)\delta_{E}^{-1/2}\\\lesssim&\Vert\vec{\lambda}\Vert_{L_{t}^{\infty}(J)}^3\left(1+\delta_{E}^{-1/2}\right)+\Vert\gamma\Vert_{L_{t}^{\infty}(J;\dot H^1)}^2\delta_{E}^{-1/2}\\\lesssim&\Vert\vec{\lambda}\Vert_{L_{t}^{\infty}(J)}^3\left(1+\delta_{E}^{-1}\right)\lesssim\Vert\vec{\lambda}\Vert_{L_{t}^{\infty}(J)}^3.
		\end{aligned}
		\end{equation}
		For
		\begin{equation}\label{hyper valid}
		|\tau|\leq \delta_{E}^{-1/2},\;\; \beta e^{\mu|\tau|}\sim \widetilde{d}(t)<\delta_{E},
		\end{equation}
		by \eqref{est:dyn lam+},\eqref{equ:dyn elam+} and a continuity argument, we see that,
		\begin{align*}
		&\lambda_{+}(\tau)=\lambda_{+}(0)e^{\mu\tau}+Ce^{\mu\tau}\int_{0}^{\tau}e^{-\mu s}\lambda_{1}^2(s)ds,\\&\lambda_{-}(\tau)=\lambda_{-}(0)e^{-\mu\tau}+Ce^{-\mu\tau}\int_{0}^{\tau}e^{\mu s}\lambda_{1}^2(s)ds.
		\end{align*}
		For $\vec{\lambda}(0)=\beta(\pm1,0)$, we have
		\begin{align*}
		&\lambda_{1}(\tau)=\pm\beta \cosh(\mu\tau)+C\left( e^{\mu\tau}\int_{0}^{\tau}e^{-\mu s}\lambda_{1}^2(s)ds+e^{-\mu\tau}\int_{0}^{\tau}e^{\mu s}\lambda_{1}^2(s)ds \right),\\&\lambda_{2}(\tau)=\pm\beta \sinh(\mu\tau)+C\left( e^{\mu\tau}\int_{0}^{\tau}e^{-\mu s}\lambda_{1}^2(s)ds-e^{-\mu\tau}\int_{0}^{\tau}e^{\mu s}\lambda_{1}^2(s)ds \right),
		\end{align*}
		which leads to 
		\begin{equation*}
		\lambda_{1}(\tau)=\pm\beta \cosh(\mu\tau)\left( 1+O(\beta^\frac12) \right),\;\;\lambda_{2}(\tau)=\pm\beta \sinh(\mu\tau)\left( 1+O(\beta^\frac12) \right).
		\end{equation*}
		Similarly, for $\vec{\lambda}(0)=\beta(0,\pm1)$, we have
		\begin{equation*}
		\lambda_{1}(\tau)=\pm\beta \sinh(\mu\tau)\left( 1+O(\beta^\frac12) \right),\;\;\lambda_{2}(\tau)=\pm\beta \cosh(\mu\tau)\left( 1+O(\beta^\frac12) \right).
		\end{equation*}
		In conclusion, we have
		\begin{equation}\label{hyperbolic}
		\left\{\begin{aligned}
		&\vec{\lambda}(0)=\beta(\pm1,0) \Longrightarrow \vec{\lambda}=\pm \beta\left(\cosh(\mu\tau),\sinh(\mu\tau) \right)\left(1+O(\beta^{1/2}) \right),\\
		&\vec{\lambda}(0)=\beta(0,\pm1) \Longrightarrow \vec{\lambda}=\pm \beta\left(\sinh(\mu\tau),\cosh(\mu\tau) \right)\left(1+O(\beta^{1/2}) \right). \end{aligned}\right.
		\end{equation}
		
		By Proposition \ref{prop:orth decom}, we have $u(t)=e^{i\theta(t)}S_{-1}^{\sigma(t)}\left(W+v(t) \right)$. Using \eqref{est:E expa} and \eqref{est:K expa}, we have
		\begin{align*}
		&E(u)-E(W)=-\mu\lambda_+\lambda_-+\frac{1}{2}\langle \mathcal{L}\gamma,\gamma \rangle-C(v),\\&K(u)=K(W+v)=-2\mu\lambda_{1} \langle W,g_2 \rangle+\langle 2\Delta W,\gamma \rangle+O(\Vert v\Vert_{\dot H^1}^2),
		\end{align*}
		where
		\begin{equation*}
		|2\langle\Delta W, \gamma\rangle|\lesssim\Vert \gamma\Vert_{\dot H^1},\;\;\langle\mathcal{L}\gamma,\gamma  \rangle\sim\Vert \gamma\Vert_{\dot H^1}^2,\;\;C(v)=O(\Vert v\Vert_{\dot H^1}^3)=O(\beta^3).
		\end{equation*}
		
		If $\vec{\lambda}(0)=\pm\beta(1,0)$, then we have
		\begin{equation*}
		E(u)-E(W)\sim -\beta^2<0,\;\;K(u(0))\sim \mp\beta.
		\end{equation*}
		Hence by \cite{MXZ:crit Hart:f rad},
		\begin{equation*}
		\left\{\begin{aligned}
		&\vec{\lambda}(0)=\beta(1,0) \Longrightarrow u(0)\in \mathcal{B}_{-}\cap \mathcal{B}_{+},\\
		&\vec{\lambda}(0)=-\beta(1,0) \Longrightarrow u(0)\in \mathcal{S}_{-}\cap \mathcal{S}_{+} \end{aligned}\right.
		\end{equation*}
		
		If $\vec{\lambda}(0):=\pm\beta(0,1)$, then we have
		\begin{equation*}
		0<E(u)-E(W)\sim \beta^2 \sim \widetilde{d}(0)^2\ll \epsilon^2.
		\end{equation*}
		Near the boundary of the interval \eqref{hyper valid}, we have
		\begin{equation*}
		\widetilde{d}(t)\sim \min\left(\delta_{E},\beta e^{\mu \delta_{E}^{-1/2}} \right)\gg \beta.
		\end{equation*}
		Applying Proposition \ref{prop:eject mod} to $u$ at some $t_{+}>0$ in the forward direction and at some $t_{-}<0$ in the backward direction, where $t_{\pm}$ belongs to the interval \eqref{hyper valid}. Using \eqref{est:H1 gam} and \eqref{hyperbolic}, we have
		\begin{equation*}
		\Vert \gamma(t_{\pm}) \Vert_{\dot H^1} \ll |\lambda_{1}(t_{\pm})| \sim |\lambda_{2}(t_{\pm})|.
		\end{equation*}
		Combined the above estimate with \eqref{est:K expa}, we have
		\begin{equation*}
		sign K(u(t_{\pm}))=-sign \lambda_{1}(t_{\pm})=\Theta(u(t_{\pm})).
		\end{equation*}
		By Proposition \ref{prop:asyp behar} and \eqref{hyperbolic}, we have
		\begin{equation*}
		\left\{\begin{aligned}
		&\vec{\lambda}(0)=\beta(0,1) \Longrightarrow \Theta(u(t_{\pm}))=\mp 1, \Longrightarrow u(0)\in \mathcal{S}_{-}\cap \mathcal{B}_{+},\\
		&\vec{\lambda}(0)=-\beta(0,1) \Longrightarrow \Theta(u(t_{\pm}))=\pm 1, \Longrightarrow u(0)\in \mathcal{S}_{+}\cap \mathcal{B}_{-}. \end{aligned}\right.
		\end{equation*}
		
		It is obvious that the above argument is stable for adding small perturbation in $\R^2$ to $\vec{\lambda}(0)$ and small perturbation in $H^1$ (in the orthogonal subspace) to $\gamma(0)$. Hence we obtain a small open set in $H^1$ around each of the four solutions.
	\end{proof}


	\def\cprime{$'$}

\end{document}